\documentclass[a4paper,12
pt]{article}

\usepackage{amsfonts,amsmath,amssymb,amsthm,nicefrac,enumerate,bbm,hyperref,vmargin,mathrsfs,comment,color}
\usepackage[latin1]{inputenc}

\newcommand{\R}{\mathbb{R}}
\newcommand{\N}{\mathbb{N}}
\newcommand{\E}{\mathbb{E}}
\newcommand{\id}{\operatorname{id}}
\renewcommand{\P}{\mathbb{P}}

\newcommand\numberthis{\addtocounter{equation}{1}\tag{\theequation}}

\newtheorem{theorem}{Theorem}[section]
\newtheorem{lemma}[theorem]{Lemma}

\newtheorem{cor}[theorem]{Corollary}

\newtheorem{prop}[theorem]{Proposition}

\begin{document}
\title{A proof that deep artificial neural networks overcome the curse
	of dimensionality in the numerical approximation of Kolmogorov\\
	partial differential equations
	with constant diffusion and nonlinear drift coefficients}

\author{Arnulf Jentzen$^1$,  Diyora Salimova$^2$, and Timo Welti$^3$
	\bigskip
	\\
	\small{$^1$Seminar for Applied Mathematics, Department of Mathematics,}\\
		\small{ETH Zurich, Switzerland, e-mail:   arnulf.jentzen@sam.math.ethz.ch}
	\smallskip
	\\
\small{$^2$Seminar for Applied Mathematics, Department of Mathematics,}\\
\small{ETH Zurich, Switzerland, e-mail:  diyora.salimova@sam.math.ethz.ch}
	\smallskip
	\\
	\small{$^3$Seminar for Applied Mathematics, Department of Mathematics,}\\
		\small{ETH Zurich, Switzerland, e-mail:   timo.welti@sam.math.ethz.ch}}

\maketitle

\begin{abstract}
In recent years deep artificial neural networks (DNNs)  have been successfully
employed in numerical simulations for a multitude of computational problems 
including, for example, object and face recognition, natural language processing, 
fraud detection, computational advertisement, 
and numerical approximations of partial differential equations (PDEs).
 These numerical simulations indicate that 
DNNs seem to  possess the fundamental flexibility to overcome 
the curse of dimensionality in the sense that the 
number of real parameters used to describe the 
DNN grows at most polynomially in both the 
reciprocal of the prescribed approximation 
accuracy $ \varepsilon > 0 $ and the dimension $ d \in \N $
of the function which the DNN aims to approximate in such computational problems.
There is also a large number of rigorous mathematical approximation results 
for artificial neural networks in the scientific literature 
but there are only a few special situations where results in the literature 
can rigorously  justify the success of DNNs in high-dimensional function approximation. 
The key contribution of this  paper is to reveal that DNNs do overcome the curse of dimensionality 
in the numerical approximation of Kolmogorov PDEs with constant diffusion and nonlinear drift coefficients. 
We prove that the number of parameters used to describe the employed DNN grows at most 
polynomially in both
 the PDE dimension $ d \in \N $ and the reciprocal of the prescribed approximation accuracy $ \varepsilon > 0 $. 
A crucial ingredient in our proof is the fact that the artificial neural network 
used to approximate the solution of the PDE is indeed 
a deep artificial neural network with a large number of hidden layers. 
\end{abstract}

\tableofcontents

\section{Introduction}
\label{sec:intro}
In recent years deep artificial neural networks (DNNs) 
 have been successfully
employed in numerical simulations for a multitude of computational problems 
including, for example, object and face recognition  (cf., e.g., \cite{huang2017densely,krizhevsky2012imagenet,simonyan2014very,Taigman2014,wangfacerecognition2015}
  and the references mentioned therein), 
natural language processing (cf., e.g., \cite{dahl2012context,graves2013speech,hinton2012deep,HuConvolutional2014,Kalchbrenner14aconvolutional,wu2016stimulated}
and the references mentioned therein), 
fraud detection
(cf., e.g., \cite{CHOUIEKH2018,Roy2018} and the references mentioned therein), computational advertisement
(cf., e.g., \cite{Wang2017Ad,Zhai2016} and the references mentioned therein), 
and numerical approximations of partial differential equations (PDEs) 
(cf., e.g., \cite{Kolmogorov,beck2017machine,becker2018deep, weinan2017deep, weinan2018deep,ElbraechterSchwab2018, fujii2017asymptotic, GrohsWurstemberger2018, Han2018PNAS, henry2017deep, khoo2017solving, mishra2018machine, nabian2018deep, raissi2018forward,sirignano2017dgm}). 
 These numerical simulations indicate that 
DNNs seem to possess the fundamental flexibility to overcome 
the curse of dimensionality in the sense that the 
number of real parameters used to describe the 
DNN grows at most polynomially in both the 
reciprocal of the prescribed approximation 
accuracy $ \varepsilon > 0 $ and the dimension $ d \in \N $
of the function which the DNN aims to approximate in such computational problems.
There is also a large number of rigorous mathematical approximation results for artificial neural networks 
in the scientific literature (see, for instance,  \cite{bach2017breaking,Barron1993, Barron1994,blum1991approximation,bolcskei2017optimal,Burger2001235,CandesDiss,chen1995approximation,ChuXM1994networksforlocApprox,Cybenko1989,DeVore1997approxfeedforward,EWang2018,ElbraechterSchwab2018,Eldan2016PowerofDepth,ellacott1994aspects,Funahashi1989183,GrohsWurstemberger2018,hartman1990layered,Hornik1991251,hornik1993some,Hornik1989universalApprox,hornik1990universal,leshno1993multilayer,Mhaskar1995151,Mhaskar1996NNapprox,Mhaskar2016DeepVSShallow, MontanelliDu2017, NguyenThien1999687,park1991universal,perekrestenko2018universal,petersen2018topological,petersen2017optimal,pinkus1999approximation,Schmitt1999lowercomplbounds,SchwabZech2018,ShaCC2015provableAppDNN,sirignano2017dgm,yarotsky2017error,yarotsky2018universal} and the references mentioned therein)
but there are only a few special situations where results in the literature 
can rigorously justify the success of DNNs in high-dimensional function approximation.

The key contribution of this  paper is to reveal that DNNs do overcome the curse of dimensionality 
in the numerical approximation of Kolmogorov PDEs with constant diffusion and nonlinear drift coefficients. 
More specifically, the main result of this article, Theorem~\ref{thm:PDE_approx_Lp} 
in Subsection~\ref{sec:main_result} below, 
proves that the number of parameters used to describe the employed DNN grows 
at most polynomially in both  the PDE dimension $ d \in \N $ and the reciprocal of the prescribed approximation accuracy $ \varepsilon > 0 $ and, thereby, 
we establish that DNN approximations do indeed overcome the curse of dimensionality in the 
numerical approximation of such PDEs.
To illustrate the statement of Theorem~\ref{thm:PDE_approx_Lp} below in more details, 
we now present the following special case of Theorem~\ref{thm:PDE_approx_Lp}.

\begin{theorem}
\label{thm:intro}
Let 
$
A_d = ( a_{ d, i, j } )_{ (i, j) \in \{ 1, \dots, d \}^2 } \in \R^{ d \times d }
$,
$ d \in \N $,
be symmetric positive semidefinite matrices, 
for every $ d \in \N $ 
let 
$
\left\| \cdot \right\|_{ \R^d } \colon \R^d \to [0,\infty)
$
be the $ d $-dimensional Euclidean norm,
let
$ f_{0,d} \colon \R^d \to \R $, $ d \in \N $,
and
$ f_{ 1, d } \colon \R^d \to \R^d $,
$ d \in \N $,
be functions,
let 
$ \mathbf{A}_d \colon \R^d \to \R^d $, 
$ d \in \N $, 
be the functions 
which satisfy 
for all 
$
d \in \N
$,
$ x = ( x_1, \dots, x_d ) \in \R^d $
that
$ 
\mathbf{A}_d(x)
=
( \max\{x_1, 0\}, \ldots, \max\{x_d, 0\} )
$,
let 
\begin{equation}
\label{eq:set_N_intro}
\mathcal{N}
=
\cup_{ L \in \{ 2, 3, 4, \dots \} }
\cup_{ ( l_0, l_1, \ldots, l_L ) \in \N^{ L + 1 } }
(
\times_{ n = 1 }^L 
(
\R^{ l_n \times l_{ n - 1 } } \times \R^{ l_n } 
)
)
,
\end{equation}
let 
$
\mathcal{P}
\colon \mathcal{N} \to \N
$
and
$
\mathcal{R} \colon 
\mathcal{N} 
\to 
\cup_{ k, l \in \N } C( \R^k, \R^l )
$
be the functions which satisfy 
for all 
$ L \in \{ 2, 3, 4, \dots \} $, 
$ l_0, l_1, \ldots, l_L \in \N $, 
$ 
\Phi = ((W_1, B_1), \ldots, $ $ (W_L, B_L)) \in 
( \times_{ n = 1 }^L (\R^{ l_n \times l_{n-1} } \times \R^{ l_n } ) )
$,
$ x_0 \in \R^{l_0} $, 
$ \ldots $, 
$ x_{ L - 1 } \in \R^{ l_{ L - 1 } } $ 
with 
$ 
\forall \, n \in \N \cap [1,L) \colon 
x_n = \mathbf{A}_{ l_n }( W_n x_{ n - 1 } + B_n )
$
that 
$
\mathcal{P}( \Phi )
=
\textstyle
\sum\nolimits_{
	n = 1
}^L
l_n ( l_{ n - 1 } + 1 )
$,
$
\mathcal{R}(\Phi) \in C( \R^{ l_0 } , \R^{ l_L } )
$, 
and
\begin{equation}
\label{eq:def_R_intro}
( \mathcal{R} \Phi )( x_0 ) = W_L x_{L-1} + B_L ,
\end{equation}
let 
$ T, \kappa \in (0,\infty) $,  
$
( \phi^{ m, d }_{ \varepsilon } )_{ 
	(m, d, \varepsilon) \in \{ 0, 1 \} \times \N \times (0,1] 
} 
\subseteq \mathcal{N}
$,
assume for all
$ d \in \N $, 
$ \varepsilon \in (0,1] $, 
$ 
x, y \in \R^d
$
that
$
\mathcal{R}( \phi^{ 0, d }_{ \varepsilon } )
\in 
C( \R^d, \R )
$,
$
\mathcal{R}( \phi^{ 1, d }_{ \varepsilon } )
\in
C( \R^d, \R^d )
$,
$
|
f_{ 0, d }( x )
| 
+
\sum_{ i , j = 1 }^d
| a_{ d, i, j } |
\leq 
\kappa d^{ \kappa }
( 1 + \| x \|^{ \kappa }_{ \R^d } )
$,
$
\| 
f_{ 1, d }( x ) 
- 
f_{ 1, d }( y )
\|_{ \R^d }
\leq 
\kappa 
\| x - y \|_{ \R^d } 
$,
$
\|
( \mathcal{R} \phi^{ 1, d }_{ \varepsilon } )(x)    
\|_{ \R^d }	
\leq 
\kappa ( d^{ \kappa } + \| x \|_{ \R^d } )
$,
$ 
\sum_{ m = 0 }^1
\mathcal{P}( \phi^{ m, d }_{ \varepsilon } ) 
\leq \kappa d^{ \kappa } \varepsilon^{ - \kappa }
$, $ |( \mathcal{R} \phi^{ 0, d }_{ \varepsilon } )(x) - ( \mathcal{R} \phi^{ 0, d }_{ \varepsilon } )(y)| \leq \kappa d^{\kappa} (1 + \|x\|_{\R^d}^{\kappa} + \|y \|_{\R^d}^{\kappa})\|x-y\|_{\R^d}$, 
and
\begin{equation}
\label{eq:intro_hypo}
| 
f_{ 0, d }(x) 
- 
( \mathcal{R} \phi^{ 0, d }_{ \varepsilon } )(x)
|
+
\| 
f_{ 1, d }(x) 
- 
( \mathcal{R} \phi^{ 1, d }_{ \varepsilon } )(x)
\|_{ \R^d }
\leq 
\varepsilon \kappa d^{ \kappa }
(
1 + \| x \|^{ \kappa }_{ \R^d }
)
,
\end{equation}
and for every $ d \in \N $ let
	$ u_d \colon [0,T] \times \R^{d} \to \R $
	be an 
	at most polynomially growing viscosity solution of
	\begin{equation}
	\label{eq:PDE_intro}
	\begin{split}
	( \tfrac{ \partial }{\partial t} u_d )( t, x ) 
	& = 
	( \tfrac{ \partial }{\partial x} u_d )( t, x )
	\,
	f_{ 1, d }( x )
	+
	\textstyle
	\sum\limits_{ i, j = 1 }^d
	\displaystyle
	a_{ d, i, j }
	\,
	( \tfrac{ \partial^2 }{ \partial x_i \partial x_j } u_d )( t, x )
	\end{split}
	\end{equation}
	with $ u_d( 0, x ) = f_{ 0, d }( x ) $
	for $ ( t, x ) \in (0,T) \times \R^d $. 
	Then for every $ p \in (0,\infty) $ 
	there exist
	$
	( 
	\psi_{ d, \varepsilon } 
	)_{ (d , \varepsilon)  \in \N \times (0,1] } \subseteq \mathcal{N}
	$,
	$
	c \in \R
	$
	such that
	for all 
	$
	d \in \N 
	$,
	$
	\varepsilon \in (0,1] 
	$
	it holds that
	$
	\mathcal{P}( \psi_{ d, \varepsilon } ) 
	\leq
	c \, d^c \varepsilon^{ - c } 
	$,
	$
	\mathcal{R}( \psi_{ d, \varepsilon } )
	\in C( \R^{ d }, \R )
	$,
	and
	\begin{equation}
	\label{eq:intro_statement}
	\left[
	\int_{ [0,1]^d }
	|
	u_d(T,x) - ( \mathcal{R} \psi_{ d, \varepsilon } )( x )
	|^p
	\,
	dx
	\right]^{ \nicefrac{ 1 }{ p } }
	\leq
	\varepsilon 
	.
	\end{equation}
\end{theorem}

Theorem~\ref{thm:intro} 
is an immediate consequence of Corollary~\ref{cor:lebesgue} 
in Subsection~\ref{sec:lebesgue} below. 
Corollary~\ref{cor:lebesgue}, in turn, 
is a special case of Theorem~\ref{thm:PDE_approx_Lp}. 
Next we add some comments regarding the mathematical objects appearing 
in Theorem~\ref{thm:intro}. 
Theorem~\ref{thm:intro} 
is an approximation result for rectified DNNs and 
for every $ d \in \N $ the function
$ {\bf A}_d \colon \R^d \to \R^d $ 
in Theorem~\ref{thm:intro} above
describes the $ d $-dimensional 
rectifier function. 
The set $ \mathcal{N} $ in \eqref{eq:set_N_intro} in Theorem~\ref{thm:intro} above 
is a set of tuples of real numbers which, in turn, 
represents the set of all artificial neural networks. 
For every artificial neural network $ \Phi \in \mathcal{N} $ 
in Theorem~\ref{thm:intro} above
we have that $ \mathcal{R}( \Phi ) \in \cup_{ k, l \in \N } C( \R^k, \R^l ) $ 
represents the function 
associated to the artificial neural network $ \Phi $ 
(cf.\ \eqref{eq:def_R_intro} in Theorem~\ref{thm:intro}). 
The function 
$ \mathcal{R} \colon \mathcal{N} \to \cup_{ k, l \in \N } C( \R^k, \R^l ) $ 
from the set $ \mathcal{N} $ of all artificial neural networks to the union 
$ \cup_{ k, l \in \N } C( \R^k, \R^l ) $ 
of continuous functions thus describes the realizations 
associated to the artificial neural networks. 
Moreover, 
for every artificial neural network $ \Phi \in \mathcal{N} $
in Theorem~\ref{thm:intro} above we have that 
$ \mathcal{P}( \Phi ) \in \N $ represents the number of real parameters 
which are used to describe the artificial neural network $ \Phi $. 
In particular, for every artificial neural network $ \Phi \in \mathcal{N} $
in Theorem~\ref{thm:intro} we can think of 
$ \mathcal{P}( \Phi ) \in \N $ 
as a quantity related to the amount of memory storage 
which is needed to store the artificial neural network.
The real number $ \kappa > 0 $ in Theorem~\ref{thm:intro} 
is an arbitrary constant used to formulate the hypotheses in Theorem~\ref{thm:intro} 
(cf.\ \eqref{eq:intro_hypo} in Theorem~\ref{thm:intro} above) 
and the real number $ T > 0 $ 
in Theorem~\ref{thm:intro} describes the time horizon under consideration. 
Our key hypothesis in Theorem~\ref{thm:intro} is the assumption 
that both the possibly nonlinear initial value functions $ f_{ 0, d } \colon \R^d \to \R $,
$ d \in \N $, 
and the possibly nonlinear drift coefficient functions 
$ f_{ 1, d } \colon \R^d \to \R^d $, $ d \in \N $,
of the PDEs in \eqref{eq:PDE_intro} can be approximated 
without the curse of dimensionality by means of DNNs 
(see \eqref{eq:intro_hypo} above for details). 
Simple examples for the functions 
$ f_{ 0, d } \colon \R^d \to \R $, $ d \in \N $,
and 
$ f_{ 1, d } \colon \R^d \to \R^d $, $ d \in \N $, 
which fulfill the hypotheses of Theorem~\ref{thm:intro} above
are, for instance, provided by the choice that for all $ d \in \N $,
$ x = ( x_1, x_2, \dots, x_d ) \in \R^d $ it holds that
$
  f_{ 0, d }( x ) = \max\{ x_1, x_2, \dots, x_d \}
$
and 
$
  f_{ 1, d }( x ) = x \, [ 1 + \| x \|^2_{ \R^d } ]^{ - 1 }
$.
A natural example for the matrices 
$ A_d = ( a_{ d, i, j } )_{ (i,j) \in \{ 1, \dots, d \}^2 } \in \R^{ d \times d } $,
$ d \in \N $,
fulfilling the hypotheses in Theorem~\ref{thm:intro} above 
is, for instance, provided by the choice that for all 
$ d \in \N $
it holds that
$ A_d \in \R^{ d \times d } $ is the $ d $-dimensional identity matrix 
in which case the second order term in \eqref{eq:PDE_intro} 
reduces to the $ d $-dimensional Laplace operator. 
Roughly speaking, Theorem~\ref{thm:intro} above proves that 
if both the initial value functions and the drift coefficient functions 
in the PDEs in \eqref{eq:PDE_intro} can be approximated without the curse of dimensionality 
by means of DNNs, then the solutions of the PDEs can also be approximated 
without the curse of dimensionality by means of DNNs 
(see \eqref{eq:intro_statement} above for details). 
In numerical simulations involving DNNs for computational problems from data science (e.g., object and face recognition, natural language processesing, fraud detection, computational advertisement, etc.) it is often not entirely clear how to precisely describe what the involved DNN approximations should achieve and it is thereby often not entirely clear how to precisely specify the approximation error of the employed DNN. The recent articles \cite{weinan2017deep, Han2018PNAS} (cf., e.g., also \cite{Kolmogorov,beck2017machine,becker2018deep, weinan2018deep,ElbraechterSchwab2018, fujii2017asymptotic, GrohsWurstemberger2018, henry2017deep, khoo2017solving, mishra2018machine, nabian2018deep, raissi2018forward,sirignano2017dgm}) suggest to use machine learning algorithms which employ DNNs to approximate solutions and derivatives of solutions, respectively, of PDEs and in the framework of these references it is perfectly clear what the involved DNN approximations should achieve as well as how to specify the approximation error: the DNN should approximate the unique deterministic function which is the solution of the given deterministic PDE (cf., e.g., Han et al. \cite[\emph{Neural Network Architecture} on page 5]{Han2018PNAS} and Beck et al.~\cite[Proposition 2.7 and (103)]{Kolmogorov}. The above named references thereby open up the possibility for a complete and rigorous mathematical error analysis for the involved deep learning algorithms and Theorems~\ref{thm:intro} and~\ref{thm:PDE_approx_Lp}, in particular, provide some first contributions to this new research topic. 
The statements of Theorems~\ref{thm:intro} and~\ref{thm:PDE_approx_Lp} and their strategies of proof, respectively, are inspired by the article Grohs et al.~\cite{GrohsWurstemberger2018} (cf., e.g., Theorem 1.1 in~\cite{GrohsWurstemberger2018}) in which similar results as Theorems~\ref{thm:intro} and~\ref{thm:PDE_approx_Lp}, respectively, but for Kolmogorov PDEs with affine linear drift and diffusion coefficient functions have been proved. The main difference of the arguments in \cite{GrohsWurstemberger2018} to this paper is the deepness of the involved artificial neural networks. Roughly speaking, the affine linear structure of the coefficients of the Kolmogorov PDEs in \cite{GrohsWurstemberger2018} allowed the authors in \cite{GrohsWurstemberger2018} to essentially employ a flat artificial neural network for approximating the solution flow mapping of such PDEs. In this work the drift coefficient is nonlinear and, in view of this property, we employ in our proofs of Theorem~\ref{thm:intro} and Theorem~\ref{thm:PDE_approx_Lp}, respectively, iterative Euler-type discretizations for the underlying stochastic dynamics associated to the PDEs in~\eqref{eq:PDE_intro}. The iterative Euler-type discretizations result in multiple compositions which, in turn, result in deep artificial neural networks with a large number of hidden layers. In particular, in our proof of Theorem~\ref{thm:intro} and Theorem~\ref{thm:PDE_approx_Lp}, respectively, the artificial neural networks $ \psi_{ d, \varepsilon } \in \mathcal{N} $,
$ d \in \N $, $ \varepsilon \in (0,1] $,  approximating the solutions of the PDEs in~\eqref{eq:PDE_intro} (see~\eqref{eq:intro_statement} above) are also deep artificial neural networks with a large number of hidden layers  even 
if the artificial neural networks approximating or representing 
$ f_{ 0, d } \colon \R^d \to \R $, $ d \in \N $,
and 
$ f_{ 1, d } \colon \R^d \to \R^d $, $ d \in \N $,
are flat with one hidden layer only. Moreover,  our proofs of Theorem~\ref{thm:intro}
and Theorem~\ref{thm:PDE_approx_Lp}, respectively, 
reveal that the number of hidden layers increases to infinity 
as the prescribed approximation accuracy $ \varepsilon > 0 $ 
decreases to zero and the PDE dimension increases to infinity, respectively (cf.~\eqref{eq:control:delta} and \eqref{eq:bar:omega} below).

Theorem~\ref{thm:intro} above and 
Theorem~\ref{thm:PDE_approx_Lp}, respectively, 
are purely deterministic approximation results for DNNs and 
solutions of a class of deterministic PDEs. 
Our proofs of Theorem~\ref{thm:intro} and 
Theorem~\ref{thm:PDE_approx_Lp}, respectively, 
are, however, heavily relying on probabilistic arguments 
on a suitable artificial probability space. 
Roughly speaking, in our proof of Theorem~\ref{thm:PDE_approx_Lp} 
we
\begin{enumerate}[(I)]
\item
\label{item:I}
design a suitable random DNN on this artificial probability space,
\item 
\label{item:II}
show that this suitable random DNN is in a suitable sense 
close to the solution of the considered deterministic PDE, 
and
\item
\label{item:III}
employ items~\eqref{item:I}--\eqref{item:II} above to establish the existence of a realization 
with the desired approximation properties on the artificial probability space.
\end{enumerate}
The specific realization of this random DNN is then a deterministic DNN approximation 
of the solution of the considered deterministic PDE with the desired approximation properties.
The main work of the paper is the construction and the analysis of this random DNN. 
For the construction of the random DNN we need suitable general flexibility results 
for rectified DNNs which, roughly speaking, 
demonstrate how rectified DNNs can be composed with a 
moderate growth of the number of involved parameters (see Subsection~\ref{sec:composition:ANN} below for details). 
The construction of the random DNN 
(cf.\ \eqref{item:I} above) 
is essentially performed in Section~\ref{sec:dnn:calculus}
and Section~\ref{sec:DNN_PDEs}
and the analysis of the random DNN 
(cf.\ \eqref{item:II} above)
is essentially the subject of Section~\ref{sec:feynman:kac}, Section~\ref{sec:sdes}, 
and 
Subsection~\ref{sec:PDE_approx_Lp}.
The argument for the existence of the realization with  suitable approximation properties 
on the artificial probability space (cf.\ \eqref{item:III} above) is provided in Section~\ref{sec:existence} 
and Subsection~\ref{sec:PDE_approx_Lp}.

\section[Existence of a realization on a suitable artificial probability space]{On the existence of a realization with the desired approximation properties on a suitable artificial probability space}
\label{sec:existence}

In this section we establish in Corollary~\ref{cor:random_field} in Subsection~\ref{sec:existence_realization} below on a very abstract level, roughly speaking, the argument that good approximation properties of the random DNN (cf.\ items~\eqref{item:I}--\eqref{item:II} in Section~\ref{sec:intro} above) imply the existence of a realization with  suitable approximation properties on the artificial probability space (cf.\ item~\eqref{item:III} in Section~\ref{sec:intro} above). The function $ u \colon \R^d \to \R $ in Corollary~\ref{cor:random_field} will essentially take the role of the solution of the considered deterministic PDE and the random field $ X \colon \R^d \times \Omega \to \R $ will essentially take the role of the random DNN.
Our proof of Corollary~\ref{cor:random_field} is based on an application of Proposition~\ref{prop:markov3} in Subsection~\ref{sec:existence_realization} below.
Proposition~\ref{prop:markov3} is, very loosely speaking, an abstract generalized version of Corollary~\ref{cor:random_field}.
Our proof of Proposition~\ref{prop:markov3} is based on an application of the elementary Markov-type estimate in Lemma~\ref{lem:Markov2} in Subsection~\ref{sec:markov} below.
Lemma~\ref{lem:Markov2}, in turn, follows from the Markov inequality in Lemma~\ref{lem:markov} in Subsection~\ref{sec:markov} below.
For completeness we also provide the short proof of the Markov inequality in Lemma~\ref{lem:markov}.
Results related to Lemma~\ref{lem:Markov2} and Proposition~\ref{prop:markov3} can, e.g., be found in Grohs et al.~\cite[Subsection~3.1]{GrohsWurstemberger2018}.
In particular, Lemma~\ref{lem:Markov2} is somehow an elementary extension of~\cite[Proposition~3.3 in Subsection~3.1]{GrohsWurstemberger2018}.

\subsection{Markov-type estimates}
\label{sec:markov}

\begin{lemma}[Markov inequality]
	\label{lem:markov}
	Let 
	$ \left( \Omega, \mathcal{F}, \mu \right) $
	be a measure space,
	let $ \varepsilon \in (0,\infty) $, 
	and let 
	$ X \colon \Omega \to [0,\infty] $ 
	be an $ \mathcal{F} $/$ \mathcal{B}( [0,\infty] ) $-measurable
	function. Then
	\begin{equation}
	\label{eq:markov}
	\mu\big(
	X \geq \varepsilon
	\big)
	\leq
	\frac{
		\int_{ \Omega } X \, d\mu
	}{
		\varepsilon
	}
	.
	\end{equation}
\end{lemma}

\begin{proof}[Proof
	of Lemma~\ref{lem:markov}]
	Note that the fact that $ X \geq 0 $ proves that
	\begin{equation}
	\label{eq:markov_use}
	\mathbbm{1}_{
		\{ X \geq \varepsilon \}
	}
	=
	\frac{
		\varepsilon 
		\cdot
		\mathbbm{1}_{
			\{ X \geq \varepsilon \}
		}
	}{
		\varepsilon
	}
	\leq
	\frac{
		X 
		\cdot
		\mathbbm{1}_{
			\{ X \geq \varepsilon \}
		}
	}{
		\varepsilon
	}
	\leq
	\frac{
		X
	}{
		\varepsilon
	}.
	\end{equation}
	Integration with respect to $ \mu $ hence
	establishes \eqref{eq:markov}.
	The proof of Lemma~\ref{lem:markov} is thus completed.
\end{proof}

\begin{lemma}
	\label{lem:Markov2}
	Let $ ( \Omega, \mathcal{F}, \P ) $ be a probability space, 
	let $ X \colon \Omega \to [-\infty,\infty] $ be a random variable, 
	and let $ \varepsilon, q \in (0,\infty) $.
	Then
	\begin{equation}
	\left[ 
	\P\big(
	|X| \geq \varepsilon
	\big)
	\right]^{ \nicefrac{ 1 }{ q } }
	\leq
	\frac{ 
		\big(
		\E\big[
		| X |^q
		\big]
		\big)^{ 1 / q }
	}{
		\varepsilon
	}
	.
	\end{equation}
\end{lemma}

\begin{proof}[Proof of Lemma~\ref{lem:Markov2}]
	Observe that Lemma~\ref{lem:markov} ensures that
	\begin{equation}
	\begin{split}
	\left[ 
	\P\!\left(
	|X| \geq \varepsilon
	\right)
	\right]^{ \nicefrac{ 1 }{ q } }
	& =
	\left[ 
	\P\!\left(
	|X|^q \geq \varepsilon^q
	\right)
	\right]^{ \nicefrac{ 1 }{ q } }
	\leq
	\left[ 
	\frac{ 
		\E\big[
		|X|^q
		\big]
	}{
		\varepsilon^q
	}
	\right]^{ \nicefrac{ 1 }{ q } }
	=
	\frac{ 
		\big(
		\E\big[
		| X |^q
		\big]
		\big)^{ 1 / q }
	}{
		\varepsilon
	}
	.
	\end{split}
	\end{equation}
	The proof of Lemma~\ref{lem:Markov2} is thus completed.
\end{proof}

\subsection{Existence of a realization with the desired approximation properties}
\label{sec:existence_realization}

\begin{prop}
	\label{prop:markov3}
	Let $ \varepsilon \in (0,\infty) $,
	let $ ( \Omega, \mathcal{F}, \P ) $ be a probability space, 
	and
	let $ X \colon \Omega \to [-\infty,\infty] $ be a random variable which satisfies that 
	\begin{equation}
	\inf\nolimits_{ q \in (0,\infty) }
	\big(
	\E\big[
	| X |^q
	\big]
	\big)^{ 1 / q }
	< \varepsilon
	.
	\end{equation}
	Then there exists $ \omega \in \Omega $ such that
	$
	| X( \omega ) | < \varepsilon
	$.
\end{prop}

\begin{proof}[Proof of Proposition~\ref{prop:markov3}] 
	First, observe that 
	Lemma~\ref{lem:Markov2} assures that for all $ q \in (0,\infty) $ it holds that
	\begin{equation}
	\label{eq:consequence_Markov}
	\left[ 
	\P\big(
	|X| \geq \varepsilon
	\big)
	\right]^{ \nicefrac{ 1 }{ q } }
	\leq
	\frac{ 
		\big(
		\E\big[
		| X |^q
		\big]
		\big)^{ 1 / q }
	}{
		\varepsilon
	}
	.
	\end{equation}
	Next note that the hypothesis that 
	$
	\inf\nolimits_{ q \in (0,\infty) }
	\big(
	\E\big[
	| X |^q
	\big]
	\big)^{ 1 / q }
	< \varepsilon
	$
	demonstrates that there exists $ q \in (0,\infty) $ 
	such that
	\begin{equation}
	\big(
	\E\big[
	| X |^q
	\big]
	\big)^{ 1 / q }
	< \varepsilon
	.
	\end{equation}
	Combining this with \eqref{eq:consequence_Markov} 
	proves that
	\begin{equation}
	\left[ 
	\P\big(
	|X| \geq \varepsilon
	\big)
	\right]^{ \nicefrac{ 1 }{ q } }
	< 1
	.
	\end{equation}
	Hence, we obtain that
	\begin{equation}
	\P\big(
	|X| \geq \varepsilon
	\big)
	< 1
	.
	\end{equation}
	This shows that
	\begin{equation}
	\P\big(
	|X| < \varepsilon
	\big)
	=
	1 - 
	\P\big(
	|X| \geq \varepsilon
	\big)
	> 0
	.
	\end{equation}
	Therefore, we obtain that 
	\begin{equation}
	\{ 
	|X| < \varepsilon
	\}
	=
	\big\{ 
	\omega \in \Omega 
	\colon
	| X( \omega ) | < \varepsilon
	\big\} 
	\neq \emptyset
	.
	\end{equation}
	The proof of Proposition~\ref{prop:markov3} is thus completed.
\end{proof}


\begin{cor}[Existence of approximating realizations of a random field]
		\label{cor:random_field}
		Let $ d \in \N $, $ p, \varepsilon \in (0,\infty) $, 
		let $ u \colon \R^d \to \R $ be 
		$ \mathcal{B}( \R^d ) $/$ \mathcal{B}( \R ) $-measurable, 
		let $ ( \Omega, \mathcal{F}, \P ) $ be a probability space, let $\nu \colon \mathcal{B}(\R^d) \to [0,1]$ be a probability measure on $\R^d$, 
		let $ X \colon \R^d \times \Omega \to \R $ be 
		$ ( \mathcal{B}( \R^d ) \otimes \mathcal{F} ) $/$ \mathcal{B}( \R ) $-measurable, 
		and assume that
		\begin{equation}
		\label{eq:assumption_random_field}
		\left[ 
		\int_{ \R^d } \E\big[ | u(x) - X(x) |^p \big] \,  \nu(dx) 
		\right]^{ \nicefrac{ 1 }{ p } } < \varepsilon
		.
		\end{equation}
		Then there exists $ \omega \in \Omega $ such that
		\begin{equation}
		\label{eq:random_field_to_show}
		\left[ 
		\int_{ \R^d } 
		\left| u(x) - X(x,\omega) \right|^p 
		\nu(dx) \right]^{ 
			\nicefrac{ 1 }{ p }
		} 
		< \varepsilon .
		\end{equation}
	\end{cor}

\begin{proof}[Proof of Corollary~\ref{cor:random_field}]
	Throughout this proof let $ Y \colon \Omega \to [-\infty,\infty] $ be the random variable 
	given by
	\begin{equation}
	\label{eq:random_field_defY}
	Y = 
	\left[ 
	\int_{ \R^d } 
	\left| u(x) - X(x) \right|^p 
	\nu(dx) 
	\right]^{ 
		\nicefrac{ 1 }{ p }
	} 
	.
	\end{equation}
	Observe that Fubini's theorem and \eqref{eq:assumption_random_field} ensure that
	\begin{equation}
	\begin{split}
	\big( \E\big[ | Y |^p \big] \big)^{ \nicefrac{ 1 }{ p } }
	& =
	\left(
	\E\!\left[ 
	\int_{ \R^d } 
	\left| u(x) - X(x) \right|^p 
	\nu(dx) 
	\right]
	\right)^{ \! \nicefrac{ 1 }{ p } }\\
& =
	\left(
	\int_{ \R^d } 
	\E\big[ 
	| u(x) - X(x) |^p 
	\big]
	\,
	\nu(dx) 
	\right)^{ \! \nicefrac{ 1 }{ p } }
	< \varepsilon .
	\end{split}
	\end{equation}
	Hence, we obtain that 
	\begin{equation}
	\inf_{ q \in (0,\infty) } 
	\big( \E\big[ | Y |^q \big] \big)^{ \nicefrac{ 1 }{ q } }
	\leq 
	\big( \E\big[ | Y |^p \big] \big)^{ \nicefrac{ 1 }{ p } }
	< 
	\varepsilon .
	\end{equation}
	This allows us to apply Proposition~\ref{prop:markov3} to obtain that 
	there exists $ \omega \in \Omega $ such that
	\begin{equation}
	| Y(\omega) | < \varepsilon .
	\end{equation}
	Combining this with \eqref{eq:random_field_defY} establishes 
	\eqref{eq:random_field_to_show}. 
	The proof of Corollary~\ref{cor:random_field} is thus completed.
\end{proof}

\section{The Feynman-Kac formula revisited}
\label{sec:feynman:kac}

Theorem~\ref{thm:PDE_approx_Lp} in Subsection~\ref{sec:main_result} below
(the main result of this article)
and Theorem~\ref{thm:intro} in the introduction, respectively, are, as mentioned above, purely deterministic approximation results for DNNs and a class of deterministic PDEs.
In contrast, our proofs of Theorem~\ref{thm:PDE_approx_Lp} and Theorem~\ref{thm:intro}, respectively,
are based on a probabilistic argument on a suitable artificial probability space on which we, roughly speaking, design random DNNs.
Our construction of the random DNNs is based on suitable Monte Carlo approximations of the solutions of the considered deterministic PDEs.
These suitable Monte Carlo approximations, in turn, are based on the link between deterministic Kolmogorov PDEs
and solutions of SDEs which is provided by the famous Feynman-Kac formula.
In this section we recall in Theorem~\ref{thm:feynman} below a special case of this famous formula (cf., e.g., Hairer et al.~\cite[Subsection~4.4]{HairerHutzenthalerJentzen2015}).
Theorem~\ref{thm:feynman} below will be used in our proof of Theorem~\ref{thm:PDE_approx_Lp} below
(cf.\ \eqref{eq:X_processes} and \eqref{eq:apply_feynman}
in the proof of Proposition~\ref{prop:PDE_approx_Lp},
Proposition~\ref{prop:PDE_approx_Lp}, Corollary~\ref{cor:PDE_approx_Lp_gen},
and
Theorem~\ref{thm:PDE_approx_Lp}).

\begin{theorem}
\label{thm:feynman}
Let $ ( \Omega, \mathcal{F}, \P ) $ be a probability space, 
let $ T \in (0,\infty) $, $ d, m \in \N $, $ B \in \R^{ d \times m } $, 
let $ W \colon [0,T] \times \Omega \to \R^m $ be a standard Brownian motion, 
let $ \left\| \cdot \right\| \colon \R^d \to [0,\infty) $
be the $ d $-dimensional Euclidean norm, 
let $ \left< \cdot , \cdot \right> \colon \R^d \times \R^d \to \R $
be the $ d $-dimensional Euclidean scalar product, 
let $ \varphi \colon \R^d \to \R $ be a continuous function, 
let $ \mu \colon \R^d \to \R^d $ be a locally Lipschitz continuous function, 
and assume that
\begin{equation}
\inf_{ p \in (0,\infty) }
\sup_{ x \in \R^d }
\left[
\frac{ | \varphi(x) | }{
(
1 + \| x \|^p
)
}
+
\frac{
\left\| \mu(x) \right\|
}{
(
1 + \| x \|
)
}
\right]
< \infty
.
\end{equation}
Then 
\begin{enumerate}[(i)]
\item 
there exist unique stochastic processes
$ X^x \colon [0,T] \times \Omega \to \R^d $, $ x \in \R^d $, 
with continuous sample paths 
which satisfy for all $ x \in \R^d $, $ t \in [0,T] $ that
\begin{equation}
X^x_t = x + \int_0^t \mu( X^x_s ) \, ds + B W_t 
,
\end{equation}
\item 
there exists a unique function
$
u \colon [0,T] 
\times \R^d \to \R
$ 
such that for all $ x \in \R^d $ it holds that
$
u( 0, x ) = \varphi(x)
$,
such that 
$
\inf_{ p \in (0,\infty) }
\sup_{ (t,x) \in [0,T] \times \R^d }
\frac{ | u(t,x) | }{
1 + \| x \|^p
}
< \infty
$,
and such that $ u $ is a viscosity solution of 
\begin{equation}
\label{eq:CD.F4}
( \tfrac{ \partial }{ \partial t } 
u )(t,x)
=
\big\langle 
(\nabla_x u)(t,x),
\mu(x) 
\big\rangle 
+
\tfrac{1}{2}
\operatorname{Trace}\!\big(
B 
B^{*}       
(\textup{Hess}_x u)(t,x) 
\big)
\end{equation}
for $ (t,x) \in (0,T) \times \R^d $, and
\item 
it holds for all $ t \in [0,T] $, $ x \in \R^d $ that 
$
\E\big[ 
| \varphi( X^x_t ) | 
\big] < \infty
$
and
\begin{equation}
u(t,x) = \E\big[ \varphi( X^x_t ) \big]
.
\end{equation}
\end{enumerate}
\end{theorem}

\section{Stochastic differential equations (SDEs)}
\label{sec:sdes}

In our proofs of Theorem~\ref{thm:intro} above
and Theorem~\ref{thm:PDE_approx_Lp} below (the main result of this article), 
respectively, we design and analyse (cf.\ items~\eqref{item:I}--\eqref{item:II} 
in Section~\ref{sec:intro} above) a suitable random DNN. 
The construction of this suitable random DNN is based on 
Euler-Maruyama discretizations of solutions of the SDEs 
associated to the Kolmogorov PDEs in~\eqref{eq:PDE_intro} 
and for our error analysis of this suitable random DNN 
we employ appropriate weak error estimates 
for Euler-Maruyama discretizations of solutions of SDEs. 
These weak error estimates are established in Lemma~\ref{lem:perturbation_PDE} and Proposition~\ref{prop:perturbation_PDE_2}
in Subsection~\ref{sec:weak_perturbation} below.
Our proofs of Lemma~\ref{lem:perturbation_PDE} and Proposition~\ref{prop:perturbation_PDE_2}, respectively, use
suitable strong error estimates for Euler-Maruyama discretizations. 
These strong error estimates are
the subject of Proposition~\ref{prop:perturbation_SDE} in Subsection~\ref{sec:strong_perturbation} below. 
Proposition~\ref{prop:perturbation_SDE} follows from an application of the deterministic perturbation-type 
inequality in Lemma~\ref{lem:per:pathwise} in Subsection~\ref{sec:strong_perturbation} below. 
Perturbation estimates which are related to Lemma~\ref{lem:per:pathwise} and Proposition~\ref{prop:perturbation_SDE}
can, e.g., be found 
in Hutzenthaler et al.~\cite[Proposition~2.9 and Corollary~2.12]{HutzenthalerJentzen2014}.
In particular, our proof of Lemma~\ref{lem:per:pathwise} is inspired by the proof 
of Proposition~2.9 in Hutzenthaler et al.~\cite{HutzenthalerJentzen2014}. 
Furthermore, our proof of Proposition~\ref{prop:perturbation_PDE_2} employs the elementary a~priori estimate 
in Lemma~\ref{lem:sde-lp-bound} in Subsection~\ref{sec:a_priori_SDEs} below. 
Lemma~\ref{lem:sde-lp-bound}, in turn, is a straightforward consequence of Gronwall's integral 
inequality (see, e.g., Grohs et al.~\cite[Lemma~2.11]{GrohsWurstemberger2018})
and its proof is therefore omitted. 
In our proof of Theorem~\ref{thm:PDE_approx_Lp}
we will also employ the elementary a~priori estimate 
for standard Brownian motions in Lemma~\ref{l:exp.Gauss} in Subsection~\ref{sec:a_priori_BMs} below. 
Lemma~\ref{l:exp.Gauss} is a straightforward consequence of It\^o's formula and its proof is therefore also omitted.

\subsection{A priori bounds for SDEs}
\label{sec:a_priori_SDEs}

\begin{lemma}
	\label{lem:sde-lp-bound}
	Let $ d, m \in \N $, $ \xi \in \R^d $, $ p \in [1,\infty) $, 
	$ c, C, T \in [0,\infty) $, 
	$ B \in \R^{ d \times m } $,  
	let 
	$
	\left \| \cdot \right \| \colon \R^d \to [0,\infty)
	$ be the $d$-dimensional Euclidean norm,
	let $ (\Omega, \mathcal{F},\P) $ be a probability space, 
	let $W\colon [0,T]\times \Omega \to \R^m$ be a standard Brownian motion,
	let $\mu \colon \R^d \to \R^d$
	be a $\mathcal{B}(\R^d) / \mathcal{B}(\R^d)$-measurable function which
	satisfies for all $x \in \R^d$ that
	$\|\mu(x)\| \leq C+c\|x\|$,
	let $\chi \colon  [0,T] \to [0,T]$ be a $\mathcal{B}([0,T]) /
	\mathcal{B} ([0,T])$-measurable function which satisfies for all $t \in
	[0,T]$ that $\chi(t) \leq t$,
	and let $X  \colon [0,T]\times \Omega \to \R^d$ be a stochastic
	process with continuous sample paths which satisfies for all 
	$ t \in [0,T] $ that
	\begin{equation}
	\label{eq:apriori1-ass1}
	\P\!\left(X_t = \xi + \int_0^t \mu\!\left(X_{\chi(s)}\right) ds +
	BW_t\right) = 1.
	\end{equation}
	Then it holds that
	\begin{equation}
	\label{eq:lemma:apriori}
	\begin{split}
	\sup_{t \in [0, T]} \big(\E\! \left[\|X_t\|^p\right]\big)^{\nicefrac{1}{p}} 
	& \leq
	\Big(
	\|\xi\| + C T 
	+ 
	\big(
	\E\big[ 
	\| B W_T \|^p
	\big]
	\big)^{ \nicefrac{ 1 }{ p } }
	\Big)
	\,
	e^{ c T } .
	\end{split}
	\end{equation}
\end{lemma}

\subsection{A priori bounds for Brownian motions}
\label{sec:a_priori_BMs}

\begin{lemma}
	\label{l:exp.Gauss}
	Let
	$ d, m \in \N$,
	$ T \in [0, \infty ) $,
	$ p \in (0,\infty)$,
	$ 
	B \in \R^{d \times m} 
	$, 
	let 
	$
	\left \| \cdot \right \| \colon \R^d \to [0,\infty)
	$ 
	be the $d$-dimensional Euclidean norm,
	let
	$
	( \Omega, \mathcal{F}, \P ) 
	$
	be a probability space,
	and let
	$
	W \colon [0,T] \times \Omega \to \R^m
	$
	be a standard Brownian motion. 
	Then it holds for all $ t \in [0,T] $ that 
	\begin{equation}
	\label{eq:bm-lp}
	\begin{split}
	\big( \E\big[ \| B W_t \|^p \big] \big)^{ \nicefrac{1}{p} }
	& \leq 
	\sqrt{ \max\{1,p-1\} \operatorname{Trace}(B^{ \ast } B) \, t} 
	.
	\end{split}
	\end{equation}
\end{lemma}

\subsection{Strong perturbations of SDEs}
\label{sec:strong_perturbation}

\begin{lemma}
	\label{lem:per:pathwise}
	Let $ d \in \N $, $ L, T \in [0,\infty) $, $ \delta \in (0,\infty) $,
	$ p \in [2,\infty) $, 
	let $ \left\| \cdot \right\| \colon \R^d \to [0,\infty) $ be the $ d $-dimensional Euclidean norm, 
	let $ \mu \colon \R^d \to \R^d $ be a function which satisfies for all 
	$ v, w \in \R^d $ that
	\begin{equation}
	\label{eq:global:Lip}
	\| \mu( v ) - \mu( w ) \| \leq L \| v - w \|
	,
	\end{equation}
	let $ X, Y \colon [0,T] \to \R^d $ be continuous functions, 
	let 
	$ 
	a \colon [0,T] \to \R^d $ be a
	$  \mathcal{B}( [0,T] )  $/$ \mathcal{B}( \R^d ) 
	$-measurable function, 
	and assume for all $ t \in [0,T] $ that 
	$
	\int_0^t \| a_s \| \, ds < \infty
	$
	and
	\begin{equation}
	\label{eq:strong:per}
	X_t - Y_t = X_0 - Y_0 + \int_0^t \big[ \mu( X_s ) -a_s \big] \, ds 
	.
	\end{equation}
	Then it holds for all $ t \in [0,T] $ that
	\begin{equation}
	\label{eq:lem:per}
	\begin{split}
	&
	\| X_t - Y_t \|^p \\
	&\leq
	\exp\!\left( 
	\left[ 
	L 
	+
	\tfrac{ ( 1 - 1 / p ) }{ \delta }
	\right]
	p \, t
	\right)
	\left(
	\left\| X_0 - Y_0 \right\|^p
	+
	\delta^{(p-1) }
	\int_0^t
	\| a_s - \mu( Y_s ) \|^p
	\, ds
	\right)
	.
	\end{split}
	\end{equation}
\end{lemma}
\begin{proof}[Proof of Lemma~\ref{lem:per:pathwise}]
	Throughout this proof let $\langle \cdot, \cdot \rangle \colon \R^d \times \R^d \to \R$ be the $d$-di\-men\-sion\-al Euclidean scalar product and let $ \alpha \in (0, \infty)$ be the real number given by 
	\begin{equation}
	\alpha = \left[ 
	L 
	+
	\tfrac{ ( 1 - 1 / p ) }{ \delta }
	\right] p.
	\end{equation}	
	Note that \eqref{eq:strong:per} ensures that the function $([0, T] \ni t \mapsto (X_t - Y_t) \in \R^d )$ is absolutely continuous.
	The fundamental theorem of calculus and the chain rule  hence prove that for all $t \in [0, T]$ it holds that
	\begin{equation}
	\label{eq:strong:ftc}
	\begin{split}
	\frac{\|X_t - Y_t \|^p}{\exp(\alpha t)} & = \|X_0 - Y_0 \|^p + \int_0^t \frac{p \, \|X_s - Y_s\|^{p-2} \langle X_s - Y_s, \mu(X_s) - a_s \rangle}{\exp(\alpha s)} \, ds \\
	& \quad - \int_0^t \frac{\alpha \|X_s - Y_s\|^p}{\exp(\alpha s)} \, ds \\
	& = \| X_0 - Y_0 \|^p + \int_0^t \frac{p \, \|X_s - Y_s\|^{p-2} \langle X_s - Y_s, \mu(X_s) - \mu(Y_s) \rangle}{\exp(\alpha s)} \, ds \\
	& \quad + \int_0^t \frac{p \, \|X_s - Y_s\|^{p-2} \langle X_s - Y_s, \mu(Y_s) - a_s \rangle - \alpha \|X_s - Y_s\|^p}{\exp(\alpha s)} \, ds.
	\end{split}
	\end{equation}
	Next observe that \eqref{eq:global:Lip} and the Cauchy-Schwartz inequality ensure that for all $ s \in [0, T]$ it holds that
	\begin{equation}
	\begin{split}
	\langle X_s - Y_s, \mu(X_s) - \mu(Y_s) \rangle \leq \|X_s - Y_s\| \|\mu(X_s) - \mu(Y_s) \| \leq L \|X_s - Y_s \|^2.
	\end{split}
	\end{equation}
	This and \eqref{eq:strong:ftc} demonstrate that for all $t \in [0, T]$ it holds that
	\begin{align}
	\label{eq:apriori}
	\nonumber
	& \frac{\|X_t - Y_t \|^p}{\exp(\alpha t)} \leq \| X_0 - Y_0 \|^p + \int_0^t \frac{p L \|X_s - Y_s\|^{p} }{\exp(\alpha s)} \, ds  \\
	& \quad + \int_0^t \frac{p \, \|X_s - Y_s\|^{p-2} \langle X_s - Y_s, \mu(Y_s) - a_s \rangle - \alpha \|X_s - Y_s\|^p}{\exp(\alpha s)} \, ds\\
	\nonumber
	& =   \| X_0 - Y_0 \|^p + \int_0^t \frac{p \, \|X_s - Y_s\|^{p-2} \langle X_s - Y_s, \mu(Y_s) - a_s \rangle - ( \alpha - p L) \|X_s - Y_s\|^p}{\exp(\alpha s)} \, ds\\
	\nonumber
	& = \| X_0 - Y_0 \|^p + \int_0^t \frac{p \, \|X_s - Y_s\|^{p-2} \langle X_s - Y_s, \mu(Y_s) - a_s \rangle - \tfrac{(p-1)}{\delta} \|X_s - Y_s\|^p}{\exp(\alpha s)} \, ds.
	\end{align}
	Next observe that the Cauchy-Schwartz inequality and Young's inequality prove that for all $s \in [0, T]$ it holds that
	\begin{equation}
	\begin{split}
	&\|X_s - Y_s\|^{p-2} \langle X_s - Y_s, \mu(Y_s) - a_s \rangle \leq \|X_s - Y_s\|^{p-1} \|\mu(Y_s) - a_s \| \\
	&= \delta^{\nicefrac{(1-p)}{p}}  \|X_s - Y_s\|^{p-1}  \delta^{\nicefrac{(p-1)}{p}}  \|\mu(Y_s) - a_s \|\\
	& \leq \tfrac{(p-1)}{p} \big[\delta^{\nicefrac{(1-p)}{p}}  \|X_s - Y_s\|^{p-1}\big]^{\nicefrac{p}{(p-1)}} + \tfrac{1}{p} \big[\delta^{\nicefrac{(p-1)}{p}}  \|\mu(Y_s) - a_s \| \big]^p\\
	& = \tfrac{(p-1)}{\delta p} \|X_s - Y_s \|^p + \tfrac{\delta^{(p-1)}}{p} \|\mu(Y_s) - a_s\|^p.
	\end{split}
	\end{equation}
	Combining this with \eqref{eq:apriori} assures that for all $t \in [0, T]$ it holds that
	\begin{align}
	\nonumber
	&\frac{\|X_t - Y_t \|^p}{\exp(\alpha t)}  \\
	\nonumber
	&\leq \|X_0 - Y_0 \|^p 
	+ \int_0^t \frac{ \tfrac{(p-1)}{\delta } \|X_s - Y_s \|^p + \delta^{(p-1)} \|\mu(Y_s) - a_s\|^p - \tfrac{(p-1)}{\delta} \|X_s - Y_s\|^p}{\exp(\alpha s)} \, ds \\
	& = \| X_0 - Y_0 \|^p + \int_0^t \frac{ \delta^{(p-1)} \|\mu(Y_s) - a_s\|^p}{\exp(\alpha s)} \, ds
	\\ & \nonumber
	\leq \| X_0 - Y_0 \|^p + \int_0^t \delta^{(p-1)} \|\mu(Y_s) - a_s\|^p \, ds.
	\end{align}
	This implies \eqref{eq:lem:per}. The proof of Lemma~\ref{lem:per:pathwise} is thus completed.
\end{proof}

\begin{prop}[Perturbation]
	\label{prop:perturbation_SDE}
	Let $ d, m \in \N $, $ x, y \in \R^d $, $ L, T \in [0,\infty) $, $ \delta \in (0,\infty) $,
	$ p \in [2,\infty) $, 
	$ B \in \R^{ d \times m } $, 
	let $ \left\| \cdot \right\| \colon \R^d \to [0,\infty) $ be the $ d $-dimensional Euclidean norm, 
	let $ ( \Omega, \mathcal{F}, \P ) $ be a probability space, 
	let $ W \colon [0,T] \times \Omega \to \R^m $ be a standard Brownian motion, 
	let $ \mu \colon \R^d \to \R^d $ be a function which satisfies for all 
	$ v, w \in \R^d $ that
	\begin{equation}
	\label{eq:prop:global:Lip}
	\| \mu( v ) - \mu( w ) \| \leq L \| v - w \|
	,
	\end{equation}
	let $ X, Y \colon [0,T] \times \Omega \to \R^d $ be stochastic processes with 
	continuous sample paths, 
	let 
	$ 
	a \colon [0,T] \times \Omega \to \R^d $ be a
	$ ( \mathcal{B}( [0,T] ) \otimes \mathcal{F} ) $/$ \mathcal{B}( \R^d ) 
	$-measurable function, 
	and assume for all $ t \in [0,T] $ that 
	$
	\int_0^t \| a_s \| \, ds < \infty
	$,
	$
	Y_t = y + \int_0^t a_s \, ds + B W_t
	$,
	and
	\begin{equation}
	X_t = x + \int_0^t \mu( X_s ) \, ds + B W_t
	.
	\end{equation}
	Then it holds for all $ t \in [0,T] $ that
	\begin{equation}
	\label{eq:prop:per}
	\begin{split}
	&
	\big( \E\big[ \| X_t - Y_t \|^p \big] \big)^{ \nicefrac{ 1 }{ p } }
	\\ & 
	\leq
	\exp\!\left( 
	\left[ 
	L 
	+
	\tfrac{ ( 1 - 1 / p ) }{ \delta }
	\right]
	t
	\right)
	\left(
	\left\| x - y \right\|
	+
	\delta^{ ( 1 - 1 / p ) }
	\left[ 
	\int_0^t
	\E\big[
	\| a_s - \mu( Y_s ) \|^p
	\big]
	\, ds
	\right]^{ \nicefrac{ 1 }{ p } }
	\right)
	.
	\end{split}
	\end{equation}
\end{prop}

\begin{proof}[Proof of Proposition~\ref{prop:perturbation_SDE}]
	First, note that for all $t \in [0, T]$ it holds that
	\begin{equation}
	X_t - Y_t = x-y + \int_0^t \big[ \mu(X_s) - a_s \big] \, ds.
	\end{equation}	
	Lemma~\ref{lem:per:pathwise} hence ensures that for all $ t \in [0,T] $ that
	\begin{equation}
	\begin{split}
	&
	\| X_t - Y_t \|^p \\
	&\leq
	\exp\!\left( 
	\left[ 
	L 
	+
	\tfrac{ ( 1 - 1 / p ) }{ \delta }
	\right]
	p \, t
	\right)
	\left(
	\left\| x - y \right\|^p
	+
	\delta^{(p-1) }
	\int_0^t
	\| a_s - \mu( Y_s ) \|^p
	\, ds
	\right)
	.
	\end{split}
	\end{equation}
	This implies that for all $t \in [0, T]$ it holds that
	\begin{equation}
	\begin{split}
	&
	\E \big[ \| X_t - Y_t \|^p  \big]\\
	&\leq
	\exp\!\left( 
	\left[ 
	L 
	+
	\tfrac{ ( 1 - 1 / p ) }{ \delta }
	\right]
	p \, t
	\right)
	\left(
	\left\| x - y \right\|^p
	+
	\delta^{(p-1) }
	\int_0^t
	\E \big[\| a_s - \mu( Y_s ) \|^p \big]
	\, ds
	\right)
	.
	\end{split}
	\end{equation}
	The fact that $ \forall \, b, c \in \R \colon |b +c|^{\nicefrac{1}{p}} \leq |b|^{\nicefrac{1}{p}} + |c|^{\nicefrac{1}{p}} $ hence demonstrates that for all $t \in [0, T]$ it holds that
	\begin{equation}
	\begin{split}
	&
	\big( \E\big[ \| X_t - Y_t \|^p \big] \big)^{ \nicefrac{ 1 }{ p } }
	\\ & 
	\leq
	\exp\!\left( 
	\left[ 
	L 
	+
	\tfrac{ ( 1 - 1 / p ) }{ \delta }
	\right]
	t
	\right)
	\left(
	\left\| x - y \right\|
	+
	\delta^{ ( 1 - 1 / p ) }
	\left[ 
	\int_0^t
	\E\big[
	\| a_s - \mu( Y_s ) \|^p
	\big]
	\, ds
	\right]^{ \nicefrac{ 1 }{ p } }
	\right)
	.
	\end{split}
	\end{equation}	
	The proof of Proposition~\ref{prop:perturbation_SDE} is thus completed.	
\end{proof}

\subsection{Weak perturbations of SDEs}
\label{sec:weak_perturbation}

\begin{lemma}
	\label{lem:perturbation_PDE}
	Let $ d, m \in \N $, $ \xi \in \R^d $, 
	$ h, T, \varepsilon_0, \varepsilon_1, \varsigma_0, \varsigma_1, L_0, L_1, \ell \in [0,\infty) $, 
	$ \delta \in (0,\infty) $, $ B \in \R^{ d \times m } $, 
	$p \in [2, \infty)$,
	$q \in (1, 2]$ 
	satisfy
	$
	\nicefrac{ 1 }{ p } + \nicefrac{ 1 }{ q } = 1
	$,
	let $ \left\| \cdot \right\| \colon \R^d \to [0,\infty) $ be the $ d $-dimensional Euclidean norm, 
	let $ ( \Omega, \mathcal{F}, \P ) $ be a probability space, 
	let $ W \colon [0,T] \times \Omega \to \R^m $ be a standard Brownian motion, 
	let 
	$ \phi_0 \colon \R^d \to \R $, 
	$ f_1 \colon \R^d \to \R^d $, 
	$ \phi_2 \colon \R^d \to \R^d $, 
	and
	$ \chi \colon [0,T] \to [0,T] $ be functions,
	let 
	$ f_0 \colon \R^d \to \R $
	be a
	$ \mathcal{B}( \R^d ) $/$ \mathcal{B}( \R ) $-measurable function,
	let
	$ \phi_1 \colon \R^d \to \R^d $
	be a $ \mathcal{B}( \R^d ) $/$ \mathcal{B}( \R^d ) $-measurable function, 
	assume for all 
	$ t \in [0,T] $, $ x, y \in \R^d $ that
	\begin{equation}
	| \phi_0( x ) - f_0( x ) |
	\leq 
	\varepsilon_0 
	( 1 + \| x \|^{ \varsigma_0 } )
	,
	\qquad
	\| \phi_1( x ) - f_1( x ) \|
	\leq 
	\varepsilon_1
	( 1 + \| x \|^{ \varsigma_1 } )
	,
	\end{equation}
	\begin{equation}
	| \phi_0( x ) - \phi_0( y ) | 
	\leq 
	L_0
	\left(
	1 
	+ 
	\int_0^1
	\big[ 
	r \| x \| 
	+ 
	( 1 - r ) 
	\| y \|
	\big]^{ \ell }
	\,
	dr
	\right)
	\left\| x - y \right\| ,
	\end{equation}
	\begin{equation}
	\| f_1( x ) - f_1( y ) \| \leq L_1 \| x - y \| 
	,
	\qquad 
	\text{and}
	\qquad
	\chi( t ) 
	= 
	\max\!\left(
	\{ 0, h , 2 h, \dots \}
	\cap [0,t] 
	\right)
	,
	\end{equation}
	and 
	let $ X, Y \colon [0,T] \times \Omega \to \R^d $ be stochastic processes with 
	continuous sample paths which satisfy
	for all $ t \in [0,T] $ that 
	$
	Y_t = \phi_2( \xi ) + \int_0^t \phi_1\big( Y_{ \chi( s ) } \big) \, ds + B W_t
	$
	and
	\begin{equation}
	X_t = \xi + \int_0^t f_1( X_s ) \, ds + B W_t
	.
	\end{equation}
	Then it holds that
	\begin{align*}
	&
	\big| \E\big[ f_0( X_T ) \big] - \E\big[ \phi_0( Y_T ) \big] \big| \numberthis
	\\ &  
	\leq
	\varepsilon_0 
	\left(
	1 
	+
	\E\!\left[ 
	\| X_T \|^{ \varsigma_0 }
	\right]
	\right)
	\\ & \quad
	+
	L_0
	\,
	2^{ \max\{ \ell - 1, 0 \} }
	\exp\!\left( 
	\left[ 
	L_1 
	+
	\tfrac{ ( 1 - 1 / p ) }{ \delta }
	\right]
	T
	\right)
	\left[
	1 
	+
	\left(
	\E\big[
	\| X_T \|^{ \ell q }
	\big]
	\right)^{ \nicefrac{ 1 }{ q } }
	+
	\left(
	\E\big[
	\| Y_T \|^{ \ell q }
	\big]
	\right)^{ \nicefrac{ 1 }{ q } }
	\right]
	\\ & \quad \cdot 
	\Bigg[
	\left\| \xi - \phi_2( \xi ) \right\|
	+
	\varepsilon_1 
	\,
	\delta^{ ( 1 - \nicefrac{ 1 }{ p } ) }
	\, 
	T^{ \nicefrac{ 1 }{ p } }
	\!
	\left[
	1
	+
	\sup_{ t \in [0,T] }
	\big(
	\E\big[
	\| Y_t \|^{ p \varsigma_1 }
	\big]
	\big)^{ \nicefrac{ 1 }{ p } }
	\right]
	\\ & \quad
	+
	h
	\,
	\delta^{ ( 1 - \nicefrac{ 1 }{ p } ) }
	\,
	T^{ \nicefrac{ 1 }{ p } }
	L_1 
	\!
	\left[
	\sup_{ t \in [0,T] }
	\big( 
	\E\big[
	\|
	\phi_1( Y_t )
	\|^p
	\big]
	\big)^{ \nicefrac{ 1 }{ p } }
	\right]
	+
	\delta^{ ( 1 - \nicefrac{ 1 }{ p } ) }
	\,
	T^{ \nicefrac{ 1 }{ p } }
	L_1 
	\big(
	\E\big[
	\| 
	B W_h 
	\|^p
	\big]
	\big)^{ \nicefrac{ 1 }{ p } }
	\Bigg]
	.
	\end{align*}
\end{lemma}

\begin{proof}[Proof of Lemma~\ref{lem:perturbation_PDE}]
	First, note that the triangle inequality ensures that
	\begin{equation}
	\begin{split}
	&
	\big| \E\big[ f_0( X_T ) \big] - \E\big[ \phi_0( Y_T ) \big] \big|
	\\ & 
	\leq
	\big| \E\big[ f_0( X_T ) \big] - \E\big[ \phi_0( X_T ) \big] \big|
	+
	\big| \E\big[ \phi_0( X_T ) \big] - \E\big[ \phi_0( Y_T ) \big] \big|
	\\ &
	\leq
	\E\big[ 
	| f_0( X_T ) - \phi_0( X_T ) |
	\big]
	+
	\E\big[ 
	| \phi_0( X_T ) - \phi_0( Y_T ) |
	\big]
	\\ &
	\leq
	\varepsilon_0 
	\,
	\E\!\left[ 
	1 + \| X_T \|^{ \varsigma_0 }
	\right]
	+
	\E\big[ 
	| \phi_0( X_T ) - \phi_0( Y_T ) |
	\big]
	.
	\end{split}
	\end{equation}
	This implies that
	\begin{equation}
	\begin{split}
	&
	\big| \E\big[ f_0( X_T ) \big] - \E\big[ \phi_0( Y_T ) \big] \big|
	\\ &
	\leq
	\varepsilon_0 
	\,
	\E\!\left[ 
	1 + \| X_T \|^{ \varsigma_0 }
	\right]
	\\ & \quad
	+
	L_0 \,
	\E\!\left[ 
	\left(
	1 
	+
	\int_0^1
	\big[ 
	r \| X_T \| 
	+ 
	( 1 - r ) 
	\| Y_T \|
	\big]^{ \ell }
	\, dr
	\right)
	\left\| 
	X_T - Y_T
	\right\|
	\right]
	\\ &
	\leq
	\varepsilon_0 
	\,
	\E\!\left[ 
	1 + \| X_T \|^{ \varsigma_0 }
	\right]
	\\ & \quad
	+
	L_0 \,
	\E\!\left[ 
	\left(
	1 
	+
	2^{ \max\{ \ell - 1, 0 \} }
	\int_0^1
	\| r X_T \|^{ \ell } + 
	\| ( 1 - r ) Y_T \|^{ \ell }
	\, dr
	\right)
	\left\| 
	X_T - Y_T
	\right\|
	\right]
	.
	\end{split}
	\end{equation}
	Therefore, we obtain that
	\begin{equation}
	\begin{split}
	&
	\big| \E\big[ f_0( X_T ) \big] - \E\big[ \phi_0( Y_T ) \big] \big|
	\\ &
	\leq
	\varepsilon_0 
	\,
	\E\!\left[ 
	1 + \| X_T \|^{ \varsigma_0 }
	\right]
	\\ & \quad
	+
	L_0 \,
	\E\!\left[ 
	\left(
	1 
	+
	2^{ \max\{ \ell - 1, 0 \} }
	\left[
	\int_0^1
	r^{ \ell }
	\, 
	dr
	\right]
	\big[
	\| X_T \|^{ \ell } 
	+ 
	\| Y_T \|^{ \ell }
	\big]
	\right)
	\left\| 
	X_T - Y_T
	\right\|
	\right]
	\\ &
	=
	\varepsilon_0 
	\,
	\E\!\left[ 
	1 + \| X_T \|^{ \varsigma_0 }
	\right]
	\\ & \quad
	+
	L_0 \,
	\E\!\left[ 
	\left(
	1 
	+
	\left[
	\frac{ 
		2^{ \max\{ \ell - 1, 0 \} }
	}{
		( \ell + 1 )
	}
	\right]
	\big[
	\| X_T \|^{ \ell } 
	+ 
	\| Y_T \|^{ \ell }
	\big]
	\right)
	\left\| 
	X_T - Y_T
	\right\|
	\right]
	.
	\end{split}
	\end{equation}
	H\"{o}lder's inequality hence demonstrates that
	\begin{align*}
	\label{eq:after_Hoelder}
	&
	\big| \E\big[ f_0( X_T ) \big] - \E\big[ \phi_0( Y_T ) \big] \big|
	\\ & 
	\leq
	\varepsilon_0 
	\left(
	1 
	+
	\E\!\left[ 
	\| X_T \|^{ \varsigma_0 }
	\right]
	\right) \numberthis
	\\ & \quad
	+
	L_0
	\left(
	1 
	+
	\left[
	\frac{ 
		2^{ \max\{ \ell - 1, 0 \} }
	}{
		( \ell + 1 )
	}
	\right]
	\left[
	\left(
	\E\big[
	\| X_T \|^{ \ell q }
	\big]
	\right)^{ \nicefrac{ 1 }{ q } }
	+
	\left(
	\E\big[
	\| Y_T \|^{ \ell q }
	\big]
	\right)^{ \nicefrac{ 1 }{ q } }
	\right]
	\right)
	\left(
	\E\!\left[ 
	\left\| 
	X_T - Y_T
	\right\|^p
	\right]
	\right)^{ \nicefrac{ 1 }{ p } }
	.
	\end{align*}
	Next observe that 
	Proposition~\ref{prop:perturbation_SDE}
	(with $d = d$, $m =m$, $x = \xi$, $y = \phi_2(\xi)$, $L = L_1$, $T = T$, $\delta = \delta$, $p = p$, $B=B$,  $(\Omega, \mathcal{F}, \P) = (\Omega, \mathcal{F}, \P )$, $W = W$, $\mu = f_1$, $X = X$, $Y = Y$, $a = ([0, T] \times \Omega \ni (t, \omega) \mapsto \phi_1(Y_{\chi(t)}(\omega)) \in \R^d )$ in the notation of Proposition~\ref{prop:perturbation_SDE})
	ensures that 
	\begin{align*}
	&
	\left(
	\E\!\left[ 
	\left\| 
	X_T - Y_T
	\right\|^p
	\right]
	\right)^{ \nicefrac{ 1 }{ p } }
	\\ &
	\leq
	\exp\!\left( 
	\left[ 
	L_1 
	+
	\tfrac{ ( 1 - 1 / p ) }{ \delta }
	\right]
	T
	\right)
	\\ & \quad \cdot
	\left(
	\left\| \xi - \phi_2( \xi ) \right\|
	+
	\delta^{ ( 1 - \nicefrac{ 1 }{ p } ) }
	\!
	\left[ 
	\int_0^T
	\E\big[
	\| \phi_1( Y_{ \chi(s) } ) - f_1( Y_s ) \|^p
	\big]
	\, ds
	\right]^{ \nicefrac{ 1 }{ p } }
	\right)
	\\ &
	\leq
	\exp\!\left( 
	\left[ 
	L_1 
	+
	\tfrac{ ( 1 - 1 / p ) }{ \delta }
	\right]
	T
	\right)
	\\ & \quad \cdot
	\left(
	\left\| \xi - \phi_2( \xi ) \right\|
	+
	\delta^{ ( 1 - \nicefrac{ 1 }{ p } ) }
	\!
	\left[ 
	\int_0^T
	\E\big[
	\| \phi_1( Y_{ \chi(s) } ) - f_1( Y_{ \chi(s) } ) \|^p
	\big]
	\, ds
	\right]^{ \nicefrac{ 1 }{ p } }
	\right)
	\numberthis
	\\ & \quad
	+
	\exp\!\left( 
	\left[ 
	L_1 
	+
	\tfrac{ ( 1 - 1 / p ) }{ \delta }
	\right]
	T
	\right)
	\left(
	\delta^{ ( 1 - \nicefrac{ 1 }{ p } ) }
	\!
	\left[ 
	\int_0^T
	\E\big[
	\| f_1( Y_{ \chi(s) } ) - f_1( Y_s ) \|^p
	\big]
	\, ds
	\right]^{ \nicefrac{ 1 }{ p } }
	\right)
	\\ &
	\leq
	\exp\!\left( 
	\left[ 
	L_1 
	+
	\tfrac{ ( 1 - 1 / p ) }{ \delta }
	\right]
	T
	\right)
	\left(
	\left\| \xi - \phi_2( \xi ) \right\|
	+
	\varepsilon_1 
	\,
	\delta^{ ( 1 - \nicefrac{ 1 }{ p } ) }
	\!
	\left[ 
	\int_0^T
	\E\big[
	( 1 + \| Y_{ \chi( s ) } \|^{ \varsigma_1 } )^p
	\big]
	\, ds
	\right]^{ \nicefrac{ 1 }{ p } }
	\right)
	\\ & \quad
	+
	\exp\!\left( 
	\left[ 
	L_1 
	+
	\tfrac{ ( 1 - 1 / p ) }{ \delta }
	\right]
	T
	\right)
	L_1 
	\,
	\delta^{ ( 1 - \nicefrac{ 1 }{ p } ) }
	\!
	\left[ 
	\int_0^T
	\E\big[
	\| 
	Y_{ \chi(s) } - Y_s
	\|^p
	\big]
	\, ds
	\right]^{ \nicefrac{ 1 }{ p } }
	.
	\end{align*}
	This shows that
	\begin{equation}
	\label{eq:before_perturbation}
	\begin{split}
	&
	\left(
	\E\!\left[ 
	\left\| 
	X_T - Y_T
	\right\|^p
	\right]
	\right)^{ \nicefrac{ 1 }{ p } }
	\\ &
	\leq
	\exp\!\left( 
	\left[ 
	L_1 
	+
	\tfrac{ ( 1 - 1 / p ) }{ \delta }
	\right]
	T
	\right)
	\left\| \xi - \phi_2( \xi ) \right\|
	\\ & \quad
	+
	\exp\!\left( 
	\left[ 
	L_1 
	+
	\tfrac{ ( 1 - 1 / p ) }{ \delta }
	\right]
	T
	\right)
	\varepsilon_1 
	\,
	\delta^{ ( 1 - \nicefrac{ 1 }{ p } ) }
	\!
	\left[
	T^{ \nicefrac{ 1 }{ p } }
	+
	\left[ 
	\int_0^T
	\E\big[
	\| Y_{ \chi(s) } \|^{ p \varsigma_1 }
	\big]
	\, ds
	\right]^{ \nicefrac{ 1 }{ p } }
	\right]
	\\ & \quad
	+
	\exp\!\left( 
	\left[ 
	L_1 
	+
	\tfrac{ ( 1 - 1 / p ) }{ \delta }
	\right]
	T
	\right)
	L_1 
	\,
	\delta^{ ( 1 - \nicefrac{ 1 }{ p } ) }
	\\ & \quad
	\cdot
	\left[ 
	\int_0^T
	\E\Big[
	\big\| 
	\smallint\nolimits_{ \chi(s) }^s 
	\phi_1( Y_{ \chi( u ) } ) 
	\, du
	+
	B ( W_s - W_{ \chi(s) } )
	\big\|^p
	\Big]
	\,
	ds
	\right]^{ \nicefrac{ 1 }{ p } }
	.
	\end{split}
	\end{equation}
	Moreover, observe that the triangle inequality assures that
	\begin{equation}
	\begin{split}
	&\left[ 
	\int_0^T
	\E\Big[
	\big\| 
	\smallint\nolimits_{ \chi(s) }^s 
	\phi_1( Y_{ \chi( u ) } ) 
	\, du
	+
	B ( W_s - W_{ \chi(s) } )
	\big\|^p
	\Big]
	\,
	ds
	\right]^{ \nicefrac{ 1 }{ p } } \\
	& \leq 
	\left[ 
	\int_0^T
	\E\Big[
	\big\| 
	\smallint\nolimits_{ \chi(s) }^s 
	\phi_1( Y_{ \chi( u ) } ) 
	\, du
	\big\|^p
	\Big]
	\,
	ds
	\right]^{ \nicefrac{ 1 }{ p } } 
	+ 
	\left[ 
	\int_0^T
	\E\Big[
	\big\| 
	B ( W_s - W_{ \chi(s) } )
	\big\|^p
	\Big]
	\,
	ds
	\right]^{ \nicefrac{ 1 }{ p } }\\
	& = 
	\left[ 
	\int_0^T
	\E\Big[ |s - \chi(s)|^{p}
	\|\phi_1( Y_{ \chi( s ) } ) \|^p
	\Big]
	\,
	ds
	\right]^{ \nicefrac{ 1 }{ p } } 
	+  
	\left[ 
	\int_0^T
	\E\Big[
	\big\| 
	B ( W_{s -  \chi(s) } )
	\big\|^p
	\Big]
	\,
	ds
	\right]^{ \nicefrac{ 1 }{ p } }\\
	& \leq h \left[ 
	\int_0^T
	\E\Big[ 
	\|\phi_1( Y_{ \chi( s ) } ) \|^p
	\Big]
	\,
	ds
	\right]^{ \nicefrac{ 1 }{ p } } 
	+  
	\left[ 
	\int_0^T
	\E\Big[
	\big\| 
	B ( W_{s -  \chi(s) } )
	\big\|^p
	\Big]
	\,
	ds
	\right]^{ \nicefrac{ 1 }{ p } }
	.
	\end{split}
	\end{equation}
	This and \eqref{eq:before_perturbation} show that
	\begin{equation}
	\begin{split}
	&
	\left(
	\E\!\left[ 
	\left\| 
	X_T - Y_T
	\right\|^p
	\right]
	\right)^{ \nicefrac{ 1 }{ p } }
	\\ &
	\leq
	\exp\!\left( 
	\left[ 
	L_1 
	+
	\tfrac{ ( 1 - 1 / p ) }{ \delta }
	\right]
	T
	\right)
	\left\| \xi - \phi_2( \xi ) \right\|
	\\ & \quad
	+
	\exp\!\left( 
	\left[ 
	L_1 
	+
	\tfrac{ ( 1 - 1 / p ) }{ \delta }
	\right]
	T
	\right)
	\varepsilon_1 
	\,
	\delta^{ ( 1 - \nicefrac{ 1 }{ p } ) }
	\!
	\left[
	T^{ \nicefrac{ 1 }{ p } }
	+
	T^{ \nicefrac{ 1 }{ p } }
	\left[ 
	\sup_{ t \in [0,T] }
	\big(
	\E\big[
	\| Y_t \|^{ p \varsigma_1 }
	\big]
	\big)^{ \nicefrac{ 1 }{ p } }
	\right]
	\right]
	\\ & \quad
	+
	\exp\!\left( 
	\left[ 
	L_1 
	+
	\tfrac{ ( 1 - 1 / p ) }{ \delta }
	\right]
	T
	\right)
	h 
	\,
	L_1 
	\,
	\delta^{ ( 1 - \nicefrac{ 1 }{ p } ) }
	\,
	T^{ \nicefrac{ 1 }{ p } }
	\!
	\left[
	\sup_{ t \in [0,T] }
	\big( 
	\E\big[
	\|
	\phi_1( Y_{ \chi(t) } )
	\|^p
	\big]
	\big)^{ \nicefrac{ 1 }{ p } }
	\right]
	\\ & \quad
	+
	\exp\!\left( 
	\left[ 
	L_1 
	+
	\tfrac{ ( 1 - 1 / p ) }{ \delta }
	\right]
	T
	\right)
	L_1 
	\,
	\delta^{ ( 1 - \nicefrac{ 1 }{ p } ) }
	\left[ 
	\int_0^T
	\E\Big[
	\big\| 
	B ( W_{ s - \chi(s) } )
	\big\|^p
	\Big]
	\,
	ds
	\right]^{ \nicefrac{ 1 }{ p } }
	.
	\end{split}
	\end{equation}
	Therefore, we obtain that
	\begin{equation}
	\begin{split}
	&
	\left(
	\E\!\left[ 
	\left\| 
	X_T - Y_T
	\right\|^p
	\right]
	\right)^{ \nicefrac{ 1 }{ p } }
	\\ & \leq
	\exp\!\left( 
	\left[ 
	L_1 
	+
	\tfrac{ ( 1 - 1 / p ) }{ \delta }
	\right]
	T
	\right)
	\left\| \xi - \phi_2( \xi ) \right\|
	\\ & \quad
	+
	\varepsilon_1 
	\,
	\delta^{ ( 1 - \nicefrac{ 1 }{ p } ) }
	\, 
	T^{ \nicefrac{ 1 }{ p } }
	\exp\!\left( 
	\left[ 
	L_1 
	+
	\tfrac{ ( 1 - 1 / p ) }{ \delta }
	\right]
	T
	\right)
	\left[
	1
	+
	\sup_{ t \in [0,T] }
	\big(
	\E\big[
	\| Y_t \|^{ p \varsigma_1 }
	\big]
	\big)^{ \nicefrac{ 1 }{ p } }
	\right]
	\\ & \quad
	+
	h 
	\,
	\delta^{ ( 1 - \nicefrac{ 1 }{ p } ) }
	\,
	T^{ \nicefrac{ 1 }{ p } }
	\,
	L_1 
	\exp\!\left( 
	\left[ 
	L_1 
	+
	\tfrac{ ( 1 - 1 / p ) }{ \delta }
	\right]
	T
	\right)
	\left[
	\sup_{ t \in [0,T] }
	\big( 
	\E\big[
	\|
	\phi_1( Y_t )
	\|^p
	\big]
	\big)^{ \nicefrac{ 1 }{ p } }
	\right]
	\\ & \quad
	+
	\delta^{ ( 1 - \nicefrac{ 1 }{ p } ) }
	\,
	T^{ \nicefrac{ 1 }{ p } }
	\,
	L_1 
	\exp\!\left( 
	\left[ 
	L_1 
	+
	\tfrac{ ( 1 - 1 / p ) }{ \delta }
	\right]
	T
	\right)
	\big(
	\E\big[
	\| 
	B W_h 
	\|^p
	\big]
	\big)^{ \nicefrac{ 1 }{ p } }
	.
	\end{split} 
	\end{equation}
	Combining this with \eqref{eq:after_Hoelder} demonstrates that
	\begin{align*}
	&
	\big| \E\big[ f_0( X_T ) \big] - \E\big[ \phi_0( Y_T ) \big] \big|
	\\ & 
	\leq
	\varepsilon_0 
	\left(
	1 
	+
	\E\!\left[ 
	\| X_T \|^{ \varsigma_0 }
	\right]
	\right)
	\\ & \quad
	+
	L_0 
	\left(
	1 
	+
	\left[
	\frac{ 
		2^{ \max\{ \ell - 1, 0 \} }
	}{
		( \ell + 1 )
	}
	\right]
	\left[
	\left(
	\E\big[
	\| X_T \|^{ \ell q }
	\big]
	\right)^{ \nicefrac{ 1 }{ q } }
	+
	\left(
	\E\big[
	\| Y_T \|^{ \ell q }
	\big]
	\right)^{ \nicefrac{ 1 }{ q } }
	\right]
	\right) \numberthis
	\\ & \quad \cdot 
	\exp\!\left( 
	\left[ 
	L_1 
	+
	\tfrac{ ( 1 - 1 / p ) }{ \delta }
	\right]
	T
	\right)
	\Bigg[
	\left\| \xi - \phi_2( \xi ) \right\|
	+
	\varepsilon_1 
	\,
	\delta^{ ( 1 - \nicefrac{ 1 }{ p } ) }
	\, 
	T^{ \nicefrac{ 1 }{ p } }
	\!
	\left[
	1
	+
	\sup_{ t \in [0,T] }
	\big(
	\E\big[
	\| Y_t \|^{ p \varsigma_1 }
	\big]
	\big)^{ \nicefrac{ 1 }{ p } }
	\right]
	\\ & \quad
	+
	h
	\,
	\delta^{ ( 1 - \nicefrac{ 1 }{ p } ) }
	\,
	T^{ \nicefrac{ 1 }{ p } }
	L_1 
	\!
	\left[
	\sup_{ t \in [0,T] }
	\big( 
	\E\big[
	\|
	\phi_1( Y_t )
	\|^p
	\big]
	\big)^{ \nicefrac{ 1 }{ p } }
	\right]
	+
	\delta^{ ( 1 - \nicefrac{ 1 }{ p } ) }
	\,
	T^{ \nicefrac{ 1 }{ p } }
	L_1 
	\big(
	\E\big[
	\| 
	B W_h 
	\|^p
	\big]
	\big)^{ \nicefrac{ 1 }{ p } }
	\Bigg]
	.
	\end{align*}
	The proof of Lemma~\ref{lem:perturbation_PDE} is thus completed.
\end{proof}

\begin{prop}
	\label{prop:perturbation_PDE_2}
	Let $ d, m \in \N $, $ \xi \in \R^d $, 
	$ T \in (0,\infty) $,
	$ c, C, \varepsilon_0, \varepsilon_1, \varepsilon_2, \varsigma_0, \varsigma_1, \varsigma_2, L_0, L_1, $ $\ell \in [0,\infty) $, 
	$ h \in [0,T] $,
	$ B \in \R^{ d \times m } $,
	$p \in [2, \infty)$,
	$q \in (1, 2]$
	satisfy
	$
	\nicefrac{ 1 }{ p } + \nicefrac{ 1 }{ q } = 1
	$,
	let $ \left\| \cdot \right\| \colon \R^d \to [0,\infty) $ be the $ d $-dimensional Euclidean norm, 
	let $ ( \Omega, \mathcal{F}, \P ) $ be a probability space, 
	let $ W \colon [0,T] \times \Omega \to \R^m $ be a standard Brownian motion, 
	let 
	$ \phi_0 \colon \R^d \to \R $, 
	$ f_1 \colon \R^d \to \R^d $, 
	$ \phi_2 \colon \R^d \to \R^d $, 
	and
	$ \chi \colon [0,T] \to [0,T] $ be functions,
	let 
	$ f_0 \colon \R^d \to \R $
	be a
	$ \mathcal{B}( \R^d ) $/$ \mathcal{B}( \R ) $-measurable function,
	let
	$ \phi_1 \colon \R^d \to \R^d $
	be a $ \mathcal{B}( \R^d ) $/$ \mathcal{B}( \R^d ) $-measurable function, 
	assume that 
	$
	\| \xi - \phi_2( \xi ) \|
	\leq
	\varepsilon_2 ( 1 + \| \xi \|^{ \varsigma_2 } )
	$,
	assume for all 
	$ t \in [0,T] $, $ x, y \in \R^d $ that
	\begin{equation}
	| \phi_0( x ) - f_0( x ) |
	\leq 
	\varepsilon_0 
	( 1 + \| x \|^{ \varsigma_0 } )
	,
	\qquad
	\| \phi_1( x ) - f_1( x ) \|
	\leq 
	\varepsilon_1
	( 1 + \| x \|^{ \varsigma_1 } )
	,
	\end{equation}
	\begin{equation}
	| \phi_0( x ) - \phi_0( y ) | 
	\leq 
	L_0
	\left(
	1 
	+ 
	\int_0^1
	\big[
	r \| x \| + ( 1 - r ) \| y \|
	\big]^{ \ell }
	\,
	dr
	\right)
	\left\| x - y \right\| ,
	\end{equation}
	\begin{equation}
	\| f_1( x ) - f_1( y ) \| \leq L_1 \| x - y \| 
	,
	\qquad
	\chi( t ) 
	= 
	\max\!\left(
	\{ 0, h, 2 h , \dots \}
	\cap [0,t] 
	\right)
	,
	\end{equation}
	and 
	$
	\| \phi_1( x ) \|
	\leq 
	C + c \| x \|
	$,
	let $ \varpi_r \in \R $, $ r \in (0,\infty) $, 
	satisfy for all $ r \in (0,\infty) $ that
	$
	\varpi_r = 
	\big(
	\E\big[ 
	\| B W_T \|^r
	\big]
	\big)^{ \nicefrac{1}{r} }   
	$,
	and 
	let $ X, Y \colon [0,T] \times \Omega \to \R^d $ be stochastic processes with 
	continuous sample paths which satisfy
	for all $ t \in [0,T] $ that 
	$
	Y_t = \phi_2( \xi ) + \int_0^t \phi_1\big( Y_{ \chi( s ) } \big) \, ds + B W_t
	$
	and
	\begin{equation}
	X_t = \xi + \int_0^t f_1( X_s ) \, ds + B W_t
	.
	\end{equation}
	Then it holds that
	\begin{equation}
	\begin{split}
	&
	\big| \E\big[ f_0( X_T ) \big] - \E\big[ \phi_0( Y_T ) \big] \big|
	\leq 
	\left[
	\varepsilon_0
	+
	\varepsilon_1
	+
	\varepsilon_2
	+
	( h / T )^{ \nicefrac{ 1 }{ 2 } } 
	\right]
	\\ & \cdot 
	e^{
		( 
		\ell + 3 + 2 L_1 +
		\left[ 
		\ell
		\max\{ 
		L_1, 
		c
		\} 
		+
		c
		\max\{ 
		\varsigma_1 ,
		1
		\}
		+
		L_1
		\max\{ \varsigma_0, 1 \}
		+
		2
		\right]
		T
		)
	}
	\big[ 
	\| \xi \|
	+
	\max\{ 1, \varepsilon_2 \}
	( 1 + \| \xi \|^{ \varsigma_2 } )
	\\ &
	+ 
	\max\{ 1, C, \| f_1( 0 ) \| \} \max\{ 1, T \} 
	+ 
	\varpi_{ \max\{ \varsigma_0, \varsigma_1 p, p, \ell q \} } 
	\big]^{ \max\{ 1, \varsigma_0, \varsigma_1 \} + \ell }
	\max\{ 1, L_0 \}
	.
	\end{split}
	\end{equation}
\end{prop}

\begin{proof}[Proof of Proposition~\ref{prop:perturbation_PDE_2}]
	First, observe that Lemma~\ref{lem:perturbation_PDE} shows that
	\begin{equation}
	\begin{split}
	&
	\big| \E\big[ f_0( X_T ) \big] - \E\big[ \phi_0( Y_T ) \big] \big|
	\\ & 
	\leq
	\varepsilon_0 
	\left(
	1 
	+
	\E\!\left[ 
	\| X_T \|^{ \varsigma_0 }
	\right]
	\right)
	\\ & 
	+
	L_0
	\,
	2^{ \max\{ \ell - 1, 0 \} }
	\,
	e^{
		\left[ 
		L_1 
		+
		( 1 - \nicefrac{1}{p} ) 
		\right]
		T
	}
	\left[
	1 
	+
	\left(
	\E\big[
	\| X_T \|^{ \ell q }
	\big]
	\right)^{ \nicefrac{ 1 }{ q } }
	+
	\left(
	\E\big[
	\| Y_T \|^{ \ell q }
	\big]
	\right)^{ \nicefrac{ 1 }{ q } }
	\right]
	\\ & \cdot 
	\Bigg[
	\left\| \xi - \phi_2( \xi ) \right\|
	+
	\varepsilon_1 
	T^{ \nicefrac{ 1 }{ p } }
	\!
	\left[
	1
	+
	\sup_{ t \in [0,T] }
	\big(
	\E\big[
	\| Y_t \|^{ p \varsigma_1 }
	\big]
	\big)^{ \nicefrac{ 1 }{ p } }
	\right]
	\\ &
	+
	h
	T^{ \nicefrac{ 1 }{ p } }
	L_1 
	\!
	\left[
	\sup_{ t \in [0,T] }
	\big( 
	\E\big[
	\|
	\phi_1( Y_t )
	\|^p
	\big]
	\big)^{ \nicefrac{ 1 }{ p } }
	\right]
	+
	T^{ \nicefrac{ 1 }{ p } }
	L_1 
	\big(
	\E\big[
	\| 
	B W_h 
	\|^p
	\big]
	\big)^{ \nicefrac{ 1 }{ p } }
	\Bigg]
	.
	\end{split}
	\end{equation}
	Hence, we obtain that
	\begin{equation}
	\label{eq:last_estimate_in_the}
	\begin{split}
	&
	\big| \E\big[ f_0( X_T ) \big] - \E\big[ \phi_0( Y_T ) \big] \big|
	\\ & 
	\leq
	\varepsilon_0 
	\left(
	1 
	+
	\E\!\left[ 
	\| X_T \|^{ \varsigma_0 }
	\right]
	\right)
	\\ & 
	+
	L_0 
	\,
	2^{ \max\{ \ell - 1, 0 \} }
	\,
	e^{
		\left[ 
		L_1 
		+
		( 1 - \nicefrac{1}{p} ) 
		\right]
		T
	}
	\left[
	1 
	+
	\left(
	\E\big[
	\| X_T \|^{ \ell q }
	\big]
	\right)^{ \nicefrac{ 1 }{ q } }
	+
	\left(
	\E\big[
	\| Y_T \|^{ \ell q }
	\big]
	\right)^{ \nicefrac{ 1 }{ q } }
	\right]
	\\ & \cdot 
	\Bigg[
	\left\| \xi - \phi_2( \xi ) \right\|
	+
	\varepsilon_1 
	T^{ \nicefrac{ 1 }{ p } }
	\!
	\left[
	1
	+
	\sup_{ t \in [0,T] }
	\big(
	\E\big[
	\| Y_t \|^{ p \varsigma_1 }
	\big]
	\big)^{ \nicefrac{ 1 }{ p } }
	\right]
	\\ &
	+
	h
	\,
	T^{ \nicefrac{ 1 }{ p } }
	L_1
	\varepsilon_1
	\!
	\left[
	\sup_{ t \in [0,T] }
	\big( 
	\E\big[
	( 1 + \| Y_t \|^{ \varsigma_1 } )^p
	\big]
	\big)^{ \nicefrac{ 1 }{ p } }
	\right]
	\\ &
	+
	h
	\,
	T^{ \nicefrac{ 1 }{ p } }
	L_1 
	\!
	\left[
	\sup_{ t \in [0,T] }
	\big( 
	\E\big[
	\|
	f_1( Y_t )
	\|^p
	\big]
	\big)^{ \nicefrac{ 1 }{ p } }
	\right]
	+
	L_1 
	\varpi_p
	h^{ \nicefrac{ 1 }{ 2 } }
	T^{ \nicefrac{ 1 }{ p } - \nicefrac{ 1 }{ 2 } }
	\Bigg]
	.
	\end{split}
	\end{equation}
	In addition, note that for all $ x \in \R^d $ it holds that
	\begin{equation}
	\label{eq:f1_linear_growth}
	\left\| f_1( x ) \right\|
	\leq
	\left\| f_1( x ) - f_1( 0 ) \right\|
	+
	\left\| f_1( 0 ) \right\|
	\leq
	\left\| f_1( 0 ) \right\|
	+
	L_1 \| x \|
	.
	\end{equation}
	This and \eqref{eq:last_estimate_in_the} ensure that
	\begin{equation}
	\label{eq:main_estimate_proof}
	\begin{split}
	&
	\big| \E\big[ f_0( X_T ) \big] - \E\big[ \phi_0( Y_T ) \big] \big|
	\\ & 
	\leq
	\varepsilon_0 
	\,
	\big(
	1 
	+
	\E\big[ 
	\| X_T \|^{ \varsigma_0 }
	\big]
	\big)
	\\ & 
	+
	L_0 
	\,
	2^{ \max\{ \ell - 1, 0 \} }
	\,
	e^{
		\left[ 
		L_1 
		+
		1 - \nicefrac{ 1 }{ p }
		\right]
		T
	}
	\left[
	1 
	+
	\left(
	\E\big[
	\| X_T \|^{ \ell q }
	\big]
	\right)^{ \nicefrac{ 1 }{ q } }
	+
	\left(
	\E\big[
	\| Y_T \|^{ \ell q }
	\big]
	\right)^{ \nicefrac{ 1 }{ q } }
	\right]
	\\ & \cdot 
	\Bigg[
	\left\| \xi - \phi_2( \xi ) \right\|
	+
	\varepsilon_1 
	T^{ \nicefrac{ 1 }{ p } }
	\!
	\left[
	1
	+
	\sup_{ t \in [0,T] }
	\big(
	\E\big[
	\| Y_t \|^{ p \varsigma_1 }
	\big]
	\big)^{ \nicefrac{ 1 }{ p } }
	\right]
	\\ &
	+
	h
	T^{ \nicefrac{ 1 }{ p } }
	L_1
	\varepsilon_1
	\!
	\left[
	1 +
	\sup_{ t \in [0,T] }
	\big( 
	\E\big[
	\| Y_t \|^{ p \varsigma_1 }
	\big]
	\big)^{ \nicefrac{ 1 }{ p } }
	\right]
	\\ &
	+
	h
	T^{ \nicefrac{ 1 }{ p } }
	L_1 
	\!
	\left[
	\| f_1( 0 ) \|
	+
	L_1 
	\left[
	\sup_{ t \in [0,T] }
	\big( 
	\E\big[
	\|
	Y_t 
	\|^p
	\big]
	\big)^{ \nicefrac{ 1 }{ p } }
	\right]
	\right]
	+
	L_1 
	\varpi_p
	h^{ \nicefrac{ 1 }{ 2 } }
	T^{ \nicefrac{ 1 }{ p } - \nicefrac{ 1 }{ 2 } }
	\Bigg]
	\\ & =
	\varepsilon_0 
	\,
	\big(
	1 
	+
	\E\big[ 
	\| X_T \|^{ \varsigma_0 }
	\big]
	\big)
	\\ & 
	+
	L_0 \,
	2^{ \max\{ \ell - 1, 0 \} }
	\, 
	e^{
		\left[ 
		L_1 
		+
		1 - \nicefrac{ 1 }{ p }
		\right]
		T
	}
	\left[
	1 
	+
	\left(
	\E\big[
	\| X_T \|^{ \ell q }
	\big]
	\right)^{ \nicefrac{ 1 }{ q } }
	+
	\left(
	\E\big[
	\| Y_T \|^{ \ell q }
	\big]
	\right)^{ \nicefrac{ 1 }{ q } }
	\right]
	\\ & \cdot 
	\Bigg[
	\left\| \xi - \phi_2( \xi ) \right\|
	+
	\varepsilon_1
	T^{ \nicefrac{ 1 }{ p } }
	\left[
	1
	+
	h
	L_1
	\right]
	\left[
	1
	+
	\sup_{ t \in [0,T] }
	\big(
	\E\big[
	\| Y_t \|^{ p \varsigma_1 }
	\big]
	\big)^{ \nicefrac{ 1 }{ p } }
	\right]
	\\ &
	+
	h
	T^{ \nicefrac{ 1 }{ p } }
	L_1 
	\!
	\left[
	\| f_1( 0 ) \|
	+
	L_1 
	\left[
	\sup_{ t \in [0,T] }
	\big( 
	\E\big[
	\|
	Y_t 
	\|^p
	\big]
	\big)^{ \nicefrac{ 1 }{ p } }
	\right]
	\right]
	+
	L_1 
	\varpi_p
	h^{ \nicefrac{ 1 }{ 2 } }
	T^{ \nicefrac{ 1 }{ p } - \nicefrac{ 1 }{ 2 } }
	\Bigg]
	.
	\end{split}
	\end{equation} 
	Next observe that 
	Lemma~\ref{lem:sde-lp-bound} 
	and \eqref{eq:f1_linear_growth} demonstrate that 
	for all $ r \in [1,\infty) $, $ t \in [0,T] $ it holds that
	\begin{equation}
	\begin{split}
	\sup_{ t \in [0,T] }
	\big(
	\E\big[ 
	\| Y_t \|^r
	\big]
	\big)^{ \nicefrac{ 1 }{ r } }
	& \leq
	\left(
	\| \phi_2( \xi ) \| 
	+ 
	C T 
	+ 
	\big( 
	\E\big[ \| B W_T \|^r \big]
	\big)^{ \nicefrac{1}{r} }
	\right)
	e^{ c T } 
	\\ & =
	\left(
	\| \phi_2( \xi ) \| 
	+ 
	C T 
	+ 
	\varpi_r 
	\right)
	e^{ c T } 
	\end{split}
	\end{equation}
	and 
	\begin{equation}
	\begin{split}
	\sup_{ t \in [0,T] }
	\big(
	\E\big[ 
	\| X_t \|^r
	\big]
	\big)^{ \nicefrac{ 1 }{ r } }
	& \leq
	\left(
	\| \xi \| 
	+ 
	\| f_1(0) \| T 
	+ 
	\big( 
	\E\big[ \| B W_T \|^r \big]
	\big)^{ \nicefrac{1}{r} }
	\right)
	e^{ L_1 T }
	\\ & =
	\left(
	\| \xi \| 
	+ 
	\| f_1(0) \| T 
	+ 
	\varpi_r 
	\right)
	e^{ L_1 T }
	.
	\end{split}
	\end{equation}
	Combining this with \eqref{eq:main_estimate_proof} shows that
	\begin{align*}
	&
	\big| \E\big[ f_0( X_T ) \big] - \E\big[ \phi_0( Y_T ) \big] \big|
	\\ & \leq
	\varepsilon_0 
	\,
	\Big(
	1 
	+
	\big[ 
	\| \xi \| 
	+ 
	\| f_1(0) \| T 
	+ 
	\varpi_{ \max\{ \varsigma_0, 1 \} } 
	\big]^{ \varsigma_0 }
	e^{ \varsigma_0 L_1 T }
	\Big)
	+
	L_0 
	\,
	2^{ \max\{ \ell - 1, 0 \} }
	\,
	e^{
		\left[ 
		L_1 
		+
		1 - \nicefrac{ 1 }{ p }
		\right]
		T
	}
	\\ &
	\cdot 
	\Big(
	1 
	+
	\big[ 
	\| \xi \| 
	+ 
	\| f_1(0) \| T 
	+ 
	\varpi_{ \max\{ \ell q, 1 \} } 
	\big]^{ \ell }
	\,
	e^{ \ell L_1 T }
	+
	\big[ 
	\| \phi_2( \xi ) \| 
	+ 
	C T 
	+ 
	\varpi_{ \max\{ \ell q, 1 \} }
	\big]^{ \ell }
	\,
	e^{ \ell c T }
	\Big)
	\\ & \cdot 
	\Bigg[
	\left\| \xi - \phi_2( \xi ) \right\|
	+
	\varepsilon_1
	T^{ \nicefrac{ 1 }{ p } }
	[
	1
	+
	h
	L_1
	]
	\Big(
	1
	+
	\big[ 
	\| \phi_2( \xi ) \| 
	+ 
	C T 
	+ 
	\varpi_{ \max\{ \varsigma_1 p, 1 \} } 
	\big]^{ \varsigma_1 }
	e^{ \varsigma_1 c T }
	\Big)
	\\ &
	+
	h
	T^{ \nicefrac{ 1 }{ p } }
	L_1 
	\!
	\left(
	\| f_1( 0 ) \|
	+
	L_1 
	\big[ 
	\| \phi_2( \xi ) \| 
	+ 
	C T 
	+ 
	\varpi_p 
	\big]
	e^{ c T }
	\right)
	+
	L_1 
	\varpi_p
	h^{ \nicefrac{ 1 }{ 2 } }
	T^{ \nicefrac{ 1 }{ p } - \nicefrac{ 1 }{ 2 } }
	\Bigg] \numberthis
	.
	\end{align*}
	Hence, we obtain that
	\begin{equation}
	\begin{split}
	&
	\big| \E\big[ f_0( X_T ) \big] - \E\big[ \phi_0( Y_T ) \big] \big|
	\\ & \leq
	\varepsilon_0 
	\,
	\Big(
	1 
	+
	\big[ 
	\| \xi \| 
	+ 
	\| f_1(0) \| T 
	+ 
	\varpi_{ \max\{ \varsigma_0, 1 \} } 
	\big]^{ \varsigma_0 }
	\,
	e^{ \varsigma_0 L_1 T }
	\Big)
	\\ &
	+
	L_0 
	\,
	2^{ \max\{ \ell - 1, 0 \} }
	\,
	e^{
		\left[ 
		\max\{ 
		\ell L_1, 
		\ell c
		\} 
		+
		c
		\max\{ 
		\varsigma_1 ,
		1
		\}
		+
		L_1 
		+
		1 - \nicefrac{ 1 }{ p }
		\right]
		T
	}
	\max\{ 1, T^{ \nicefrac{ 1 }{ p } } \}
	\\ & 
	\cdot
	\Big(
	1 
	+
	\big[ 
	\| \xi \| 
	+ 
	\| f_1(0) \| T 
	+ 
	\varpi_{ \max\{ \ell q, 1 \} } 
	\big]^{ \ell }
	+
	\big[ 
	\| \phi_2( \xi ) \| 
	+ 
	C T 
	+ 
	\varpi_{ \max\{ \ell q, 1 \} } 
	\big]^{ \ell }
	\Big)
	\\ & \cdot 
	\Bigg[
	\left\| \xi - \phi_2( \xi ) \right\|
	+
	\varepsilon_1
	[
	1
	+
	h
	L_1
	]
	\Big(
	1
	+
	\big[ 
	\| \phi_2( \xi ) \| 
	+ 
	C T 
	+ 
	\varpi_{ \max\{ \varsigma_1 p, 1 \} } 
	\big]^{ \varsigma_1 }
	\Big)
	\\ &
	+
	h
	L_1 
	\Big(
	\| f_1( 0 ) \|
	+
	L_1 
	\big[ 
	\| \phi_2( \xi ) \| 
	+ 
	C T 
	+ 
	\varpi_p 
	\big]
	\Big)
	+
	( h / T )^{ \nicefrac{ 1 }{ 2 } }
	L_1 
	\varpi_p
	\Bigg]
	.\end{split}
	\end{equation}
	This implies that
	\begin{align*}
	&
	\big| \E\big[ f_0( X_T ) \big] - \E\big[ \phi_0( Y_T ) \big] \big|
	\\ & \leq
	\varepsilon_0 
	\,
	\Big(
	1 
	+
	\big[ 
	\| \xi \| 
	+ 
	\| f_1(0) \| T 
	+ 
	\varpi_{ \max\{ \varsigma_0, 1 \} } 
	\big]^{ \varsigma_0 }
	\,
	e^{ \varsigma_0 L_1 T }
	\Big)
	\\ &
	+
	L_0 
	\,
	2^{ \max\{ \ell - 1, 0 \} }
	\,
	e^{
		\left[ 
		\max\{ 
		\ell L_1, 
		\ell c
		\} 
		+
		c
		\max\{ 
		\varsigma_1 ,
		1
		\}
		+
		L_1 
		+
		1
		\right]
		T
	} \numberthis
	\\ & 
	\cdot
	\Big(
	1 
	+
	\big[ 
	\| \xi \| 
	+ 
	\| f_1(0) \| T 
	+ 
	\varpi_{ \max\{ \ell q, 1 \} } 
	\big]^{ \ell }
	+
	\big[ 
	\| \xi \|
	+
	\varepsilon_2 
	( 1 + \| \xi \|^{ \varsigma_2 } )
	+ 
	C T 
	+ 
	\varpi_{ \max\{ \ell q, 1 \} } 
	\big]^{ \ell }
	\Big)
	\\ & \cdot 
	\Bigg[
	\varepsilon_2 
	( 1 + \| \xi \|^{ \varsigma_2 } )
	+
	\varepsilon_1
	[
	1
	+
	T
	L_1
	]
	\Big(
	1
	+
	\big[ 
	\| \xi \|
	+
	\varepsilon_2 
	( 1 + \| \xi \|^{ \varsigma_2 } )
	+ 
	C T 
	+ 
	\varpi_{ \max\{ \varsigma_1 p, 1 \} } 
	\big]^{ \varsigma_1 }
	\Big)
	\\ &
	+
	( h / T )^{ \nicefrac{ 1 }{ 2 } } 
	T
	L_1 
	\Big(
	\| f_1( 0 ) \|
	+
	L_1 
	\big[ 
	\| \phi_2( \xi ) \| 
	+ 
	C T 
	+ 
	\varpi_p 
	\big]
	\Big)
	+
	( h / T )^{ \nicefrac{ 1 }{ 2 } }
	L_1 
	\varpi_p
	\Bigg]
	.
	\end{align*}
	Therefore, we obtain that
	\begin{align*}
	&
	\big| \E\big[ f_0( X_T ) \big] - \E\big[ \phi_0( Y_T ) \big] \big|  \numberthis
	\\ & \leq
	\varepsilon_0 
	\,
	\Big(
	1 
	+
	\big[ 
	\| \xi \| 
	+ 
	\| f_1(0) \| T 
	+ 
	\varpi_{ \max\{ \varsigma_0, 1 \} } 
	\big]^{ \varsigma_0 }
	\,
	e^{ \varsigma_0 L_1 T }
	\Big)
	\\ &
	+
	L_0 
	\,
	2^{ \max\{ \ell - 1, 0 \} }
	\,
	e^{
		\left[ 
		\max\{ 
		\ell L_1, 
		\ell c
		\} 
		+
		c
		\max\{ 
		\varsigma_1 ,
		1
		\}
		+
		L_1 
		+
		1
		\right]
		T
	}
	\\ & 
	\cdot
	\Big(
	1 
	+
	\big[ 
	\| \xi \| 
	+ 
	\| f_1(0) \| T 
	+ 
	\varpi_{ \max\{ \ell q, 1 \} } 
	\big]^{ \ell }
	+
	\big[ 
	\| \xi \|
	+
	\varepsilon_2 
	( 1 + \| \xi \|^{ \varsigma_2 } )
	+ 
	C T 
	+ 
	\varpi_{ \max\{ \ell q, 1 \} } 
	\big]^{ \ell }
	\Big)
	\\ & \cdot 
	\Bigg[
	\left[
	\varepsilon_1
	+
	\varepsilon_2
	\right]
	\max\{ 1, T \}
	[
	1
	+
	L_1
	]
	\\ & \cdot
	\Big(
	1
	+
	\big[ 
	\| \xi \|
	+
	\max\{ 1, \varepsilon_2 \}
	( 1 + \| \xi \|^{ \varsigma_2 } )
	+ 
	C T 
	+ 
	\varpi_{ \max\{ \varsigma_1 p, 1 \} } 
	\big]^{ \max\{ 1, \varsigma_1 \} }
	\Big)
	\\ &
	+
	( h / T )^{ \nicefrac{ 1 }{ 2 } }
	\max\{ 1, T \}
	[ 1 + L_1 ]
	\Big(
	\| f_1( 0 ) \| 
	+
	L_1 
	\big[ 
	\| \xi \|
	+
	\varepsilon_2 
	( 1 + \| \xi \|^{ \varsigma_2 } )
	+ 
	C T 
	+ 
	\varpi_p 
	\big]
	\Big)
	\Bigg]
	.
	\end{align*}
	This and the fact that $\forall \, x \in [0, \infty) \colon \max\{x, 1\} \leq x+1 \leq e^x$ demonstrate that
	\begin{align*}
	&
	\big| \E\big[ f_0( X_T ) \big] - \E\big[ \phi_0( Y_T ) \big] \big|
	\\ & \leq
	\varepsilon_0 
	\,
	\Big(
	1 
	+
	\big[ 
	\| \xi \| 
	+ 
	\| f_1(0) \| T 
	+ 
	\varpi_{ \max\{ \varsigma_0, 1 \} } 
	\big]^{ \varsigma_0 }
	\,
	e^{ \varsigma_0 L_1 T }
	\Big)
	\\ &
	+
	L_0 
	\,
	2^{ \max\{ \ell - 1, 0 \} }
	\,
	e^{
		( L_1
		+
		\left[ 
		\max\{ 
		\ell L_1, 
		\ell c
		\} 
		+
		c
		\max\{ 
		\varsigma_1 ,
		1
		\}
		+
		L_1 
		+
		2
		\right]
		T
		)
	}
	\left[
	\varepsilon_1
	+
	\varepsilon_2
	+
	( h / T )^{ \nicefrac{ 1 }{ 2 } } 
	\right]  \numberthis
	\\ & 
	\cdot
	\Big(
	1 
	+
	\big[ 
	\| \xi \| 
	+ 
	\| f_1(0) \| T 
	+ 
	\varpi_{ \max\{ \ell q, 1 \} } 
	\big]^{ \ell }
	+
	\big[ 
	\| \xi \|
	+
	\varepsilon_2 
	( 1 + \| \xi \|^{ \varsigma_2 } )
	+ 
	C T 
	+ 
	\varpi_{ \max\{ \ell q, 1 \} } 
	\big]^{ \ell }
	\Big)
	\\ & \cdot 
	\Big[
	\max\{ 1, \| f_1(0) \| \}
	\\ &
	+
	\max\{ 1, L_1 \}
	\big[ 
	\| \xi \|
	+
	\max\{ 1, \varepsilon_2 \}
	( 1 + \| \xi \|^{ \varsigma_2 } )
	+ 
	C T 
	+ 
	\varpi_{ \max\{ p, \varsigma_1 p \} } 
	\big]^{ \max\{ 1, \varsigma_1 \} }
	\Big]
	.
	\end{align*}
	Hence, we obtain that
	\begin{align*}
	&
	\big| \E\big[ f_0( X_T ) \big] - \E\big[ \phi_0( Y_T ) \big] \big|
	\\ & \leq
	2^{ \max\{ \ell - 1, 0 \} }
	\,
	e^{
		( L_1
		+
		\left[ 
		\max\{ 
		\ell L_1, 
		\ell c
		\} 
		+
		c
		\max\{ 
		\varsigma_1 ,
		1
		\}
		+
		\max\{ \varsigma_0, 1 \} L_1 
		+
		2
		\right]
		T
		)
	}
	\left[
	\varepsilon_0
	+
	\varepsilon_1
	+
	\varepsilon_2
	+
	( h / T )^{ \nicefrac{ 1 }{ 2 } } 
	\right]
	\\ & 
	\cdot
	\Big(
	1 
	+
	\big[ 
	\| \xi \| 
	+ 
	\| f_1(0) \| T 
	+ 
	\varpi_{ \max\{ \ell q, 1 \} } 
	\big]^{ \ell }
	+
	\big[ 
	\| \xi \|
	+
	\varepsilon_2 
	( 1 + \| \xi \|^{ \varsigma_2 } )
	+ 
	C T 
	+ 
	\varpi_{ \max\{ \ell q, 1 \} } 
	\big]^{ \ell }
	\Big)
	\\ & \cdot 
	\max\{ 1, L_0 \}
	\Big[
	\max\{ 1, \| f_1(0) \| \}
	+
	\max\{ 1, L_1 \}
	\big[ 
	\| \xi \|
	+
	\max\{ 1, \varepsilon_2 \}
	( 1 + \| \xi \|^{ \varsigma_2 } )
	\\ &
	+ 
	\max\{ C, \| f_1( 0 ) \| \} T 
	+ 
	\varpi_{ \max\{ p, \varsigma_1 p, \varsigma_0 \} } 
	\big]^{ \max\{ 1, \varsigma_0, \varsigma_1 \} }
	\Big]  \numberthis
	.
	\end{align*}
	This and the fact that $\forall \, x \in [0, \infty) \colon \max\{x, 1\} \leq x+1 \leq e^x$ show that
	\begin{align*}
	&
	\big| \E\big[ f_0( X_T ) \big] - \E\big[ \phi_0( Y_T ) \big] \big|
	\\ & \leq
	2^{ \max\{ \ell, 1 \} }
	\,
	e^{
		( 
		2 L_1 +
		\left[ 
		\ell
		\max\{ 
		L_1, 
		c
		\} 
		+
		c
		\max\{ 
		\varsigma_1 ,
		1
		\}
		+
		L_1 \max\{ \varsigma_0, 1 \}
		+
		2
		\right]
		T
		)
	}
	\left[
	\varepsilon_0
	+
	\varepsilon_1
	+
	\varepsilon_2
	+
	( h / T )^{ \nicefrac{ 1 }{ 2 } } 
	\right]
	\\ & 
	\cdot
	\Big(
	1 
	+
	\big[ 
	\| \xi \| 
	+ 
	\| f_1(0) \| T 
	+ 
	\varpi_{ \max\{ \ell q, 1 \} } 
	\big]^{ \ell }
	+
	\big[ 
	\| \xi \|
	+
	\varepsilon_2 
	( 1 + \| \xi \|^{ \varsigma_2 } )
	+ 
	C T 
	+ 
	\varpi_{ \max\{ \ell q, 1 \} } 
	\big]^{ \ell }
	\Big)
	\\ & \cdot 
	\max\{ 1, L_0 \}
	\big[ 
	\| \xi \|
	+
	\max\{ 1, \varepsilon_2 \}
	( 1 + \| \xi \|^{ \varsigma_2 } )
	\\ &
	+ 
	\max\{ 1, C, \| f_1( 0 ) \| \} \max\{ 1, T \} 
	+ 
	\varpi_{ \max\{ \varsigma_0, \varsigma_1 p, p \} } 
	\big]^{ \max\{ 1, \varsigma_0, \varsigma_1 \} }  \numberthis
	.
	\end{align*}
	Therefore, we obtain that
	\begin{equation}
	\begin{split}
	&
	\big| \E\big[ f_0( X_T ) \big] - \E\big[ \phi_0( Y_T ) \big] \big|
	\\ & \leq
	e^{
		( 
		\ell + 3 + 2 L_1 +
		\left[ 
		\ell
		\max\{ 
		L_1, 
		c
		\} 
		+
		c
		\max\{ 
		\varsigma_1 ,
		1
		\}
		+
		L_1 \max\{ \varsigma_0, 1 \}
		+
		2
		\right]
		T
		)
	}
	\\ & \cdot 
	\left[
	\varepsilon_0
	+
	\varepsilon_1
	+
	\varepsilon_2
	+
	( h / T )^{ \nicefrac{ 1 }{ 2 } } 
	\right]
	\max\{ 1, L_0 \}
	\big[ 
	\| \xi \|
	+
	\max\{ 1, \varepsilon_2 \}
	( 1 + \| \xi \|^{ \varsigma_2 } )
	\\ &
	+ 
	\max\{ 1, C, \| f_1( 0 ) \| \} \max\{ 1, T \} 
	+ 
	\varpi_{ \max\{ \varsigma_0, \varsigma_1 p, p, \ell q \} } 
	\big]^{ \max\{ 1, \varsigma_0, \varsigma_1 \} + \ell }
	.
	\end{split}
	\end{equation}
	The proof of Proposition~\ref{prop:perturbation_PDE_2} is thus completed.
\end{proof}

\section{Deep artificial neural network (DNN) calculus }
\label{sec:dnn:calculus}

In Section~\ref{sec:DNN_PDEs} below we establish the existence of a DNN approximating the solution of the PDE without the curse of dimensionality. To demonstrate the existence of such a DNN, we need a few properties about representation flexibilities of DNNs, which we establish in this section. In particular, we state in the elementary and essentially well-known result in Lemma~\ref{lem:sum:ANN} in Subsection~\ref{sec:sum:ANN} below that every linear combination of realizations of DNNs with the same architecture is again a realization of a suitable DNN.  Similar results to Lemma~\ref{lem:sum:ANN} can, e.g., be found in~Yarotsky~\cite{yarotsky2017error}.

Moreover, in Proposition~\ref{prop:composition:ANN} in Subsection~\ref{sec:composition:ANN} below we demonstrate under suitable hypotheses that the composition of the realizations of two DNNs is again a realization of a suitable DNN and the number of parameters of this suitable DNN grows at most additively in the number of parameters of the composed DNNs.
For the construction of this suitable DNN in Proposition~\ref{prop:composition:ANN} we plug an artificial identity in between the two DNNs and for this we employ in Proposition~\ref{prop:composition:ANN} the hypothesis that the identity can within the class of considered fully-connected neural networks (see \eqref{eq:composition:ANN_class}--\eqref{eq:composition:ANN_realization} in Proposition~\ref{prop:composition:ANN} below) be described by a suitable flat artificial neural network.
In Proposition~\ref{prop:composition:ANN}
the tuples $ \phi_1 $ and $ \phi_2 $ represent the DNNs which we intend to compose
(where the realization of $ \phi_1 $ is a function from $ \R^{ d_2 } $ to $ \R^{ d_3 } $
and where the realization of $ \phi_2 $ is a function from $ \R^{ d_1 } $ to $ \R^{ d_2 } $),
the tuple $ \mathbb{I} $ represents the artificial neural network
which describes the identity on $ \R^{ d_2 } $,
and the tuple $ \psi $ represents the DNN whose realization coincides with the composition of the realizations of $ \phi_1 $ and $ \phi_2 $
(the realization of $ \psi $ is thus a function from $ \R^{ d_1 } $ to $ \R^{ d_3 } $).
The hypothesis of the existence of the artificial neural network $ \mathbb{I} $ can, roughly speaking, be viewed as a hypothesis on the activation function $ \mathbf{a} \colon \R \to \R $ used in Proposition~\ref{prop:composition:ANN}.
Proposition~\ref{prop:composition:ANN}, loosely speaking, then asserts that the number of parameters of $ \psi $ can up to a constant be bounded by the sum of the number of parameters of $ \phi_1 $ and of the number of parameters of $ \phi_2 $.
A straightforward DNN construction of the composition of $ \phi_1 $ and $ \phi_2 $
(without artificially plugging the identity on $ \R^{ d_2 } $ in between $ \phi_1 $ and $ \phi_2 $) would possibly result in a DNN whose number of parameters is essentially equal to the product of the number of parameters of $ \phi_1 $ and of the number of parameters of $ \phi_2 $.
Such a construction, in turn, would in our proof of the main result of this article (Theorem~\ref{thm:PDE_approx_Lp} below)
not allow us to conclude that DNNs do indeed overcome the curse of dimensionality in the numerical approximation of the considered PDEs
(see~\eqref{eq:bar:psi} in the proof of Proposition~\ref{prop:PDE_approx_Lp} for details).
Moreover, in Proposition~\ref{prop:sum:comp:ANN} in Subsection~\ref{sec:composition:ANN} below
we establish under similar hypotheses as in Proposition~\ref{prop:composition:ANN}
a result similar to Proposition~\ref{prop:composition:ANN}
which is tailor-made to the DNNs which we design in the proof of our main result in Theorem~\ref{thm:PDE_approx_Lp} below.
In particular, \eqref{eq:sum:comp:ann} in Proposition~\ref{prop:sum:comp:ANN} is tailor-made to construct a DNN which is based on an Euler discretization of a (stochastic) differential equation.
We refer to~\eqref{eq:Y_processes} and~\eqref{eq:sum:comp:ANN} in the proof of Proposition~\ref{prop:PDE_approx_Lp} below for further details. 

To apply Proposition~\ref{prop:composition:ANN} and Proposition~\ref{prop:sum:comp:ANN}, respectively, we need to verify that the class of considered DNNs does indeed enjoy the property to be able to represent the identity on $ \R^{ d_2 } $.
Fortunately, DNNs with the rectifier function as the activation function do indeed admit this property.
This fact is verified in the elementary result in Lemma~\ref{lem:identity} in Subsection~\ref{sec:rep_identity} below.
In particular, Lemma~\ref{lem:identity} shows for every $ d \in \N $ that the $ d $-dimensional identity can be explicitly represented by a suitable rectified flat  artificial neural network
(with one hidden layer with $ 2d $ neurons and the rectifier function as the activation function in front of the $ 2d $-dimensional hidden layer).


\subsection{Sums of DNNs with the same architecture}
\label{sec:sum:ANN}

\begin{lemma}
\label{lem:sum:ANN}
Let 
$ \mathbf{A}_n \colon \R^n \to \R^n $, 
$ n \in \N $, 
and 
$ \mathbf{a} \colon \R \to \R$
be continuous functions 
which satisfy 
for all 
$
n \in \N
$,
$ x = ( x_1, \dots, x_n ) \in \R^n $
that
$ 
\mathbf{A}_n(x)
=
( \mathbf{a}(x_1), \ldots, \mathbf{a}(x_n) )
$,
let 
\begin{equation}
\mathcal{N}
=
\cup_{ L \in \{ 2, 3, 4, \dots \} }
\cup_{ ( l_0, l_1, \ldots, l_L ) \in \N^{ L + 1 } }
(
\times_{ n = 1 }^L 
(
\R^{ l_n \times l_{ n - 1 } } \times \R^{ l_n } 
)
)
,
\end{equation}
let 
$
\mathcal{P}
\colon \mathcal{N} \to \N
$ 
and
$
\mathcal{R} \colon 
\mathcal{N} 
\to 
\cup_{ k, l \in \N } C( \R^k, \R^l )
$
be the functions which satisfy 
for all 
$ L \in \{ 2, 3, 4, \dots \} $, 
$ l_0, l_1, \ldots, l_L \in \N $, 
$ 
\Phi = ((W_1, B_1), \ldots, (W_L, B_L)) \in 
( \times_{ n = 1 }^L (\R^{ l_n \times l_{n-1} } \times \R^{ l_n } ) )
$,
$ x_0 \in \R^{l_0} $, 
$ \ldots $, 
$ x_{ L - 1 } \in \R^{ l_{ L - 1 } } $ 
with 
$ 
\forall \, n \in \N \cap [1,L) \colon 
x_n = \mathbf{A}_{ l_n }( W_n x_{ n - 1 } + B_n )
$
that 
$
\mathcal{P}( \Phi )
=
\textstyle
\sum\nolimits_{
n = 1
}^L
l_n ( l_{ n - 1 } + 1 )
$,
$
\mathcal{R}(\Phi) \in C( \R^{ l_0 } , \R^{ l_L } )
$,
and
\begin{equation}
( \mathcal{R} \Phi )( x_0 ) = W_L x_{L-1} + B_L ,
\end{equation}
let $ \mathbb{L} \in \{ 2, 3, 4, \dots \} $, $ M $, $\mathfrak{L}_0, \mathfrak{L}_1, \ldots, \mathfrak{L}_{\mathbb{L}} \in \N $, $h_1, h_2, \ldots, h_M \in \R$,
and let 
$
( \phi_{ m } )_{ 
m \in \{1, 2, \ldots, M\}
} $ $
\subseteq ( \times_{ n = 1 }^{\mathbb{L}} (\R^{ \mathfrak{L}_n \times \mathfrak{L}_{n-1} } \times \R^{ \mathfrak{L}_n } ) )
$. Then there exists $\psi \in \mathcal{N}$ such that for all $x \in \R^{\mathfrak{L}_0}$ it holds that  $\mathcal{R}(\psi) \in C(\R^{\mathfrak{L}_0}, \R^{\mathfrak{L}_{\mathbb{L}}})$, $\mathcal{P}(\psi) \leq M^2 \mathcal{P}(\phi_1)$, and
\begin{equation}
\label{eq:sum:ann}
(\mathcal{R} \psi)(x) = \sum_{m=1}^M h_m (\mathcal{R} \phi_m)(x).
\end{equation}
\end{lemma}
\begin{proof}[Proof of Lemma~\ref{lem:sum:ANN}]
Throughout this proof let 
$((W_{m,1}, B_{m,1}), \ldots, (W_{m,\mathbb{L}}, B_{m,\mathbb{L}})) \in ( \times_{ n = 1 }^{\mathbb{L}} (\R^{ \mathfrak{L}_n \times \mathfrak{L}_{n-1} } \times \R^{ \mathfrak{L}_n } ) )$, 
$m \in \{1, 2, \ldots, M\}$, 
satisfy for all 
$i \in \{1, 2, \ldots, M\}$
that 
$\phi_i = ((W_{i,1}, B_{i,1}), \ldots, (W_{i,\mathbb{L}}, B_{i,\mathbb{L}}))$,
let
$(l_0, l_1, \ldots, l_{\mathbb{L}}) \in \N^{\mathbb{L}+1}$ 
satisfy for all 
$i \in \{1, 2, \ldots, \mathbb{L}-1\}$ 
that 
$l_0 = \mathfrak{L}_0$, 
$l_i = M \mathfrak{L}_i$, 
and 
$l_{\mathbb{L}} = \mathfrak{L}_{\mathbb{L}}$, 
let 
$((W_1, B_1), \ldots, (W_{\mathbb{L}}, B_{\mathbb{L}})) \in 
( \times_{ n = 1 }^{\mathbb{L}} (\R^{ l_n \times l_{n-1} } \times \R^{ l_n } ) )$ 
satisfy that
\begin{equation}
W_1 = 
\begin{pmatrix}
W_{1,1}\\
W_{2,1}\\
\vdots \\
W_{M,1}
\end{pmatrix}
\in \R^{(M \mathfrak{L}_1) \times \mathfrak{L}_0} = \R^{l_1 \times l_0}, \qquad B_1 = 
\begin{pmatrix}
B_{1,1}\\
B_{2,1}\\
\vdots \\
B_{M,1}
\end{pmatrix}
\in \R^{(M \mathfrak{L}_1)} = \R^{l_1},
\end{equation}
\begin{equation}
W_{\mathbb{L}} = \Big( h_1 W_{1, \mathbb{L}} \quad h_2 W_{2, \mathbb{L}} \quad \cdots \quad h_M W_{M, \mathbb{L}} \Big) \in \R^{\mathfrak{L}_{\mathbb{L}} \times (M \mathfrak{L}_{\mathbb{L}-1})} = \R^{l_{\mathbb{L}} \times l_{\mathbb{L}-1}},
\end{equation}
\begin{equation}
\text{and} \qquad B_{\mathbb{L}} = \sum_{m=1}^M h_m B_{m, \mathbb{L}} \in \R^{\mathfrak{L}_{\mathbb{L}}} = \R^{l_{\mathbb{L}}},
\end{equation}
assume for all $i \in \{ 2, 3, 4, \ldots \} \cap [ 0,  \mathbb{L}-1 ] $ that
\begin{equation}
W_i =
\begin{pmatrix}
W_{1,i}&		0&		\cdots& 	0\\
0&		W_{2,i}&	 \cdots& 	\vdots\\
\vdots& 	\vdots&		\ddots&		0\\
0&		\cdots& 		0& 		W_{M,i}
\end{pmatrix}
\in \R^{(M \mathfrak{L}_i) \times (M\mathfrak{L}_{i-1})} = \R^{l_i \times l_{i-1}}
\end{equation}
\begin{equation}
\text{and} \qquad B_i = 
\begin{pmatrix}
B_{1,i}\\
B_{2,i}\\
\vdots \\
B_{M,i}
\end{pmatrix}
\in \R^{(M \mathfrak{L}_i)} = \R^{l_i},
\end{equation}
and let $\psi = ((W_1, B_1), \ldots, (W_{\mathbb{L}}, B_{\mathbb{L}})) \in \mathcal{N}$.
Note that for all $x \in \R^{l_0}$ it holds that
\begin{equation}
\label{eq:sum:first}
W_1 x + B_1 = 
\begin{pmatrix}
W_{1,1}x +B_{1,1}\\
W_{2,1}x  + B_{2,1}\\
\vdots \\
W_{M,1}x + B_{M,1}
\end{pmatrix}.
\end{equation}
Moreover, observe that for all $i \in \N \cap [ 0,  \mathbb{L}-2 ] $, $x_1, x_2, \ldots, x_M \in \R^{l_i}$ it holds that  
\begin{equation}
\label{eq:sum:middle}
\begin{split}
W_{i+1} 
\begin{pmatrix}
x_1 \\
x_2\\
\vdots\\
x_M
\end{pmatrix} 
+ B_{i+1} &= 
\begin{pmatrix}
W_{1,i+1}&		0&		\cdots& 	0\\
0&		W_{2,i+1}&	 \cdots& 	\vdots\\
\vdots& 	\vdots&		\ddots&		0\\
0&		\cdots& 		0& 		W_{M,i+1}
\end{pmatrix}
\begin{pmatrix}
x_1 \\
x_2\\
\vdots\\
x_M
\end{pmatrix} +
\begin{pmatrix}
B_{1,i+1}\\
B_{2,i+1}\\
\vdots \\
B_{M,i+1}
\end{pmatrix}\\
& = 
\begin{pmatrix}
W_{1,i+1}x_1 +B_{1,i+1}\\
W_{2,i+1}x_2  + B_{2,i+1}\\
\vdots \\
W_{M,i+1}x_M + B_{M,i+1}
\end{pmatrix}.
\end{split}
\end{equation}
Next note that for all $x_1, x_2, \ldots, x_M \in \R^{l_{\mathbb{L}-1}}$ it holds that 
\begin{equation}
\begin{split}
W_{\mathbb{L}} 
\begin{pmatrix}
x_1 \\
x_2\\
\vdots\\
x_M
\end{pmatrix} 
+ B_{\mathbb{L}}  
&= \Big( h_1 W_{1, \mathbb{L}} \quad h_2 W_{2, \mathbb{L}} \quad \cdots \quad h_M W_{M, \mathbb{L}} \Big)
\begin{pmatrix}
x_1 \\
x_2\\
\vdots\\
x_M
\end{pmatrix}
+ \sum_{m=1}^M h_m B_{m, \mathbb{L}}\\
& = \Biggl[ \sum_{m=1}^M h_m W_{m, \mathbb{L}} x_m \Biggr] + \Biggl[ \sum_{m=1}^M h_m B_{m, \mathbb{L}} \Biggr] = \sum_{m=1}^M h_m  \big( W_{m, \mathbb{L}} x_m + B_{m, \mathbb{L}} \big).
\end{split}
\end{equation}
This, \eqref{eq:sum:first}, and \eqref{eq:sum:middle} ensure that for all $x \in \R^{\mathfrak{L}_0}$ it holds that $\mathcal{R}(\psi) \in C(\R^{\mathfrak{L}_0}, \R^{\mathfrak{L}_{\mathbb{L}}})$ 
and
\begin{equation}
\label{eq:sum:psi}
(\mathcal{R} \psi)(x) = \sum_{m=1}^M h_m (\mathcal{R} \phi_m)(x).
\end{equation}
Moreover, observe that 
the assumption that
for all 
$i \in \{1, 2, \ldots, \mathbb{L}-1\}$ 
it holds that 
$l_0 = \mathfrak{L}_0$, 
$l_i = M \mathfrak{L}_i$, 
and 
$l_{\mathbb{L}} = \mathfrak{L}_{\mathbb{L}}$ assures that
\begin{equation}
\begin{split}
\mathcal{P}(\psi) &= \sum_{n=1}^{\mathbb{L}} l_n(l_{n-1} +1) 
= l_1(l_0+1) + l_{\mathbb{L}}(l_{\mathbb{L}-1}+1) + \sum_{n=2}^{\mathbb{L}-1} l_n(l_{n-1} +1)\\
& = M \mathfrak{L}_1(\mathfrak{L}_0 +1) + \mathfrak{L}_{\mathbb{L}} (M \mathfrak{L}_{\mathbb{L}-1} + 1) + \sum_{n=2}^{\mathbb{L}-1} M \mathfrak{L}_n(M \mathfrak{L}_{n-1} +1)\\
& \leq M^2 \left[ \sum_{n=1}^{\mathbb{L}} \mathfrak{L}_n( \mathfrak{L}_{n-1} +1) \right] = M^2 \mathcal{P}(\phi_1).
\end{split}
\end{equation}
Combining this with \eqref{eq:sum:psi} establishes \eqref{eq:sum:ann}.
The proof of 
Lemma~\ref{lem:sum:ANN} is thus completed.
\end{proof}

\subsection{Compositions of DNNs involving artificial identities}
\label{sec:composition:ANN}

\begin{prop}[Composition of neural networks]
\label{prop:composition:ANN}
Let $ d_1, d_2, d_3 \in \N $,
let 
$ \mathbf{A}_n \colon \R^n \to \R^n $, 
$ n \in \N $, 
and 
$ \mathbf{a} \colon \R \to \R$
be continuous functions 
which satisfy 
for all 
$
n \in \N
$,
$ x = ( x_1, \dots, x_n ) \in \R^n $
that
$ 
\mathbf{A}_n(x)
=
( \mathbf{a}(x_1), \ldots, \mathbf{a}(x_n) )
$,
let 
\begin{equation}
\label{eq:composition:ANN_class}
\mathcal{N}
=
\cup_{ L \in \{ 2, 3, 4, \dots \} }
\cup_{ ( l_0, l_1, \ldots, l_L ) \in \N^{ L + 1 } }
(
\times_{ n = 1 }^L 
(
\R^{ l_n \times l_{ n - 1 } } \times \R^{ l_n } 
)
)
,
\end{equation}
let 
$
\mathcal{P}
\colon \mathcal{N} \to \N
$,
$ \mathcal{L}
\colon \mathcal{N} \to \cup_{ L \in \{ 2, 3, 4, \dots \} } \N^{L+1}
$,
and
$
\mathcal{R} \colon 
\mathcal{N} 
\to 
\cup_{ k, l \in \N } C( \R^k, \R^l )
$
be the functions which satisfy 
for all 
$ L \in \{ 2, 3, 4, \dots \} $, 
$ l_0, l_1, \ldots, l_L \in \N $, 
$ 
\Phi = ((W_1, B_1), \ldots, $ $ (W_L, B_L)) \in 
( \times_{ n = 1 }^L (\R^{ l_n \times l_{n-1} } \times \R^{ l_n } ) )
$,
$ x_0 \in \R^{l_0} $, 
$ \ldots $, 
$ x_{ L - 1 } \in \R^{ l_{ L - 1 } } $ 
with 
$ 
\forall \, n \in \N \cap [1,L) \colon 
x_n = \mathbf{A}_{ l_n }( W_n x_{ n - 1 } + B_n )
$
that 
$
\mathcal{P}( \Phi )
=
\textstyle
\sum\nolimits_{
n = 1
}^L
l_n ( l_{ n - 1 } + 1 )
$,
$
\mathcal{R}(\Phi) \in C( \R^{ l_0 } , \R^{ l_L } )
$,
$
\mathcal{L}(\Phi) = ( l_0, l_1, \ldots, l_L ) $,
and
\begin{equation}
\label{eq:composition:ANN_realization}
( \mathcal{R} \Phi )( x_0 ) = W_L x_{L-1} + B_L ,
\end{equation}	
and let 
$ \phi_{ 1 }, \phi_{ 2 }, \mathbb{I} \in \mathcal{N} $,
$ L_1, L_2 \in  \{2, 3, 4 \ldots \}$,
$ \mathfrak{i},
l_{1,0}, l_{1,1}, \ldots, l_{1,L_1},
l_{2,0}, l_{2,1}, \ldots, l_{2,L_2} \in \N$
satisfy for all $x \in \R^{ d_2 } $, $i \in \{1, 2\}$
that
$
\mathcal{R}( \phi_{ 1} )
\in 
C( \R^{ d_2 }, \R^{ d_3 } )
$, $
\mathcal{R}( \phi_{ 2} )
\in 
C( \R^{d_1 }, \R^{d_2} )
$, $
\mathcal{R}( \mathbb{I} )
\in 
C( \R^{d_2}, \R^{d_2} )
$,
$ \mathcal{L}(\phi_i) = (l_{i,0}, l_{i,1}, \ldots, l_{i,L_i}) \in \N^{L_i+1}$,
$\mathcal{L}(\mathbb{I}) = (d_2, \mathfrak{i}, d_2) \in \N^3 $, and $(\mathcal{R} \, \mathbb{I})(x) = x$. Then there exists $\psi \in \mathcal{N}$ such that for all $x \in \R^{d_1} $ it holds that  $\mathcal{R}(\psi) \in C(\R^{d_1}, \R^{d_3})$, $\mathcal{L}(\psi) = (l_{2,0}, l_{2,1},$ $ \ldots, l_{2,L_2-1}, \mathfrak{i}, l_{1,1}, l_{1,2}, \ldots, l_{1, L_1}) \in \N^{L_1 +L_2+1}$, $\mathcal{P}(\psi) \leq \max\{1, 2^{-1} (d_2)^{-2} \, \mathcal{P}(\mathbb{I}) \} (\mathcal{P}(\phi_1) + \mathcal{P}(\phi_2)) $, and
\begin{equation}
\label{eq:comp:ann}
(\mathcal{R} \psi)(x)
=
(\mathcal{R} \phi_1) \bigl((\mathcal{R} \phi_2)(x) \bigr)
=
\bigl( (\mathcal{R} \phi_1) \circ (\mathcal{R} \phi_2) \bigr)(x).
\end{equation}
\end{prop}
\begin{proof}[Proof of Proposition~\ref{prop:composition:ANN}]
Throughout this proof let
$(W_{3,1}, B_{3,1}) \in \R^{\mathfrak{i} \times d_2} \times \R^{\mathfrak{i}}$,
$(W_{3,2}, $ $ B_{3,2}) \in \R^{d_2 \times \mathfrak{i}} \times \R^{d_2}$,
and
$((W_{j,1}, B_{j,1}), \ldots, (W_{j,L_j}, B_{j,L_j})) \in 
( \times_{ n = 1 }^{L_j} (\R^{ l_{j, n} \times l_{j,n-1} } \times \R^{ l_{j,n} } ) )$, $j \in \{1, 2\}$,
satisfy for all $j \in \{1, 2\}$ that
$\mathbb{I} = ((W_{3,1}, B_{3,1}), (W_{3,2}, B_{3,2}))$
and
$\phi_j =  ((W_{j,1}, B_{j,1}), \ldots,  (W_{j,L_j}, B_{j,L_j}))$,
let $L_4 = L_1 +L_2$, 
let
$l_{4,0}, l_{4,1}, \ldots, l_{4,L_4} \in \N$ 
satisfy for all 
$i \in \{0, 1, \ldots, L_2-1 \} $, 
$j \in \{1, 2, \ldots,  L_1\}$ 
that 
\begin{equation}
  l_{4,i}= l_{2,i} ,
  \qquad
  l_{4,L_2} = \mathfrak{i} ,
  \qquad
  \text{and} 
  \qquad
  l_{4,L_2 +j} = l_{1,j}
  ,
\end{equation}
let  
$((W_{4,1}, B_{4,1}), \ldots, $ $ (W_{4,L_4}, B_{4,L_4})) \in 
( \times_{ n = 1 }^{L_4} (\R^{ l_{4, n} \times l_{4,n-1} } \times \R^{ l_{4,n} } ) )$ 
satisfy for all 
$i \in \{ 1, 2, \ldots, L_2-1 \} $, 
$j \in \{ 2, 3, \ldots,  L_1\}$ 
that 
\begin{equation}
(W_{4,i}, B_{4,i}) = (W_{2,i}, B_{2,i}),
\end{equation}
\begin{equation}
(W_{4, L_2}, B_{4,L_2}) = (W_{3,1} W_{2, L_2}, W_{3,1} B_{2, L_2} + B_{3, 1}),
\end{equation}
\begin{equation}
(W_{4, L_2+1}, B_{4, L_2+1}) = (W_{1,1} W_{3, 2}, W_{1,1} B_{3, 2} + B_{1, 1}),
\end{equation}
and $(W_{4,L_2+j}, B_{4,L_2+j}) = (W_{1,j}, B_{1,j})$, 
and let 
$\psi = ((W_{4,1}, B_{4,1}), \ldots,  (W_{4,L_4}, B_{4,L_4})) \in \mathcal{N}$.
Observe that for all $x \in \R^{l_{4, L_2-1}} = \R^{l_{2, L_2-1}}$, $y \in \R^{l_{4, L_2}} = \R^{\mathfrak{i}} $ it holds that
\begin{equation}
W_{4, L_2} x + B_{4,L_2} = W_{3,1} W_{2, L_2} x + W_{3,1} B_{2, L_2} + B_{3, 1} = W_{3,1} (W_{2, L_2} x +  B_{2, L_2}) + B_{3, 1}
\end{equation}
and
\begin{equation}
W_{4, L_2+1} y + B_{4,L_2+1} = W_{1,1} W_{3, 2} y + W_{1,1} B_{3, 2} + B_{1, 1} = W_{1,1} (W_{3, 2} y + B_{3, 2}) + B_{1, 1}.
\end{equation}
This ensures that for all $x \in \R^{d_1}$ it holds that $\mathcal{R}(\psi) \in C(\R^{d_1}, \R^{d_3})$ and
\begin{equation}
\label{eq:composition}
(\mathcal{R} \psi)(x)
=
(\mathcal{R} \phi_1) \Bigl((\mathcal{R} \, \mathbb{I}) \bigl((\mathcal{R} \phi_2)(x) \bigr) \Bigr)
=
(\mathcal{R} \phi_1) \bigl((\mathcal{R} \phi_2)(x) \bigr).
\end{equation}
Moreover, note that
\begin{equation}
\begin{split}
\mathcal{P}(\psi) 
&= \sum_{n=1}^{L_4} l_{4,n} (l_{4,n-1} +1) 
= \sum_{n= 1}^{L_1 + L_2} l_{4,n} (l_{4,n-1} +1) \\
&= \Biggl[ \sum_{n=1}^{L_2-1} l_{4,n} (l_{4,n-1} +1) \Biggr] + \Biggl[ \sum_{n=L_2+2}^{L_1+L_2} l_{4,n} (l_{4,n-1} +1) \Biggr] \\
& \quad + l_{4,L_2} (l_{4,L_2-1} +1) + l_{4,L_2+1} (l_{4,L_2} +1)\\
& = \Biggl[ \sum_{n=1}^{L_2-1} l_{2,n} (l_{2,n-1} +1) \Biggr] + \Biggl[ \sum_{n=2}^{L_1} l_{4, L_2 + n} (l_{4,L_2 +n-1} +1) \Biggr] \\
& \quad + \mathfrak{i} \, (l_{2,L_2-1} +1) + l_{1,1} (\mathfrak{i} +1).
\end{split}
\end{equation}
Hence, we obtain that 
\begin{equation}
\label{eq:sum:parameters}
\begin{split}
\mathcal{P}(\psi) 
& = \Biggl[ \sum_{n=1}^{L_2-1} l_{2,n} (l_{2,n-1} +1) \Biggr] + \Biggl[ \sum_{n=2}^{L_1} l_{1,  n} (l_{1, n-1} +1) \Biggr] \\
& \quad + \mathfrak{i} \, (l_{2,L_2-1} +1) + l_{1,1} (\mathfrak{i} +1)\\
& = \Biggl[ \sum_{n=1}^{L_2-1} l_{2,n} (l_{2,n-1} +1) \Biggr] + \Biggl[ \sum_{n=2}^{L_1} l_{1,  n} (l_{1, n-1} +1) \Biggr] \\
& \quad + \frac{\mathfrak{i}}{d_2} \cdot l_{2, L_2} (l_{2,L_2-1} +1) + l_{1,1} \left(l_{1,0} \cdot \frac{\mathfrak{i}}{d_2} +1\right)\\
& \leq  \max \{1, \mathfrak{i}/d_2 \} \left( \Biggl[ \sum_{n=1}^{L_2} l_{2,n} (l_{2,n-1} +1) \Biggr] + \Biggl[ \sum_{n=1}^{L_1} l_{1,  n} (l_{1, n-1} +1) \Biggr] \right) \\
& = \max \{1, \mathfrak{i}/d_2 \} \big(\mathcal{P}(\phi_1) + \mathcal{P}(\phi_2)\big).
\end{split}
\end{equation}
Next observe that
\begin{equation}
\mathcal{P}(\mathbb{I}) = \mathfrak{i} \, (d_2+1) + d_2(\mathfrak{i} +1)
= 2 \, \mathfrak{i} \, d_2 +  \mathfrak{i} + d_2 
> 2 \, \mathfrak{i} \, d_2.
\end{equation}
This and \eqref{eq:sum:parameters} ensure that
\begin{equation}
\mathcal{P}(\psi)  \leq \max \{1, 2^{-1} (d_2)^{-2} \, \mathcal{P}(\mathbb{I}) \} \big(\mathcal{P}(\phi_1) + \mathcal{P}(\phi_2)\big).
\end{equation} 
Combining this with \eqref{eq:composition} establishes \eqref{eq:comp:ann}. The proof of Proposition~\ref{prop:composition:ANN} is thus completed.
\end{proof}

\begin{prop}
\label{prop:sum:comp:ANN}
Let $ d \in \N $, 
let 
$ \mathbf{A}_n \colon \R^n \to \R^n $, 
$ n \in \N $, 
and 
$ \mathbf{a} \colon \R \to \R$
be continuous functions 
which satisfy 
for all 
$
n \in \N
$,
$ x = ( x_1, \dots, x_n ) \in \R^n $
that
$ 
\mathbf{A}_n(x)
=
( \mathbf{a}(x_1), \ldots, \mathbf{a}(x_n) )
$,
let 
\begin{equation}
\mathcal{N}
=
\cup_{ L \in \{ 2, 3, 4, \dots \} }
\cup_{ ( l_0, l_1, \ldots, l_L ) \in \N^{ L + 1 } }
(
\times_{ n = 1 }^L 
(
\R^{ l_n \times l_{ n - 1 } } \times \R^{ l_n } 
)
)
,
\end{equation}
let 
$
\mathcal{P}
\colon \mathcal{N} \to \N
$,
$ \mathcal{L}
\colon \mathcal{N} \to \cup_{ L \in \{ 2, 3, 4, \dots \} } \N^{L+1}
$,
and
$
\mathcal{R} \colon 
\mathcal{N} 
\to 
\cup_{ k, l \in \N } C( \R^k, \R^l )$
be the functions which satisfy 
for all 
$ L \in \{ 2, 3, 4, \dots \} $, 
$ l_0, l_1, \ldots, l_L \in \N $, 
$ 
\Phi = ((W_1, B_1), \ldots, $ $ (W_L, B_L)) \in 
( \times_{ n = 1 }^L (\R^{ l_n \times l_{n-1} } \times \R^{ l_n } ) )
$,
$ x_0 \in \R^{l_0} $, 
$ \ldots $, 
$ x_{ L - 1 } \in \R^{ l_{ L - 1 } } $ 
with 
$ 
\forall \, n \in \N \cap [1,L) \colon 
x_n = \mathbf{A}_{ l_n }( W_n x_{ n - 1 } + B_n )
$
that 
$
\mathcal{P}( \Phi )
=
\textstyle
\sum\nolimits_{
n = 1
}^L
l_n ( l_{ n - 1 } + 1 )
$,
$
\mathcal{R}(\Phi) \in C( \R^{ l_0 } , \R^{ l_L } )
$,
$\mathcal{L}(\Phi) = ( l_0, l_1, \ldots, l_L ) $,
and
\begin{equation}
( \mathcal{R} \Phi )( x_0 ) = W_L x_{L-1} + B_L ,
\end{equation}	
and let 
$
\phi_{ 1 }, \phi_{ 2}, \mathbb{I} \in \mathcal{N} $, 
$ L_1, L_2 \in  \{2, 3, 4 \ldots \} $,
$ \mathfrak{i},
l_{1,0}, l_{1,1}, \ldots, l_{1,L_1},
l_{2,0}, l_{2,1}, \ldots, l_{2,L_2} \in \N$
satisfy for all $x \in \R^d$, 
$i \in \{1, 2\}$
that
$
\mathcal{R}(\phi_{ i })
\in 
C( \R^d, \R^d )
$,
$ \mathcal{L}(\phi_i) = (l_{i,0}, l_{i,1}, \ldots, l_{i,L_i}) \in \N^{L_i+1}$,
$\mathcal{L}(\mathbb{I}) = (d, \mathfrak{i}, d) \in \N^3 $,  $(\mathcal{R} \, \mathbb{I})(x) = x$, 
and
$l_{1,L_1-1} \leq l_{2,L_2-1} + \mathfrak{i} $. 
Then there exists 
$\psi \in \mathcal{N}$ 
such that for all 
$x \in \R^d$ 
it holds that  
$\mathcal{R}(\psi) \in C(\R^d, \R^d)$,
$\mathcal{L}(\psi) = (l_{1,0}, l_{1,1}, \ldots, l_{1, L_1-1}, l_{2,1}+\mathfrak{i}, l_{2, 2}+ \mathfrak{i}, \ldots, l_{2, L_2-1}+ \mathfrak{i}, l_{2, L_2}) \in \N^{L_1+ L_2}$,
$\mathcal{P}(\psi) \leq \mathcal{P}(\phi_1) + (\mathcal{P}(\phi_2) + \mathcal{P}(\mathbb{I}))^3$, 
and
\begin{equation}
\label{eq:sum:comp:ann}
(\mathcal{R} \psi)(x) =  (\mathcal{R} \phi_1)(x) +  (\mathcal{R} \phi_2) \bigl( (\mathcal{R} \phi_1)(x) \bigr).
\end{equation}
\end{prop}
\begin{proof}[Proof of Proposition~\ref{prop:sum:comp:ANN}]
Throughout this proof let
$(W_{3,1}, B_{3,1}) \in (\R^{\mathfrak{i} \times d} \times \R^{\mathfrak{i}})$,
$(W_{3,2}, $ $ B_{3,2}) \in (\R^{d \times \mathfrak{i}} \times \R^{d})$,
and
$((W_{j,1}, B_{j,1}), \ldots, (W_{j,L_j}, B_{j,L_j})) \in 
( \times_{ n = 1 }^{L_j} (\R^{ l_{j, n} \times l_{j,n-1} } \times \R^{ l_{j,n} } ) )$, $j \in \{1, 2\}$,
satisfy for all $j \in \{1, 2\}$ that
$\mathbb{I} = ((W_{3,1}, B_{3,1}), (W_{3,2}, B_{3,2}))$
and
$\phi_j =  ((W_{j,1}, B_{j,1}), \ldots,  (W_{j,L_j}, B_{j,L_j}))$,
let $L_4 = L_1 +L_2-1$, 
let
$l_{4,0}, l_{4,1}, \ldots, l_{4,L_4} \in \N$ 
satisfy for all 
$i \in \{0, 1, \ldots, L_1-1 \} $, 
$j \in \{0, 1,  \ldots,  L_2-2\}$ 
that 
\begin{equation}
  l_{4,i}= l_{1,i} , 
  \qquad
  l_{4,L_1 +j} = l_{2,j+1} +\mathfrak{i} ,
  \qquad 
  \text{and}
  \qquad
  l_{4,L_4} = l_{2, L_2} =d ,
\end{equation}
let  
$((W_{4,1}, B_{4,1}), \ldots, $ $ (W_{4,L_4}, B_{4,L_4})) \in 
( \times_{ n = 1 }^{L_4} (\R^{ l_{4, n} \times l_{4,n-1} } \times \R^{ l_{4,n} } ) ) $,
assume for all 
$i \in \{ 1, 2, \ldots, L_1-1 \} $
that 
\begin{equation}
\label{eq:sum:comp:1}
(W_{4,i}, B_{4,i}) = (W_{1,i}, B_{1,i}) \in (\R^{ l_{1, i} \times l_{1,i-1} } \times \R^{ l_{1,i} } ) = (\R^{ l_{4, i} \times l_{4,i-1} } \times \R^{ l_{4,i} } ),
\end{equation}
\begin{equation}
W_{4, L_1} = 
\begin{pmatrix}
W_{2,1} W_{1,L_1} \\
W_{3,1} W_{1, L_1}
\end{pmatrix}
\in \R^{(l_{2,1} + \mathfrak{i}) \times l_{1, L_1-1}} = \R^{l_{4, L_1} \times l_{4, L_1-1}},
\end{equation}
\begin{equation}
B_{4, L_1} = 
\begin{pmatrix}
W_{2,1} B_{1, L_1} + B_{2,1}\\
W_{3,1} B_{1, L_1} + B_{3,1}
\end{pmatrix}
\in \R^{(l_{2,1} + \mathfrak{i})} = \R^{l_{4, L_1}},
\end{equation}
\begin{equation}
W_{4, L_4} = \big( W_{2, L_2}  \quad W_{3,2} \big) \in \R^{ l_{2, L_2} \times (l_{2,L_2-1} + \mathfrak{i})} = \R^{l_{4, L_4} \times l_{4, L_4-1}},
\end{equation}
\begin{equation}
\text{and} \qquad B_{4, L_4} = B_{2, L_2} + B_{3,2} \in \R^{ l_{2, L_2}} = \R^{l_{4, L_4}},
\end{equation}
assume for all $j \in \N \cap [ 0, L_2-2 ] $ that
\begin{equation}
W_{4, L_1 + j} = 
\begin{pmatrix}
W_{2,j+1}&		0\\
0&		W_{3,1} W_{3,2}
\end{pmatrix}
\in \R^{(l_{2, j+1} +\mathfrak{i}) \times (l_{2,j} + \mathfrak{i})} = \R^{l_{4,L_1+j} \times l_{4, L_1 +j-1}}
\end{equation}
\begin{equation}
\text{and} \qquad B_{4, L_1 +j} =
\begin{pmatrix}
B_{2, j+1} \\
W_{3,1} B_{3,2} + B_{3,1}
\end{pmatrix}
\in \R^{(l_{2, j+1} + \mathfrak{i})} = \R^{l_{4, L_1 +j}},
\end{equation}
and let $\psi = ((W_{4,1}, B_{4,1}), \ldots,  (W_{4,L_4}, B_{4,L_4})) \in \mathcal{N}$. Observe that for all $x \in \R^{l_{4, L_1 -1}}$ it holds that
\begin{equation}
\label{eq:sum:comp:first}
\begin{split}
W_{4, L_1} x + B_{4, L_1} &= 
\begin{pmatrix}
W_{2,1} W_{1,L_1} \\
W_{3,1} W_{1, L_1}
\end{pmatrix}
x +
\begin{pmatrix}
W_{2,1} B_{1, L_1} + B_{2,1}\\
W_{3,1} B_{1, L_1} + B_{3,1}
\end{pmatrix}\\
& = 
\begin{pmatrix}
W_{2,1} W_{1,L_1}x + W_{2,1} B_{1, L_1} + B_{2,1}\\
W_{3,1} W_{1, L_1}x + W_{3,1} B_{1, L_1} + B_{3,1}
\end{pmatrix}\\
& = 
\begin{pmatrix}
W_{2,1} (W_{1,L_1}x +  B_{1, L_1}) + B_{2,1}\\
W_{3,1} (W_{1, L_1}x +  B_{1, L_1}) + B_{3,1}
\end{pmatrix}.
\end{split}
\end{equation}
Moreover, note that for all
$ i \in \N \cap [ 0, L_2-2 ] $,
$x \in \R^{l_{2,  i}}$, $y \in \R^{\mathfrak{i}}$ it holds that
\begin{equation}
\label{eq:sum:comp:middle}
\begin{split}
W_{4, L_1 +i} 
\begin{pmatrix}
x\\
y
\end{pmatrix}
+ B_{4, L_1 + i} 
& =
\begin{pmatrix}
W_{2,i+1}&		0\\
0&		W_{3,1} W_{3,2}
\end{pmatrix}
\begin{pmatrix}
x\\
y
\end{pmatrix}
+
\begin{pmatrix}
B_{2, i+1} \\
W_{3,1} B_{3,2} + B_{3,1}
\end{pmatrix}\\
& = 
\begin{pmatrix}
W_{2, i+1} x + B_{2, i+1} \\
W_{3,1} W_{3,2} y + W_{3,1} B_{3,2} + B_{3,1}
\end{pmatrix}\\
& = 
\begin{pmatrix}
W_{2, i+1} x + B_{2, i+1} \\
W_{3,1} (W_{3,2} y +  B_{3,2}) + B_{3,1}
\end{pmatrix}.
\end{split}
\end{equation}
Next observe that for all $x \in \R^{l_{2, L_2 -1}}$, $ y \in \R^{\mathfrak{i}}$ it holds that
\begin{equation}
\label{eq:sum:comp:last}
\begin{split}
W_{4, L_4} 
\begin{pmatrix}
x\\
y
\end{pmatrix} 
+ B_{4, L_4}
& = 
\big( W_{2, L_2}  \quad W_{3,2} \big)
\begin{pmatrix}
x\\
y
\end{pmatrix} 
+ B_{2, L_2} + B_{3,2} \\
& =
(W_{2, L_2} x + B_{2, L_2}) +( W_{3,2} y   + B_{3,2}).
\end{split}
\end{equation}
Moreover, note that the hypothesis that $ \forall \, y \in \R^d \colon (\mathcal{R} \, \mathbb{I}) (y) = y$ ensures that  for all $x \in \R^d$ it holds that
\begin{equation}
W_{3,2} \mathbf{A}_{\mathfrak{i}} (W_{3,1}x + B_{3,1}) + B_{3,2} = x.
\end{equation}
Combining this, \eqref{eq:sum:comp:1}, \eqref{eq:sum:comp:first}, \eqref{eq:sum:comp:middle}, and \eqref{eq:sum:comp:last} proves that for all $ x \in \R^d$ it holds that
\begin{equation}
\label{eq:sum:comp:id}
(\mathcal{R} \psi)(x) =  (\mathcal{R} \phi_1)(x) +  (\mathcal{R} \phi_2) \bigl((\mathcal{R} \phi_1)(x) \bigr).
\end{equation}
Next observe that 
\begin{equation}
\begin{split}
\mathcal{P}(\psi) & = \sum_{n=1}^{L_4} l_{4,n} (l_{4, n-1} +1) 
= \sum_{n=1}^{L_1+ L_2 -1} l_{4,n} (l_{4, n-1} +1)\\
& = \Biggl[ \sum_{n=1}^{L_1-1} l_{4,n} (l_{4, n-1} +1) \Biggr] + \Biggl[ \sum_{n=L_1 +1}^{L_1+ L_2 -2} l_{4,n} (l_{4, n-1} +1) \Biggr] \\
& \quad + l_{4,L_1} (l_{4, L_1-1} +1) + l_{4,L_1 +L_2-1} (l_{4, L_1+L_2 -2} +1)\\
& = \Biggl[ \sum_{n=1}^{L_1-1} l_{1,n} (l_{1, n-1} +1) \Biggr] + \Biggl[ \sum_{n=1}^{L_2 -2} l_{4,L_1 + n} (l_{4, L_1+ n-1} +1) \Biggr] \\
& \quad + (l_{2,1} + \mathfrak{i}) (l_{1, L_1-1} +1) + l_{2, L_2} (l_{2, L_2 -1} + \mathfrak{i}+1).
\end{split}
\end{equation}
The hypothesis that $l_{1,L_1-1} \leq l_{2,L_2-1} + \mathfrak{i}$ therefore assures that
\begin{equation}
\begin{split}
\mathcal{P}(\psi) & = \Biggl[ \sum_{n=1}^{L_1-1} l_{1,n} (l_{1, n-1} +1) \Biggr] + \Biggl[ \sum_{n=1}^{L_2 -2} (l_{2, n +1} + \mathfrak{i}) (l_{2, n} + \mathfrak{i} +1) \Biggr] \\
& \quad + (l_{2,1} + \mathfrak{i}) (l_{1, L_1-1} +1) + l_{2, L_2} (l_{2, L_2 -1} + \mathfrak{i} +1) \\
& \leq \mathcal{P}(\phi_1) + \Biggl[ \sum_{n=1}^{L_2 -2}  \left[ l_{2, n +1}(l_{2,n} +1) + l_{2, n+1} \mathfrak{i} + \mathfrak{i} \, (l_{2,n} +1) + \mathfrak{i}^2 \right] \Biggr] \\
& \quad + (l_{2,1} + \mathfrak{i}) (l_{2,L_2-1} + \mathfrak{i} +1) + l_{2, L_2} (l_{2, L_2 -1} + \mathfrak{i} +1)\\
& \leq \mathcal{P}(\phi_1) + \Biggl[ \sum_{n=2}^{L_2-1} l_{2, n}(l_{2, n-1} +1) \Biggr] + \mathfrak{i}   \Biggl[ \sum_{n=1}^{L_2 -2} ( l_{2, n+1} + l_{2, n} ) \Biggr] \\
& \quad  + (L_2-2)( \mathfrak{i} + \mathfrak{i}^2) +  (l_{2,1} + \mathfrak{i}) (l_{2,L_2-1} + \mathfrak{i} +1) \\
& \quad + l_{2, L_2} (l_{2, L_2 -1} +1) +  l_{2, L_2}  \mathfrak{i}.
\end{split}
\end{equation}
Hence, we obtain that
\begin{equation}
\label{eq:sum:comp:par}
\begin{split}
\mathcal{P}(\psi)  &\leq \mathcal{P}(\phi_1) + \Biggl[ \sum_{n=2}^{L_2} l_{2, n}(l_{2, n-1} +1) \Biggr] + \mathfrak{i}  \Biggl[ \sum_{n=1}^{L_2 -2} ( l_{2, n+1} + l_{2, n} ) \Biggr] \\
& \quad  + (L_2-2)( \mathfrak{i} + \mathfrak{i}^2) + \mathfrak{i} \, (l_{2,1} + l_{2, L_2-1} + l_{2, L_2})\\
& \quad  + l_{2,1}(l_{2, L_2-1} +1) + \mathfrak{i} \, (\mathfrak{i}+1)\\
& \leq \mathcal{P}(\phi_1) + \mathcal{P}(\phi_2) + 2 \, \mathfrak{i} \Biggl[ \sum_{n=1}^{L_2} l_{2,n} \Biggr] + (L_2-1)( \mathfrak{i}  + \mathfrak{i}^2) + l_{2,1}(l_{2, L_2-1} +1).
\end{split}
\end{equation}
Next observe that
\begin{equation}
\begin{split}
\mathcal{P}(\phi_2) = \sum_{n=1}^{L_2} l_{2,n}(l_{2, n-1}+1) \geq 2 \Biggl[ \sum_{n=1}^{L_2} l_{2,n} \Biggr] \geq 2 \, L_2
\end{split}
\end{equation} 
and 
\begin{equation}
\begin{split}
\mathcal{P}(\mathbb{I}) = \mathfrak{i} \, (d+1) + d \, (\mathfrak{i} +1) \geq 2 \, \mathfrak{i} +2.
\end{split}
\end{equation}
This and \eqref{eq:sum:comp:par} demonstrate that 
\begin{equation}
\begin{split}
\mathcal{P}(\psi)  &\leq  \mathcal{P}(\phi_1) + \mathcal{P}(\phi_2)  + \mathfrak{i} \, \mathcal{P}(\phi_2) + ( \mathfrak{i}  + \mathfrak{i}^2) \mathcal{P}(\phi_2)
+ \mathcal{P}(\phi_2) (\mathcal{P}(\phi_2) +1)\\
& = \mathcal{P}(\phi_1) + \mathcal{P}(\phi_2)(2+ 2 \, \mathfrak{i} + \mathfrak{i}^2) + (\mathcal{P}(\phi_2))^2\\
& \leq \mathcal{P}(\phi_1) + \mathcal{P}(\phi_2)( \mathcal{P}(\mathbb{I}) + (\mathcal{P}(\mathbb{I}))^2) + (\mathcal{P}(\phi_2))^2 \\
& \leq \mathcal{P}(\phi_1) + 2 \, \mathcal{P}(\phi_2) (\mathcal{P}(\mathbb{I}))^2 + (\mathcal{P}(\phi_2))^2 \\
& \leq  \mathcal{P}(\phi_1) + (\mathcal{P}(\phi_2) + \mathcal{P}(\mathbb{I}))^3.
\end{split}
\end{equation}
Combining this with \eqref{eq:sum:comp:id} establishes \eqref{eq:sum:comp:ann}. The proof of Proposition~\ref{prop:sum:comp:ANN} is thus completed.
\end{proof}

\subsection{Representations of the $ d $-dimensional identities}
\label{sec:rep_identity}

\begin{lemma}[Artificial neural networks with rectifier functions]
	\label{lem:identity}
	Let $ d \in \N $, 
	let 
	$ \mathbf{A}_n \colon \R^n \to \R^n $, 
	$ n \in \N $, 
	be the  functions 
	which satisfy 
	for all 
	$
	n \in \N
	$,
	$ x = ( x_1, \dots, x_n ) \in \R^n $
	that
	$ 
	\mathbf{A}_n(x)
	=
	( \max\{x_1, 0\}, \ldots, \max\{x_n, 0\} )
	$,
	let 
	\begin{equation}
	\mathcal{N}
	=
	\cup_{ L \in \{ 2, 3, 4, \dots \} }
	\cup_{ ( l_0, l_1, \ldots, l_L ) \in \N^{ L + 1 } }
	(
	\times_{ n = 1 }^L 
	(
	\R^{ l_n \times l_{ n - 1 } } \times \R^{ l_n } 
	)
	)
	,
	\end{equation}
	and let 
	$
	\mathcal{P}
	\colon \mathcal{N} \to \N
	$,
	$ \mathcal{L}
	\colon \mathcal{N} \to \cup_{ L \in \{ 2, 3, 4, \dots \} } \N^{L+1}
	$,
	and
	$
	\mathcal{R} \colon 
	\mathcal{N} 
	\to 
	\cup_{ k, l \in \N } C( \R^k, \R^l )$
	be the functions which satisfy 
	for all 
	$ L \in \{ 2, 3, 4, \dots \} $, 
	$ l_0, l_1, \ldots, l_L \in \N $, 
	$ 
	\Phi = ((W_1, B_1), \ldots, $ $ (W_L, B_L)) \in 
	( \times_{ n = 1 }^L (\R^{ l_n \times l_{n-1} } \times \R^{ l_n } ) )
	$,
	$ x_0 \in \R^{l_0} $, 
	$ \ldots $, 
	$ x_{ L - 1 } \in \R^{ l_{ L - 1 } } $ 
	with 
	$ 
	\forall \, n \in \N \cap [1,L) \colon 
	x_n = \mathbf{A}_{ l_n }( W_n x_{ n - 1 } + B_n )
	$
	that 
	$
	\mathcal{P}( \Phi )
	=
	\textstyle
	\sum\nolimits_{
		n = 1
	}^L
	l_n ( l_{ n - 1 } + 1 )
	$,
	$
	\mathcal{R}(\Phi) \in C( \R^{ l_0 } , \R^{ l_L } )
	$,
	$\mathcal{L}(\Phi) = ( l_0, l_1, \ldots, l_L ) $,
	and
	\begin{equation}
	( \mathcal{R} \Phi )( x_0 ) = W_L x_{L-1} + B_L.
	\end{equation}
	Then there exists 
	$\psi \in \mathcal{N}$ 
	such that for all 
	$x \in \R^d$ 
	it holds that  
	$\mathcal{R}(\psi) \in C(\R^d, \R^d)$,  $\mathcal{L}(\psi) = (d, 2d, d) \in \N^3$, and 
	\begin{equation}
	\label{eq:rectifier}
	(\mathcal{R} \psi) (x) =x.
	\end{equation}
\end{lemma}
\begin{proof}[Proof of Lemma~\ref{lem:identity}]
	Throughout this proof let $ w_1 \in \R^{2 \times 1}$, $w_2 \in \R^{1 \times 2}$, $(W_1, B_1) \in (\R^{(2d) \times d} \times \R^{2d} )$, and $(W_2, B_2) \in (\R^{d \times (2d)} \times \R^{d} )$ satisfy that
	\begin{equation}
	w_1 = \begin{pmatrix}
	1\\
	-1
	\end{pmatrix}
	\in \R^{2 \times 1}, 
	\quad
	W_1 =
	\begin{pmatrix}
	w_1& 0& 0& \cdots& 0\\
	0& w_1& 0& \cdots& 0\\
	0& 0& w_1& \cdots& 0\\
	\vdots& \vdots& \vdots& \ddots& \vdots\\
	0& 0& 0& \cdots& w_1
	\end{pmatrix} 
	\in \R^{(2d) \times d}, \quad B_1 =0 \in \R^{2d},
	\end{equation}
	\begin{equation}
	w_2 = 
	\begin{pmatrix}
	1& -1
	\end{pmatrix}
	\in \R^{1 \times 2}, 
	\quad
	W_2 = 
	\begin{pmatrix}
	w_2& 0& 0& \cdots& 0\\
	0& w_2& 0& \cdots& 0\\
	0& 0& w_2& \cdots& 0\\
	\vdots& \vdots& \vdots& \ddots& \vdots\\
	0& 0& 0& \cdots& w_2
	\end{pmatrix} 
	\in \R^{d \times (2d)}, \quad B_2 =0 \in \R^{d}
	\end{equation}
	and let $\psi = ((W_{1}, B_{1}),   (W_{2}, B_{2})) \in \mathcal{N}$.
	Observe that for all $x = (x_1, x_2, \ldots, x_d) \in \R^d$ it holds that
	\begin{equation}
	W_1 x + B_1 = W_1 x = 
	\begin{pmatrix}
	w_1 x_1 \\
	w_1 x_2 \\
	\vdots\\
	w_1 x_d
	\end{pmatrix} 
	\in \R^{2d}.
	\end{equation}
	This ensures that for all $x = (x_1, x_2, \ldots, x_d) \in \R^d$ it holds that
	\begin{equation}
	\mathbf{A}_{(2d)} (W_1 x + B_1) = 
	\begin{pmatrix}
	\max\{x_1, 0\}\\
	\max\{-x_1, 0\}\\
	\max\{x_2, 0\}\\
	\max\{-x_2, 0\}\\
	\vdots\\
	\max\{x_d, 0\}\\
	\max\{-x_d, 0\}
	\end{pmatrix} 
	\in \R^{2d}.
	\end{equation}
	Hence, we obtain that for all $x = (x_1, x_2, \ldots, x_d) \in \R^d$ it holds that 
	\begin{equation}
	\begin{split}
	W_2 \mathbf{A}_{(2d)} (W_1 x + B_1) + B_2 &= 
	W_2\begin{pmatrix}
	\max\{x_1, 0\}\\
	\max\{-x_1, 0\}\\
	\max\{x_2, 0\}\\
	\max\{-x_2, 0\}\\
	\vdots\\
	\max\{x_d, 0\}\\
	\max\{-x_d, 0\}
	\end{pmatrix} \\
	&= \begin{pmatrix}
	\max\{x_1, 0\} - \max\{-x_1, 0\}\\
	\max\{x_2, 0\} - \max\{-x_2, 0\}\\
	\vdots\\
	\max\{x_d, 0\} - \max\{-x_d, 0\}\\
	\end{pmatrix}
	= x \in \R^d.
	\end{split}
	\end{equation}
	This demonstrates that for all $x \in \R^d$ it holds that
	\begin{equation}
	\label{eq:lem:identity}
	(\mathcal{R} \psi)(x) = x.
	\end{equation}
	Combining this with the fact that $\mathcal{L}(\psi) = (d, 2d, d) \in \N^3$ establishes \eqref{eq:rectifier}.
	The proof of Lemma~\ref{lem:identity} is thus completed.
\end{proof}

\section[DNN approximations for partial differential equations (PDEs)]{DNN approximations for partial differential\\ equations (PDEs)}
\label{sec:DNN_PDEs}

In this section we establish in our main result in Theorem~\ref{thm:PDE_approx_Lp} in Subsection~\ref{sec:main_result} below
that rectified DNNs have the capacity to approximate solutions of second-order Kolmogorov PDEs
with nonlinear drift and constant diffusion coefficients
without suffering from the curse of dimensionality.
Our proof of Theorem~\ref{thm:PDE_approx_Lp} is based on an application of Corollary~\ref{cor:PDE_approx_Lp_gen} in Subsection~\ref{sec:PDE_approx_Lp} below.
Corollary~\ref{cor:PDE_approx_Lp_gen}, in turn, follows immediately from Proposition~\ref{prop:PDE_approx_Lp} in Subsection~\ref{sec:PDE_approx_Lp} below.
Proposition~\ref{prop:PDE_approx_Lp} is, roughly speaking, a generalized version of Theorem~\ref{thm:PDE_approx_Lp} which covers a more general type of activation function instead of only the rectifier function as the employed activation function.
Proposition~\ref{prop:PDE_approx_Lp} shows for every $ p \in [2,\infty) $ that the $ L^p( \nu_d ; \R ) $-distance between the solution of the PDE at the time of maturity and the DNN is smaller or equal than the prescribed approximation accuracy $ \varepsilon > 0 $,
where $ \nu_d \colon \mathcal{B}( \R^d ) \to [0,1] $, $ d \in \N $, is a suitable sequence of probability measures.
Corollary~\ref{cor:PDE_approx_Lp_gen} slightly generalizes this result, in particular, by assuming that $ p \in (0,\infty) $ is an arbitrary strictly positive real number instead of assuming that $ p \in [2,\infty) $ is greater or equal than $ 2 $ (cf.\ Proposition~\ref{prop:PDE_approx_Lp}).
Finally, in Corollary~\ref{cor:lebesgue} in Subsection~\ref{sec:lebesgue} below we specialize Theorem~\ref{thm:PDE_approx_Lp} in Subsection~\ref{sec:main_result} to the case where for every $ d \in \N $ we have that the probability measure $ \nu_d $ is nothing else but the uniform distribution on the $ d $-dimensional unit cube $ [0,1]^d $.
Theorem~\ref{thm:intro} in the introduction in Section~\ref{sec:intro} above follows directly from Corollary~\ref{cor:lebesgue} in Subsection~\ref{sec:lebesgue}.

\subsection{DNN approximations with general activation functions}
\label{sec:PDE_approx_Lp}


\begin{prop}
\label{prop:PDE_approx_Lp}
Let 
$ T, \kappa \in (0,\infty) $, $\eta \in [1, \infty)$,  $p \in [2, \infty)$,
let 
$
A_d = ( a_{ d, i, j } )_{ (i, j) \in \{ 1, \dots, d \}^2 } $ $ \in \R^{ d \times d }
$,
$ d \in \N $,
be symmetric positive semidefinite matrices, 
for every $ d \in \N $ 
let 
$
\left\| \cdot \right\|_{ \R^d } \colon \R^d \to [0,\infty)
$
be the $ d $-dimensional Euclidean norm
and let $\nu_d  \colon \mathcal{B}(\R^d) \to [0,1] $ be a probability measure,
let
$ f_{0,d} \colon \R^d \to \R $, $ d \in \N $,
and
$ f_{ 1, d } \colon \R^d \to \R^d $,
$ d \in \N $,
be functions,
let 
$ \mathbf{A}_d \colon \R^d \to \R^d $, 
$ d \in \N $, 
and 
$ \mathbf{a} \colon \R \to \R$
be  continuous functions 
which satisfy 
for all 
$
d \in \N
$,
$ x = ( x_1, \dots, x_d ) \in \R^d $
that
$ 
\mathbf{A}_d(x)
=
( \mathbf{a}(x_1), \ldots, \mathbf{a}(x_d) )
$,
let 
\begin{equation}
\mathcal{N}
=
\cup_{ L \in \{ 2, 3, 4, \dots \} }
\cup_{ ( l_0, l_1, \ldots, l_L ) \in \N^{ L + 1 } }
(
\times_{ n = 1 }^L 
(
\R^{ l_n \times l_{ n - 1 } } \times \R^{ l_n } 
)
)
,
\end{equation}
let 
$
\mathcal{P}
\colon \mathcal{N} \to \N
$,
$ \mathcal{L}
\colon \mathcal{N} \to \cup_{ L \in \{ 2, 3, 4, \dots \} } \N^{L+1}
$,
and
$
\mathcal{R} \colon 
\mathcal{N} 
\to 
\cup_{ k, l \in \N } C( \R^k, \R^l )
$
be the functions which satisfy 
for all 
$ L \in \{ 2, 3, 4, \dots \} $, 
$ l_0, l_1, \ldots, l_L \in \N $, 
$ 
\Phi = ((W_1, B_1), \ldots, $ $ (W_L, B_L)) \in 
( \times_{ n = 1 }^L (\R^{ l_n \times l_{n-1} } \times \R^{ l_n } ) )
$,
$ x_0 \in \R^{l_0} $, 
$ \ldots $, 
$ x_{ L - 1 } \in \R^{ l_{ L - 1 } } $ 
with 
$ 
\forall \, n \in \N \cap [1,L) \colon 
x_n = \mathbf{A}_{ l_n }( W_n x_{ n - 1 } + B_n )
$
that 
$
\mathcal{P}( \Phi )
=
\textstyle
\sum\nolimits_{
n = 1
}^L
l_n ( l_{ n - 1 } + 1 )
$,
$
\mathcal{R}(\Phi) \in C( \R^{ l_0 } , \R^{ l_L } )
$, $\mathcal{L}(\Phi) = ( l_0, l_1, \ldots, l_L ) $,
and
\begin{equation}
( \mathcal{R} \Phi )( x_0 ) = W_L x_{L-1} + B_L ,
\end{equation}
let 
$
( \phi^{ m, d }_{ \varepsilon } )_{ 
(m, d, \varepsilon) \in \{ 0, 1 \} \times \N \times (0,1] 
} 
\subseteq \mathcal{N}
$,
$
( \phi^{ 2, d } )_{ 
d \in \N
} 
\subseteq \mathcal{N}
$,
and 
assume for all
$ d \in \N $, 
$ \varepsilon \in (0,1] $, 
$ 
x, y \in \R^d
$
that
$
\mathcal{R}( \phi^{ 0, d }_{ \varepsilon } )
\in 
C( \R^d, \R )
$,
$
\mathcal{R}( \phi^{ 1, d }_{ \varepsilon } )
,
\mathcal{R}( \phi^{ 2, d } )
\in
C( \R^d, \R^d )
$,
$
|
f_{ 0, d }( x )
| 
+
\sum_{ i , j = 1 }^d
| a_{ d, i, j } |
\leq 
\kappa d^{ \kappa }
( 1 + \| x \|^{ \kappa }_{ \R^d } )
$,
$
\| 
f_{ 1, d }( x ) 
- 
f_{ 1, d }( y )
\|_{ \R^d }
\leq 
\kappa 
\| x - y \|_{ \R^d } 
$,
$
\|
( \mathcal{R} \phi^{ 1, d }_{ \varepsilon } )(x)    
\|_{ \R^d }	
\leq 
\kappa ( d^{ \kappa } + \| x \|_{ \R^d } )
$,
$
\mathcal{P}( \phi^{ 2, d } ) +
\sum_{ m = 0 }^1
\mathcal{P}( \phi^{ m, d }_{ \varepsilon } ) 
\leq \kappa d^{ \kappa } \varepsilon^{ - \kappa }
$, $ |( \mathcal{R} \phi^{ 0, d }_{ \varepsilon } )(x) - ( \mathcal{R} \phi^{ 0, d }_{ \varepsilon } )(y)| \leq \kappa d^{\kappa} (1 + \|x\|_{\R^d}^{\kappa} + \|y \|_{\R^d}^{\kappa})\|x-y\|_{\R^d}$, $\mathcal{L}(\phi^{2,d}) \in \N^3 $, $(\mathcal{R}\phi^{2,d})(x) = x$,
$ \int_{\R^d} 
 \|z \|_{ \R^d }^{p (2\kappa+1)}
\, \nu_d (dz) \leq \eta d^{\eta}$,
and
\begin{equation}
\label{eq:appr:coef}
| 
f_{ 0, d }(x) 
- 
( \mathcal{R} \phi^{ 0, d }_{ \varepsilon } )(x)
|
+
\| 
f_{ 1, d }(x) 
- 
( \mathcal{R} \phi^{ 1, d }_{ \varepsilon } )(x)
\|_{ \R^d }
\leq 
\varepsilon \kappa d^{ \kappa }
(
1 + \| x \|^{ \kappa }_{ \R^d }
)
.
\end{equation}
Then 
\begin{enumerate}[(i)]
\item 
\label{item:existence_vis}
there exist unique 
at most polynomially growing functions 
$ u_d \colon [0,T] \times \R^{d} \to \R $,
$ d \in \N $, 
such that
for all $ d \in \N $, $ x \in \R^d $ it holds that
$ u_d( 0, x ) = f_{ 0, d }( x ) $
and such that for all $ d \in \N $ 
it holds that
$ u_d $ is a viscosity solution of
\begin{equation}
\begin{split}
( \tfrac{ \partial }{\partial t} u_d )( t, x ) 
& = 
( \tfrac{ \partial }{\partial x} u_d )( t, x )
\,
f_{ 1, d }( x )
+
\sum_{ i, j = 1 }^d
a_{ d, i, j }
\,
( \tfrac{ \partial^2 }{ \partial x_i \partial x_j } u_d )( t, x )
\end{split}
\end{equation}
for $ ( t, x ) \in (0,T) \times \R^d $
and 
\item
\label{item:main_statement_Lp}
there exist
$
( 
\psi_{ d, \varepsilon } 
)_{ (d , \varepsilon)  \in \N \times (0,1] } \subseteq \mathcal{N}
$,
$
c \in \R
$
such that
for all 
$
d \in \N 
$,
$
\varepsilon \in (0,1] 
$
it holds that
$
\mathcal{P}( \psi_{ d, \varepsilon } ) 
\leq
c \, d^c \varepsilon^{ - c } 
$,
$
\mathcal{R}( \psi_{ d, \varepsilon } )
\in C( \R^{ d }, \R )
$,
and
\begin{equation}
\left[
\int_{ \R^d }
|
u_d(T,x) - ( \mathcal{R} \psi_{ d, \varepsilon } )( x )
|^p
\,
\nu_d(dx)
\right]^{ \nicefrac{ 1 }{ p } }
\leq
\varepsilon 
.
\end{equation}
\end{enumerate}
\end{prop}

\begin{proof}[Proof of Proposition~\ref{prop:PDE_approx_Lp}]
Throughout this proof 
let $ \iota \in \R $ be the real number given by 
$ \iota = \max\{ \kappa, 1 \} $, 
let $ \mathcal{A}_d \in \R^{ d \times d } $, $ d \in \N $, satisfy 
for all $ d \in \N $ that
$
\mathcal{A}_d = \sqrt{ 2 A_d }
$,
let $ \Phi^{ 0, d }_{ \delta } \colon \R^d \to \R $, $ \delta \in (0,1] $, $ d \in \N $,
and $ \Phi^{ 1, d }_{ \delta } \colon \R^d \to \R^d $, $ \delta \in (0,1] $, $  d \in \N $, 
be the functions which satisfy 
for all $ m \in \{ 0, 1\} $, $ d \in \N $,
$ \delta \in (0,1] $,
$ x \in \R^d $
that
\begin{equation}
\Phi^{ m, d }_{ \delta }( x ) 
= 
( 
\mathcal{R} \phi^{ m, d }_{ \delta } 
)( x )
,
\end{equation}
let $ ( \Omega, \mathcal{F}, \P ) $ be a probability space, 
let $ W^{ d, m } \colon [0,T] \times \Omega \to \R^d $, $ d, m \in \N $, 
be independent standard Brownian motions, 
let $ \varpi_{ d, q } \in \R $, $ d \in \N $, $ q \in (0,\infty) $, satisfy 
for all $ q \in (0,\infty) $, $ d \in \N $  that
\begin{equation}
\varpi_{ d, q } 
=
\big(
\E\big[ 
\| \mathcal{A}_d W^{ d, 1 }_T \|^q_{ \R^d }
\big]
\big)^{ \nicefrac{ 1 }{ q } }
,
\end{equation}
let $ X^{ d, x } \colon [0,T] \times \Omega \to \R^d $, $ d \in \N $, $ x \in \R^d $, 
be stochastic processes with continuous sample paths 
which satisfy for all $ x \in \R^d $, $ d \in \N $,  $ t \in [0,T] $ that
\begin{equation}
\label{eq:X_processes}
X^{ d, x }_t 
= x + \int_0^t f_{ 1, d }( X^{ d, x }_s ) \, ds 
+ 
\mathcal{A}_d
W^{ d, 1 }_t 
\end{equation} 
(cf.~Theorem~\ref{thm:feynman}), let $ \chi_{ \delta } \colon [0,T] \to [0,T] $, $ \delta \in (0,1] $,
be the functions which satisfy for all 
$ \delta \in (0,1] $, $ t \in [0,T] $ that
\begin{equation}
\chi_{ \delta }( t )
=
\max\!\left(
\left\{ 
0, \delta^2, 2 \delta^2, 3 \delta^2, \dots
\right\}
\cap 
[0,t]
\right)
,
\end{equation}
let 
$ Y^{ \delta, d, m, x } \colon [0,T] \times \Omega \to \R^d $,
$ \delta \in (0,1] $,
$ d, m \in \N $,
$ x \in \R^d $,
be stochastic processes with continuous sample paths 
which satisfy for all 
$ x \in \R^d $,
$ d, m \in \N $,
$
\delta \in (0,1]
$,
$ t \in [0,T] $
that
\begin{equation}
\label{eq:Y_processes}
Y^{ \delta, d, m, x }_t 
=
x
+
\int_0^t
\Phi^{ 1, d }_{ \delta }\big( 
Y^{ \delta, d, m, x }_{ \chi_{ \delta }( s ) } 
\big)
\,
ds
+
\mathcal{A}_d
W^{ d, m }_t
,
\end{equation} 
let $\mathfrak{M}_{d, \varepsilon} \in \N$, $d \in \N$, $\varepsilon \in (0,1]$, be the natural numbers which satisfy for all 
$\varepsilon \in (0,1]$,  $d \in \N$  that
\begin{equation}
\begin{split}
&
  \mathfrak{M}_{d, \varepsilon} 
\\ & 
  = \min\!\bigg( \N \cap \bigg[ \Big(\tfrac{2^{(\kappa+4)}p \kappa d^{\kappa} e^{\kappa^2 T}}{\varepsilon} \Big)^2 
\Big[ 1+  \big( \kappa d^{ \kappa } T 
+ 
\sqrt{2(p \iota-1) \kappa d^{\kappa} T}
\big)^{p \kappa} 
+  \eta d^{\eta} \Big]^{\nicefrac{2}{p}}, \infty \bigg) \bigg),
\end{split}
\end{equation}
and let $\mathcal{D}_{d, \varepsilon} \in (0, 1]$, $d \in \N$, $\varepsilon \in (0,1]$, be the real numbers which satisfy for all 
$\varepsilon \in (0,1]$, $d \in \N$  that
\begin{align}
\label{eq:control:delta}
& \mathcal{D}_{d, \varepsilon} = 
\varepsilon \! \left[
\max \{2 \kappa d^{\kappa}, 1 \}
+
T^{ -\nicefrac{ 1 }{ 2 } } 
\right]^{-1} \! e^{
	-( 
	3 + 3\kappa +
	[ 
	\kappa^2
	+
	2\kappa
	\iota
	+ 2
	]
	T
	)
} \big| \! \max\{ 1, 2 \kappa (\kappa +1 ) d^{\kappa} \}\big|^{-1} 
2^{ -(2\iota +  1) } 
\nonumber
\\
& 
\cdot 
\bigg[ \big| 2
+ 
\max\{ 1, \kappa d^{\kappa}, \| f_{1,d}( 0 ) \|_{\R^d} \} \max\{ 1, T \} 
+ 
\sqrt{2(2\iota-1)  \kappa d^{\kappa} T}
\big|^{ p\iota + p \kappa}  +  \eta d^{\eta} \bigg]^{ - \frac{ 1 }{ p } }. 
\end{align}
Observe that
the assumption that for all $ d \in \N $, 
$ \varepsilon \in (0,1] $
it holds
that
$
\mathcal{R}( \phi^{ 0, d }_{ \varepsilon } )
\in 
C( \R^d, \R )
$ and 
\eqref{eq:appr:coef} ensure that $f_{0,d} \in 
C( \R^d, \R )$. This, the fact that for all $d \in \N$ it holds that the function $f_{1,d} \colon \R^d \to \R^d$ is Lipschitz continuous, and
 Theorem~\ref{thm:feynman} establish item~\eqref{item:existence_vis}. 
It thus remains to prove item~\eqref{item:main_statement_Lp}.
For this note that 
the fact that $ \forall \, y, z \in \R, \alpha \in [1, \infty) \colon |y + z|^{\alpha}  \leq 2^{\alpha -1}(|y|^{\alpha} + |z|^{\alpha})$ and 
Theorem~\ref{thm:feynman}
ensure that 
for all $ M, d \in \N $, $ \delta \in (0,1] $
it holds that
\begin{equation}
\label{eq:apply_feynman}
\begin{split}
&
\int_{ \R^d }
\E\Bigg[
\bigg|
u_d(T,x) 
- 
\frac{ 1 }{ M } 
\bigg[ 
\textstyle
\sum\limits_{ m = 1 }^M
\Phi^{ 0, d }_{ \delta }(
Y^{ \delta, d, m, x }_T
)
\bigg]
\bigg|^p
\Bigg]
\,
\nu_d (dx)
\\ &
\leq 
2^{p-1} \int_{ \R^d }
\E\Big[
\big|
u_d(T,x) 
- 
\E\big[
\Phi^{ 0, d }_{ \delta }(
Y^{ \delta, d, 1, x }_T
)
\big]
\big|^p
\Big]
\,
\nu_d (dx)
\\ &
+
2^{p-1} \int_{ \R^d }
\E\Bigg[
\bigg|
\E\big[
\Phi^{ 0, d }_{ \delta }(
Y^{ \delta, d, 1, x }_T
)
\big]
- 
\frac{ 1 }{ M } 
\bigg[ 
\textstyle
\sum\limits_{ m = 1 }^M
\Phi^{ 0, d }_{ \delta }(
Y^{ \delta, d, m, x }_T
)
\bigg]
\bigg|^p
\Bigg]
\,
\nu_d (dx)
\\ & =
2^{p-1} \int_{ \R^d }
\big|
\E\big[ 
f_{ 0, d }( X^{ d, x }_T )
\big]
- 
\E\big[
\Phi^{ 0, d }_{ \delta }(
Y^{ \delta, d, 1, x }_T
)
\big]
\big|^p
\,
\nu_d(dx)
\\ &
+
2^{p-1} \int_{ \R^d }
\E\Bigg[
\bigg|
\E\big[
\Phi^{ 0, d }_{ \delta }(
Y^{ \delta, d, 1, x }_T
)
\big]
- 
\frac{ 1 }{ M } 
\bigg[ 
\textstyle
\sum\limits_{ m = 1 }^M
\Phi^{ 0, d }_{ \delta }(
Y^{ \delta, d, m, x }_T
)
\bigg]
\bigg|^p
\Bigg]
\,
\nu_d (dx)
.
\end{split}
\end{equation}
The fact that $ 2 \sqrt{p-1} \leq p$ and, e.g., Grohs et al.~\cite[Corollary~2.5]{GrohsWurstemberger2018}
hence prove that
for all $ M, d \in \N $, $ \delta \in (0,1] $
it holds that
\begin{equation}
\begin{split}
&
\int_{ \R^d }
\E\Bigg[
\bigg|
u_d(T,x) 
- 
\frac{ 1 }{ M } 
\bigg[ 
\textstyle
\sum\limits_{ m = 1 }^M
\Phi^{ 0, d }_{ \delta }(
Y^{ \delta, d, m, x }_T
)
\bigg]
\bigg|^p
\Bigg]
\,
\nu_d (dx)
\\ & \leq
2^{p-1} \int_{ \R^d }
\big|
\E\big[ 
f_{ 0, d }( X^{ d, x }_T )
\big]
- 
\E\big[
\Phi^{ 0, d }_{ \delta }(
Y^{ \delta, d, 1, x }_T
)
\big]
\big|^p
\,
\nu_d (dx)
\\ &
+
\frac{ 2^{p-1} (2 \sqrt{p-1})^p }{ M^{\nicefrac{p}{2}} } 
\int_{ \R^d }
\E\Big[
\big|
\Phi^{ 0, d }_{ \delta }(
Y^{ \delta, d, 1, x }_T
)
-
\E\big[
\Phi^{ 0, d }_{ \delta }(
Y^{ \delta, d, 1, x }_T
)
\big]
\big|^p
\Big]
\,
\nu_d(dx)
\\ & \leq
2^{p-1} \int_{ \R^d }
\big|
\E\big[ 
f_{ 0, d }( X^{ d, x }_T )
\big]
- 
\E\big[
\Phi^{ 0, d }_{ \delta }(
Y^{ \delta, d, 1, x }_T
)
\big]
\big|^p
\,
\nu_d (dx)
\\ &
+
\frac{ 2^{p-1} p^p }{ M^{\nicefrac{p}{2}} } 
\int_{ \R^d }
\E\Big[
\big|
\Phi^{ 0, d }_{ \delta }(
Y^{ \delta, d, 1, x }_T
)
-
\E\big[
\Phi^{ 0, d }_{ \delta }(
Y^{ \delta, d, 1, x }_T
)
\big]
\big|^p
\Big]
\,
\nu_d(dx)
.
\end{split}
\end{equation}
The fact that 
$ \forall \, y, z \in \R, \alpha \in [1, \infty) \colon |y + z|^{\alpha}  \leq 2^{\alpha -1}(|y|^{\alpha} + |z|^{\alpha})$ and Jensen's inequality
therefore assure that
for all $ M, d \in \N $, $ \delta \in (0,1] $
it holds that
\begin{equation}
\label{eq:weak:RHS}
\begin{split}
&
\int_{ \R^d }
\E\Bigg[
\bigg|
u_d(T,x) 
- 
\frac{ 1 }{ M } 
\bigg[ 
\textstyle
\sum\limits_{ m = 1 }^M
\Phi^{ 0, d }_{ \delta }(
Y^{ \delta, d, m, x }_T
)
\bigg]
\bigg|^p
\Bigg]
\,
\nu_d (dx)
\\ & \leq
2^{p-1} \int_{ \R^d }
\big|
\E\big[ 
f_{ 0, d }( X^{ d, x }_T )
\big]
- 
\E\big[
\Phi^{ 0, d }_{ \delta }(
Y^{ \delta, d, 1, x }_T
)
\big]
\big|^p
\,
\nu_d (dx)
\\ &
+
\frac{ 2^{2(p-1)} p^p }{ M^{\nicefrac{p}{2}} } 
\int_{ \R^d }
\E\Big[
\big|
\Phi^{ 0, d }_{ \delta }(
Y^{ \delta, d, 1, x }_T
) \big|^p \Big]
+ \E \Big[ \big|
\E\big[
\Phi^{ 0, d }_{ \delta }(
Y^{ \delta, d, 1, x }_T
)
\big]
\big|^p
\Big]
\,
\nu_d(dx)
\\  & \leq
2^{p-1} \int_{ \R^d }
\big|
\E\big[ 
f_{ 0, d }( X^{ d, x }_T )
\big]
- 
\E\big[
\Phi^{ 0, d }_{ \delta }(
Y^{ \delta, d, 1, x }_T
)
\big]
\big|^p
\,
\nu_d (dx)
\\ &
+
\frac{ 2^{2p-1} p^p }{ M^{\nicefrac{p}{2}} } 
\int_{ \R^d }
\E\Big[
\big|
\Phi^{ 0, d }_{ \delta }(
Y^{ \delta, d, 1, x }_T
) \big|^p \Big]
\,
\nu_d(dx)
.
\end{split}
\end{equation}
Next observe that for all $d \in \N$, $x, y \in \R^d$ it holds that
\begin{equation}
\begin{split}
&2 \int_0^1 \big[r \|x\|_{\R^d} + (1-r) \|y\|_{\R^d}\big]^{\kappa} \, dr \geq  \int_0^1 \big[ r^{\kappa} \|x\|_{\R^d}^{\kappa} + (1-r)^{\kappa} \|y\|_{\R^d}^{\kappa} \big] \, dr \\
&= \big[ \|x\|_{\R^d}^{\kappa} + \|y\|_{\R^d}^{\kappa} \big] \int_0^1 r^{\kappa} \, dr = \frac{\big[ \|x\|_{\R^d}^{\kappa} + \|y\|_{\R^d}^{\kappa} \big]}{\kappa +1}.
\end{split}
\end{equation}
This and  the hypothesis that $ \forall \, d \in \N, \delta \in (0, 1], x, y \in \R^d  \colon |( \mathcal{R} \phi^{ 0, d }_{ \delta } )(x) - ( \mathcal{R} \phi^{ 0, d }_{ \delta } )(y)| $ $ \leq \kappa d^{\kappa} (1 + \|x\|_{\R^d}^{\kappa} + \|y \|_{\R^d}^{\kappa})\|x-y\|_{\R^d} $ prove that for all $ d \in \N$, $\delta \in (0, 1]$, $x, y \in \R^d$ it holds that
\begin{equation}
\begin{split}
&|\Phi_{\delta}^{0,d} (x) - \Phi_{\delta}^{0,d}(y)|  = |( \mathcal{R} \phi^{ 0, d }_{ \delta } )(x) - ( \mathcal{R} \phi^{ 0, d }_{ \delta } )(y)| \\
& \leq \kappa d^{\kappa} (1 +  \|x\|_{\R^d}^{\kappa} + \|y \|_{\R^d}^{\kappa} )\|x-y\|_{\R^d} \\
& \leq \kappa d^{\kappa} \left( 1 + 2 (\kappa +1) \int_0^1 \big[r \|x\|_{\R^d} + (1-r) \|y\|_{\R^d}\big]^{\kappa} \, dr \right) \|x-y\|_{\R^d}\\
& \leq 2 \kappa (\kappa +1 ) d^{\kappa} \left( 1 +  \int_0^1 \big[r \|x\|_{\R^d} + (1-r) \|y\|_{\R^d}\big]^{\kappa} \, dr \right) \|x-y\|_{\R^d}.
\end{split}
\end{equation}
Proposition~\ref{prop:perturbation_PDE_2} (with 
$d = d$, $m = d $, $ \xi = x $, 
$ T = T $, $c = \kappa$, $ C = \kappa d^{\kappa}$, $\varepsilon_0 = \delta \kappa d^{\kappa}$, $\varepsilon_1 = \delta \kappa d^{\kappa}$, $\varepsilon_2 = 0$, $\varsigma_0 = \kappa$, $\varsigma_1 = \kappa$, $\varsigma_2 = 0$, $L_0  =  2 \kappa (\kappa +1 ) d^{\kappa}$, $L_1 = \kappa$, $\ell = \kappa$,
$ h = \min\{\delta^2, T\} $,
$ B = \mathcal{A}_d$, $ p = 2$, $ q = 2$,
$ \left\| \cdot \right\|  = \left\| \cdot \right\|_{\R^d}$, 
$ ( \Omega, \mathcal{F}, \P ) = ( \Omega, \mathcal{F}, \P ) $, 
$ W = W^{d,1} $, 
$ \phi_0 = \Phi_{\delta}^{0,d}$, 
$ f_1 = f_{1,d} $, 
$ \phi_2 = \id_{\R^d}$, 
$ \chi = \chi_{ \min\{\delta, \sqrt{T} \} }$, 
$ f_0 = f_{0,d} $,
$ \phi_1 = \Phi_{\delta}^{1,d} $,
$ (\varpi_r)_{r \in (0,\infty)} = (\varpi_{d,r})_{r \in (0,\infty)} $, 
$X = X^{d,x}$, $Y = Y^{\delta, d, 1,x}$
for $d \in \N$, $x \in \R^d$, $\delta \in (0, 1]$
in the notation of Proposition~\ref{prop:perturbation_PDE_2}) hence ensures that for all $d \in \N$, $\delta \in (0, 1]$, $x \in \R^d$ it holds that
\begin{equation}
\label{eq:weak:error:pointwise}
\begin{split}
&\big|
\E\big[ 
f_{ 0, d }( X^{ d, x }_T )
\big]
- 
\E\big[
\Phi^{ 0, d }_{ \delta }(
Y^{ \delta, d, 1, x }_T
)
\big]
\big|^p \leq  \left[
2 \delta \kappa d^{\kappa}
+
(  \min\{\delta^2, T\} / T )^{ \nicefrac{ 1 }{ 2 } } 
\right]^p
\\ & \cdot 
e^{
p( 
\kappa + 3 + 2 \kappa +
\left[ 
\kappa
\max\{ 
\kappa, 
\kappa
\} 
+
\kappa
\max\{ 
\kappa ,
1
\}
+
\kappa
\max\{ \kappa, 1 \}
+
2
\right]
T
)
}
\big[ 
\| x \|_{\R^d}
+
\max\{ 1, 0 \}
( 1 + 1 )
\\ &
+ 
\max\{ 1, \kappa d^{\kappa}, \| f_{1,d}( 0 ) \|_{\R^d} \} \max\{ 1, T \} 
+ 
\varpi_{d, \max\{ \kappa, 2\kappa, 2, 2\kappa \} } 
\big]^{ p\max\{ 1, \kappa, \kappa \} + p\kappa }\\
&\cdot \big| \! \max\{ 1, 2\kappa(\kappa+1)d^{\kappa} \}\big|^p\\
&\leq  \left[
2\delta \kappa d^{\kappa}
+
( \delta^2 / T )^{ \nicefrac{ 1 }{ 2 } } 
\right]^p e^{
p( 
3 + 3\kappa +
[ 
\kappa^2
+
2\kappa
\iota
+ 2
]
T
)
} \big| \! \max\{ 1, 2 \kappa (\kappa +1 ) d^{\kappa} \}\big|^p\\
& \cdot
\big[ 
\| x \|_{\R^d}
+
2
+ 
\max\{ 1, \kappa d^{\kappa}, \| f_{1,d}( 0 ) \|_{\R^d} \} \max\{ 1, T \} 
+ 
\varpi_{ d, \max\{  2 \kappa, 2 \} } 
\big]^{ p\iota + p \kappa}
.
\end{split}
\end{equation}
Moreover, note that  Lemma~\ref{l:exp.Gauss} assures that for all $q \in [2, \infty)$, $d \in \N$ it holds that
\begin{equation}
\label{eq:norm:BM}
\begin{split}
\varpi_{ d, q }  \leq \sqrt{(q-1)  \operatorname{Trace}(\mathcal{A}_d^* \mathcal{A}_d) T} = \sqrt{2(q-1)  \operatorname{Trace}(A_d) T} \leq \sqrt{2(q-1)  \kappa d^{\kappa} T}.
\end{split}
\end{equation}
This, \eqref{eq:weak:error:pointwise}, and the fact that $ \forall \, y, z \in \R, \alpha \in [1, \infty) \colon |y + z|^{\alpha}  \leq 2^{\alpha -1}(|y|^{\alpha} + |z|^{\alpha})$ demonstrate that for all  $d \in \N$, $\delta \in (0, 1]$ it holds that
\begin{align*}
\label{eq:diff:p}
&\int_{ \R^d }
\big|
\E\big[ 
f_{ 0, d }( X^{ d, x }_T )
\big]
- 
\E\big[
\Phi^{ 0, d }_{ \delta }(
Y^{ \delta, d, 1, x }_T
)
\big]
\big|^p
\,
\nu_d (dx) \\
& \leq \delta^p  \left[
2 \kappa d^{\kappa}
+
T^{ -\nicefrac{ 1 }{ 2 } } 
\right]^p e^{
p( 
3 + 3\kappa +
[ 
\kappa^2
+
2\kappa
\iota
+ 2
]
T
)
} \big| \! \max\{ 1, 2 \kappa (\kappa +1 ) d^{\kappa} \}\big|^p\\
& \cdot \int_{\R^d}
\Big[ 
\|x\|_{\R^d}
+
2
+ 
\max\{ 1, \kappa d^{\kappa}, \| f_{1,d}( 0 ) \|_{\R^d} \} \max\{ 1, T \} 
+ 
\sqrt{2(2\iota-1)  \kappa d^{\kappa} T}
\Big]^{ p\iota + p \kappa} 
\, \nu_d(dx)\\
& \leq \delta^p  \left[
2 \kappa d^{\kappa}
+
T^{ -\nicefrac{ 1 }{ 2 } } 
\right]^p e^{
	p( 
	3 + 3\kappa +
	[ 
	\kappa^2
	+
	2\kappa
	\iota
	+ 2
	]
	T
	)
} \big| \! \max\{ 1, 2 \kappa (\kappa +1 ) d^{\kappa} \}\big|^p\\
& \cdot 2^{ p\iota + p \kappa -1} \bigg( \Big[2
+ 
\max\{ 1, \kappa d^{\kappa}, \| f_{1,d}( 0 ) \|_{\R^d} \} \max\{ 1, T \} 
+ 
\sqrt{2(2\iota-1)  \kappa d^{\kappa} T}
\Big]^{ p\iota + p \kappa} \\
& + \int_{\R^d} 
\|x\|_{\R^d}^{ p\iota + p \kappa}
\, \nu_d(dx) \bigg). \numberthis
\end{align*}
Next note that the fact that $\iota \leq \kappa +1$ and H\"older's inequality prove that for all $d \in \N$ it 
holds that
\begin{equation}
\begin{split}
\int_{\R^d} 
\|x\|_{\R^d}^{ p\iota + p \kappa}
\, \nu_d(dx)  &\leq \left[ \int_{\R^d} 
\|x\|_{\R^d}^{ p(2\kappa +1)}
\, \nu_d(dx)  \right]^{\nicefrac{(\iota +\kappa)}{(2\kappa +1)}} \\
&\leq [\eta d^{\eta}]^{\nicefrac{(\iota +\kappa)}{(2\kappa +1)}} \leq \eta d^{\eta}.
\end{split}
\end{equation}
Combining this and \eqref{eq:diff:p} ensures that for all  $d \in \N$, $\delta \in (0, 1]$ it holds that 
\begin{align*}
\label{eq:weak:RHS:1}
&2^{p-1} \int_{ \R^d }
\big|
\E\big[ 
f_{ 0, d }( X^{ d, x }_T )
\big]
- 
\E\big[
\Phi^{ 0, d }_{ \delta }(
Y^{ \delta, d, 1, x }_T
)
\big]
\big|^p
\,
\nu_d (dx) \\
& \leq \delta^p  \left[
2 \kappa d^{\kappa}
+
T^{ -\nicefrac{ 1 }{ 2 } } 
\right]^p e^{
	p( 
	3 + 3\kappa +
	[ 
	\kappa^2
	+
	2\kappa
	\iota
	+ 2
	]
	T
	)
} \big| \! \max\{ 1, 2 \kappa (\kappa +1 ) d^{\kappa} \}\big|^p 2^{ p (2\iota + 1) -2} \numberthis \\
& \cdot  \bigg( \Big[2
+ 
\max\{ 1, \kappa d^{\kappa}, \| f_{1,d}( 0 ) \|_{\R^d} \} \max\{ 1, T \} 
+ 
\sqrt{2(2\iota-1)  \kappa d^{\kappa} T}
\Big]^{ p\iota + p \kappa}  + \eta d^{\eta} \bigg).
\end{align*}
Next observe that for all 
$ d \in \N $, $ \delta \in (0,1] $, $ x \in \R^d $
it holds that
\begin{equation}
\label{eq:Phi_0_d_estimate}
\begin{split}
| 
\Phi^{ 0, d }_{ \delta }( x ) 
|
& \leq
| 
\Phi^{ 0, d }_{ \delta }( x ) 
-
f_{ 0, d }( x )
|
+
| 
f_{ 0, d }( x ) 
|
\\ &
\leq
\delta \kappa d^{ \kappa } 
( 1 + \| x \|_{ \R^d }^{ \kappa } )
+
\kappa d^{ \kappa } 
( 1 + \| x \|_{ \R^d }^{ \kappa } )
\leq
2 \kappa d^{ \kappa } 
( 1 + \| x \|_{ \R^d }^{ \kappa } )
.
\end{split}
\end{equation}
Moreover, note that 
Lemma~\ref{lem:sde-lp-bound} 
shows that for all 
$ q \in [1,\infty) $, 
$ \delta \in (0,1] $,
$ d, m \in \N $,
  $ x \in \R^d $
it holds that
\begin{equation}
\begin{split}
\big(
\E\big[
\|
Y^{ \delta, d, m, x }_T
\|^q_{ \R^d }
\big]
\big)^{ \nicefrac{ 1 }{ q } }
& \leq 
\Big(
\| x \|_{ \R^d }
+
\kappa d^{ \kappa } T 
+ 
\big(
\E\big[ 
\| \mathcal{A}_d W^{ d, m }_T \|^q_{ \R^d }
\big]
\big)^{ \nicefrac{ 1 }{ q } }
\Big)
\,
e^{
\kappa T
}
\\ & =
\big(
\|x \|_{ \R^d }
+
\kappa d^{ \kappa } T 
+ 
\varpi_{ d, q }
\big)
\,
e^{
\kappa T
}
.
\end{split}
\end{equation}
This and \eqref{eq:norm:BM} demonstrate that  for all $q \in [2, \infty)$,
$ \delta \in (0,1] $,
$ d, m \in \N $,
  $ x \in \R^d $
it holds that
\begin{equation}
\begin{split}
\big(
\E\big[
\|
Y^{ \delta, d, m, x }_T
\|^q_{ \R^d }
\big]
\big)^{ \nicefrac{ 1 }{ q } } \leq 
\big(
\|x \|_{ \R^d }
+
\kappa d^{ \kappa } T 
+ 
\sqrt{2(q-1) \kappa d^{\kappa} T}
\big)
\,
e^{
\kappa T
}
.
\end{split}
\end{equation}
Combining this with \eqref{eq:Phi_0_d_estimate},
the fact that $ \forall \, y, z \in \R, \alpha \in [1, \infty) \colon |y + z|^{\alpha}  \leq 2^{\alpha -1}(|y|^{\alpha} + |z|^{\alpha})$, and H\"older's inequality ensures 
that for all 
$ \delta \in (0,1] $, $ d \in \N $
it holds that
\begin{equation}
\label{eq:average:expactation}
\begin{split}
&  \int_{ \R^d }
\E\Big[
\big|
\Phi^{ 0, d }_{ \delta }(
Y^{ \delta, d, 1, x }_T
)
\big|^p
\Big]
\,
\nu_d (dx)
\leq
\big( 2 \kappa d^{ \kappa } \big)^{p}
 \int_{\R^d} \E  \Big[ \Big( 1+ \|
Y^{ \delta, d, 1, x }_T
\|^{ \kappa }_{ \R^d } \Big)^p 
\Big]
\, \nu_d (dx)
\\ &
\leq
\big( 2 \kappa d^{ \kappa } \big)^{p} 2^{p-1}
\int_{\R^d} \E  \Big[  1+ \|
Y^{ \delta, d, 1, x }_T
\|^{ p\kappa }_{ \R^d } 
\Big]
\, \nu_d (dx)
\\ &
\leq
\big( 4 \kappa d^{ \kappa } \big)^{p} 
\left( 1+ \int_{\R^d}  \E  \Big[   \|
Y^{ \delta, d, 1, x }_T
\|^{ p\kappa }_{ \R^d } 
\Big]
\, \nu_d (dx) \right)
\\ &
\leq
\big( 4 \kappa d^{ \kappa } \big)^{p} 
\left( 1+ \int_{\R^d} \Big| \E  \Big[   \|
Y^{ \delta, d, 1, x }_T
\|^{ p\iota }_{ \R^d } 
\Big] \Big|^{\nicefrac{\kappa}{\iota}}
\, \nu_d (dx) \right)
\\ &
\leq
\big( 4 \kappa d^{ \kappa } \big)^{p} 
\left( 1+ \int_{\R^d}   \Big[\big(
\|x \|_{ \R^d }
+
\kappa d^{ \kappa } T 
+ 
\sqrt{2(p \iota-1) \kappa d^{\kappa} T}
\big)
\,
e^{
	\kappa T
}
\Big]^{p \kappa}
\, \nu_d (dx) \right).
\end{split}
\end{equation}
The fact that  $ \forall \, y, z \in \R, \alpha \in [0, \infty) \colon |y + z|^{\alpha}  \leq 2^{\alpha }(|y|^{\alpha} + |z|^{\alpha})$ hence proves that
for all 
$ \delta \in (0,1] $, $ d \in \N $
it holds that
\begin{align*}
\label{eq:Phi:p}
& \int_{ \R^d }
\E\Big[
\big|
\Phi^{ 0, d }_{ \delta }(
Y^{ \delta, d, 1, x }_T
)
\big|^p
\Big]
\,
\nu_d (dx) \numberthis \\
& \leq 
\big( 4 \kappa d^{ \kappa } \big)^{p} 
\left( 1+ 2^{p\kappa} \,
e^{
	p \kappa^2 T
}   \left[ \int_{\R^d} 
\|x \|_{ \R^d }^{p \kappa}
\, \nu_d (dx)
+
 \Big( \kappa d^{ \kappa } T 
+ 
\sqrt{2(p \iota-1) \kappa d^{\kappa} T}
\Big)^{p \kappa} \right] \right)\\
& \leq 
\big( 4 \kappa d^{ \kappa } \big)^{p} 2^{p\kappa} \,
e^{
	p \kappa^2 T
}
\left[ 1+  \Big( \kappa d^{ \kappa } T 
+ 
\sqrt{2(p \iota-1) \kappa d^{\kappa} T}
\Big)^{p \kappa} 
+   \int_{\R^d} 
\|x \|_{ \R^d }^{p \kappa}
\, \nu_d (dx) \right].
\end{align*}
Next note that H\"older's inequality shows that for all $d \in \N$ it holds that
\begin{equation}
\begin{split}
\int_{\R^d} 
\|x\|_{\R^d}^{  p \kappa}
\, \nu_d(dx)  &\leq \left[ \int_{\R^d} 
\|x\|_{\R^d}^{ p(2\kappa +1)}
\, \nu_d(dx)  \right]^{\nicefrac{\kappa}{(2\kappa +1)}} \\
&\leq [\eta d^{\eta}]^{\nicefrac{\kappa}{(2\kappa +1)}} \leq \eta d^{\eta}.
\end{split}
\end{equation}
Combining this and \eqref{eq:Phi:p} ensures 
 that for all 
$ \delta \in (0,1] $, $ d \in \N $
it holds that
\begin{align*}
& 2^{2p-1} p^p \int_{ \R^d }
\E\Big[
\big|
\Phi^{ 0, d }_{ \delta }(
Y^{ \delta, d, 1, x }_T
)
\big|^p
\Big]
\,
\nu_d (dx) \numberthis \\
& \leq 
\frac{
\big( 2^{\kappa+4} p \kappa d^{ \kappa } e^{
	 \kappa^2 T
} \big)^{p}}{2}
\left[ 1+  \Big( \kappa d^{ \kappa } T 
+ 
\sqrt{2 (p \iota-1) \kappa d^{\kappa} T}
\Big)^{p \kappa} 
+   \eta d^{\eta} \right].
\end{align*}
This, \eqref{eq:weak:RHS}, and \eqref{eq:weak:RHS:1} prove that for all $d \in \N$, $\varepsilon \in (0, 1]$ it holds that
\begin{equation}
\label{eq:lp:norm}
\begin{split}
\int_{ \R^d }
\E\Bigg[
\bigg|
u_d(T,x) 
- 
\frac{ 1 }{ \mathfrak{M}_{d, \varepsilon} } 
\bigg[ 
\textstyle
\sum\limits_{ m = 1 }^{\mathfrak{M}_{d, \varepsilon}}
\Phi^{ 0, d }_{ \mathcal{D}_{d, \varepsilon} }\Big(
Y^{ \mathcal{D}_{d, \varepsilon}, d, m, x }_T
\Big)
\bigg]
\bigg|^p
\Bigg]
\,
\nu_d(dx) \leq \frac{\varepsilon^p}{4} + \frac{\varepsilon^p}{2} < \varepsilon^p
.
\end{split}
\end{equation}
Corollary~\ref{cor:random_field} therefore assures that for all $d \in \N$, $\varepsilon \in (0, 1]$ there exists $\mathfrak{w}_{d, \varepsilon} \in \Omega$ such that
\begin{equation}
\label{eq:bar:omega}
\begin{split}
\left[
\int_{ \R^d }
\bigg|
u_d(T,x) 
- 
\frac{ 1 }{ \mathfrak{M}_{d, \varepsilon} } 
\bigg[ 
\textstyle
\sum\limits_{ m = 1 }^{\mathfrak{M}_{d, \varepsilon}}
\big(\mathcal{R} \phi^{ 0, d }_{ \mathcal{D}_{d, \varepsilon} } \big)\Big(
Y^{ \mathcal{D}_{d, \varepsilon}, d, m, x }_T (\mathfrak{w}_{d, \varepsilon})
\Big)
\bigg]
\bigg|^p
\,
\nu_d(dx)  \right]^{ \nicefrac{ 1 }{ p } } < \varepsilon
.
\end{split}
\end{equation}
Moreover, note that  for all $d \in \N$, $\varepsilon \in (0, 1]$ it holds that
\begin{align*}
\mathfrak{M}_{d, \varepsilon} & \leq  \Big(\tfrac{2^{(\kappa+4)}p \kappa d^{\kappa} e^{\kappa^2 T}}{\varepsilon} \Big)^2 
\Big[ 1+  \big( \kappa d^{ \kappa } T 
+ 
\sqrt{2(p \iota-1) \kappa d^{\kappa} T}
\big)^{p \kappa} 
+  \eta d^{\eta} \Big]^{\nicefrac{2}{p}} + 1 \\
& = 2^{2(\kappa+4)} p^2 \kappa^2 d^{2\kappa} e^{2\kappa^2 T}  
\Big[ 1+  \big( \kappa d^{ \kappa } T 
+ 
\sqrt{2(p \iota-1) \kappa d^{\kappa} T}
\big)^{p \kappa} 
+  \eta d^{\eta} \Big]^{\nicefrac{2}{p}} \varepsilon^{-2} + 1\\
& \leq 2^{2(\kappa+4)} p^2 \kappa^2 d^{2\kappa} e^{2\kappa^2 T}  
\Big[ 1+  \big( \kappa d^{ \kappa } T 
+ 
\sqrt{2(p \iota-1) \kappa d^{\kappa} T}
\big)^{p \kappa} 
+  \eta d^{\eta} \Big] \varepsilon^{-2} + 1\\
& \leq 2^{2(\kappa+4)} p^2 \kappa^2 d^{2\kappa} e^{2\kappa^2 T}  
\Big[ 1+  \big(  \iota d^{ \iota } T 
+ p \iota d^{\iota} \sqrt{T}
\big)^{p \kappa} 
+  \eta d^{\eta} \Big] \varepsilon^{-2} + 1. \numberthis
\end{align*}
Hence, we obtain that for all $d \in \N$, $\varepsilon \in (0, 1]$ it holds that
\begin{equation}
\label{eq:bound:M}
\begin{split}
\mathfrak{M}_{d, \varepsilon} &  \leq 2^{2(\kappa+4)} p^2 \kappa^2 d^{2\kappa} e^{2\kappa^2 T}  
\big[ 1+  \big| 2 p  \iota d^{ \iota } \max\{1, T\}|^{p \kappa} 
+  \eta d^{\eta} \big] \varepsilon^{-2} + 1 \\
& \leq 2^{2(\kappa+4)} p^2 \kappa^2 d^{2\kappa} e^{2\kappa^2 T} d^{\max\{p \kappa \iota, \eta\}} 
\big[ 1+  | 2 p  \iota \max\{1, T\}|^{p \kappa} 
+  \eta \big] \varepsilon^{-2} + 1\\
& \leq 2^{2(\kappa+4)} p^2 \kappa^2  e^{2\kappa^2 T}
\big[ 1+  | 2 p  \iota \max\{1, T\}|^{p \kappa} 
+  \eta \big]d^{p \kappa \iota + \eta + 2\kappa}  \varepsilon^{-2} + 1\\
& \leq 2^{2(\kappa+4)} p^2 \iota^2  e^{2\kappa^2 T}
\big[ 1+  | 2 p  \iota \max\{1, T\}|^{p \kappa} 
+  \eta \big]d^{p \kappa \iota + \eta + 2\kappa}  \varepsilon^{-2} + 1\\
&  \leq 2^{2(\kappa+4)} p^2 \iota^2  e^{2\kappa^2 T}
\big[ 2+  | 2 p  \iota \max\{1, T\}|^{p \kappa} 
+  \eta \big]d^{p \kappa \iota + \eta + 2\kappa}  \varepsilon^{-2}.
\end{split}
\end{equation}
Next observe that for all $d \in \N$, $\varepsilon \in (0, 1]$ it holds that
\begin{equation}
\begin{split}
\| f_{1,d}( 0 ) \|_{\R^d} &\leq \| 
f_{ 1, d }(0) 
- 
( \mathcal{R} \phi^{ 1, d }_{ \varepsilon } )(0)
\|_{ \R^d } + \|( \mathcal{R} \phi^{ 1, d }_{ \varepsilon } )(0)
\|_{ \R^d } \\
& \leq \varepsilon \kappa d^{\kappa} + \kappa d^{\kappa} \leq 2 \kappa d^{\kappa} \leq 2 \iota d^{\kappa}.
\end{split}
\end{equation}
Therefore, we obtain that for all $d \in \N$, $\varepsilon \in (0,1]$ it holds that
\begin{align*}
&\mathcal{D}_{d, \varepsilon} = \varepsilon \! \left[
\max \{2 \kappa d^{\kappa}, 1 \}
+
T^{ -\nicefrac{ 1 }{ 2 } } 
\right]^{-1} e^{
	-( 
	3 + 3\kappa +
	[ 
	\kappa^2
	+
	2\kappa
	\iota
	+ 2
	]
	T
	)
} \big| \! \max\{ 1, 2 \kappa (\kappa +1 ) d^{\kappa} \}\big|^{-1} 2^{ -(2\iota +  1) } \\
& \cdot \bigg( \Big[2
+ 
\max\{ 1, \kappa d^{\kappa}, \| f_{1,d}( 0 ) \|_{\R^d} \} \max\{ 1, T \} 
+ 
\sqrt{2(2\iota-1)  \kappa d^{\kappa} T}
\Big]^{ p\iota + p \kappa}  +  \eta d^{\eta} \bigg)^{\nicefrac{-1}{p}} \\
& \geq \varepsilon \! \left[
\max \{2 \kappa d^{\kappa}, 1 \}
+
T^{ -\nicefrac{ 1 }{ 2 } } 
\right]^{-1} e^{
	-( 
	3 + 3\kappa +
	[ 
	\kappa^2
	+
	2\kappa
	\iota
	+ 2
	]
	T
	)
} \big| \! \max\{ 1, 2 \kappa (\kappa +1 ) d^{\kappa} \}\big|^{-1}  2^{ -(2\iota +  1) }\\
& \cdot \bigg( \Big[2
+ 
\max\{ 1, \kappa d^{\kappa}, 2\iota d^{\kappa} \} \max\{ 1, T \} 
+ 
\sqrt{2 (2\iota-1)  \kappa d^{\kappa} T}
\Big]^{ p\iota + p \kappa}  +  \eta d^{\eta} \bigg)^{\nicefrac{-1}{p}} \\
& \geq \varepsilon \! \left[
2\max \{2 \kappa d^{\kappa}, 1, T^{ -\nicefrac{ 1 }{ 2 } }  \}
\right]^{-1} e^{
-( 
3 + 3\iota +
[ 
3\iota^2
+ 2
]
T
)
} \big| \!\max\{ 1, 2 \kappa (\kappa +1 ) d^{\kappa} \}\big|^{-1} 2^{ -(2\iota +  1) }\\
& \cdot \bigg( \Big[2
+ 
2 \iota d^{\kappa} \max\{ 1, T \} 
+ 2 \iota d^{\kappa} \sqrt{T}
\Big]^{ p\iota + p \kappa}  +  \eta d^{\eta} \bigg)^{\nicefrac{-1}{p}}. \numberthis
\end{align*}
This proves that for all $d \in \N$, $\varepsilon \in (0,1]$ it holds that
\begin{equation}
\begin{split}
\mathcal{D}_{d, \varepsilon} &\geq \varepsilon \big|
4 \iota d^{\kappa} \max \{ 1, T^{ -\nicefrac{ 1 }{ 2 } }  \}
\big|^{-1} e^{
	-( 3\iota^2+3)(T+1)
} \, \big|  4 \iota^2 d^{\kappa} \big|^{-1} 2^{ -(2\iota +  1) }\\
& \cdot  \Big( \big[
6\iota d^{\kappa}  \max\{ 1, T \} 
\big]^{ p\iota + p \kappa}  +  \eta d^{\eta} \Big)^{\nicefrac{-1}{p}}\\
& \geq \varepsilon \big|
4 \iota d^{\kappa} \max \{ 1, T^{ -\nicefrac{ 1 }{ 2 } }  \}
\big|^{-1} e^{
	-( 3\iota^2+3)(T+1)
} \, \big|  4 \iota^2 d^{\kappa} \big|^{-1} 2^{ -(2\iota +  1) }\\
& \cdot  \Big( \big[
6\iota \max\{ 1, T \} 
\big]^{ p\iota + p \kappa}  +  \eta \Big)^{\nicefrac{-1}{p}} d^{\nicefrac{[-\kappa(p \iota + p \kappa) - \eta]}{p}}\\
& \geq \big|
4 \iota \max \{ 1, T^{ -\nicefrac{ 1 }{ 2 } }  \}
\big|^{-1} e^{
	-( 3\iota^2+3)(T+1)
} \, \big|  4 \iota^2  \big|^{-1}  2^{ -(2\iota +  1) } \\
& \cdot   \big( [
6 \iota   \max\{ 1, T \} 
]^{ p\iota + p \kappa}  +  \eta \big)^{\nicefrac{-1}{p}} d^{- (2\kappa + \kappa(\kappa + \iota) + \eta )} \varepsilon.
\end{split}
\end{equation}
Hence, we obtain that for all $d \in \N$, $\varepsilon \in (0, 1]$ it holds that
\begin{equation}
\label{eq:bound:delta}
\begin{split}
\mathcal{D}_{d, \varepsilon} &\geq |\!
 \min \{ 1, T^{ \nicefrac{ 1 }{ 2 } }  \}
| e^{
	-( 3\iota^2+3)(T+1)
} \, \iota^{-3}  2^{ -(2\iota +  5) }     \\
& \cdot \big( [
6 \iota  \max\{ 1, T \} 
]^{ p\iota + p \kappa}  +  \eta \big)^{\nicefrac{-1}{p}} d^{- (\kappa( 2+ \kappa + \iota) + \eta )} \varepsilon.
\end{split}
\end{equation}
Moreover, note that Proposition~\ref{prop:sum:comp:ANN} ensures that for all $\delta \in (0, 1]$, $d \in \N$, $t \in [0,T]$, $\omega \in \Omega$ there exist $(\psi_{ \delta, d,m, t, \omega})_{m \in \N} \subseteq \mathcal{N}$ such that for all $x \in \R^d$, $m \in \N$ it holds that $\mathcal{R}(\psi_{ \delta, d,m, t, \omega}) \in C(\R^d, \R^d)$, $\mathcal{P}(\psi_{ \delta, d,m, t, \omega}) \leq \mathcal{P}(\phi^{2,d})+  (\mathcal{P}(\phi^{2,d}) + \mathcal{P}(\phi^{1,d}_{\delta}))^3 [\chi_{ \delta }(t) /\delta^2 +1] $, $\mathcal{L}(\psi_{ \delta, d,m, t, \omega})= \mathcal{L}(\psi_{ \delta, d,1, t, \omega})$, and 
\begin{equation}
\label{eq:sum:comp:ANN}
(\mathcal{R} \psi_{ \delta, d,m, t, \omega})(x) = Y^{ \delta, d, m, x }_t (\omega).
\end{equation}
This demonstrates that for all  $\delta \in (0, 1]$, $d, m \in \N$,  $\omega \in \Omega$ it holds that 
\begin{equation}
\label{eq:psi:bound}
\begin{split}
\mathcal{P}(\psi_{ \delta, d,m, T, \omega}) &\leq \big(\mathcal{P}(\phi^{2,d}) + \mathcal{P}(\phi^{1,d}_{\delta})\big)^3 \big[\chi_{ \delta }(T) /\delta^2 +2\big]\\
& \leq \big(\kappa d^{\kappa} \delta^{-\kappa}\big)^3 \big[T /\delta^2 +2\big] \leq \kappa^3 (T+2) d^{3\kappa} \delta^{-3\kappa-2}.
\end{split}	
\end{equation}
Proposition~\ref{prop:composition:ANN} hence proves that for all $\delta \in (0, 1]$, $d \in \N$,  $\omega \in \Omega$  there exist $(\varphi_{\delta, d, m, \omega})_{m \in \N} \subseteq \mathcal{N}$ such that for all $x \in \R^d$, $m \in \N$ it holds that $\mathcal{R}(\varphi_{\delta, d, m, \omega}) \in C(\R^d, \R)$, $\mathcal{P}(\varphi_{\delta, d, m, \omega}) \leq \mathcal{P}(\phi^{2, d})(\mathcal{P}(\phi^{0, d}_{\delta}) +  \mathcal{P}(\psi_{ \delta, d,m, T, \omega}) )$, $\mathcal{L}(\varphi_{\delta, d, m, \omega})= \mathcal{L}(\varphi_{\delta, d, 1, \omega})$, and 
\begin{equation}
\label{eq:bar:psi}
(\mathcal{R} \varphi_{\delta, d, m, \omega})(x) = \Phi^{0,d}_{\delta}(Y^{\delta,d,m,x}_T(\omega)).
\end{equation}
This and \eqref{eq:psi:bound} ensure that for all  $\delta \in (0, 1]$, $d, m \in \N$,  $\omega \in \Omega$ it holds that
\begin{equation}
\begin{split}
\mathcal{P}(\varphi_{\delta, d, m, \omega}) &\leq \kappa d^{\kappa} \big[\kappa d^{\kappa} \delta^{-\kappa} +  \kappa^3 (T+2) d^{3\kappa} \delta^{-3\kappa-2} \big]\\
& \leq \kappa^2 d^{4\kappa} \delta^{-3\kappa-2}(T+2) [1+\kappa^2] \leq 2\iota^2 \kappa^2 (T+2) d^{4\kappa} \delta^{-3\kappa-2}.
\end{split}
\end{equation}
Lemma~\ref{lem:sum:ANN} and \eqref{eq:bar:psi} therefore show that for all $M, d \in \N$, $\delta \in (0,1]$, $\omega \in \Omega$ there exists $\Psi_{M, d, \delta, \omega} \in \mathcal{N}$ such that for all $x \in \R^d$ it holds that
$\mathcal{R}(\Psi_{M, d, \delta, \omega}) \in C(\R^d, \R)$, $\mathcal{P}(\Psi_{M, d, \delta, \omega}) \leq 2 M^2 \iota^2 \kappa^2 (T+2) d^{4\kappa} \delta^{-3\kappa-2}$, and 
\begin{equation}
(\mathcal{R}\Psi_{M, d, \delta, \omega})(x) = \frac{1}{M} \sum_{m=1}^M \Phi^{0,d}_{\delta}(Y^{\delta,d,m,x}_T(\omega)).
\end{equation}
This, \eqref{eq:bound:M}, and \eqref{eq:bound:delta} assure that for all $d \in \N$, $\varepsilon \in (0, 1]$, $\omega \in \Omega$ it holds that
\begin{align*}
&\mathcal{P}(\Psi_{\mathfrak{M}_{d, \varepsilon}, d, \mathcal{D}_{d, \varepsilon}, \omega}) \leq 2 |\mathfrak{M}_{d, \varepsilon}|^2 \iota^2 \kappa^2 (T+2) d^{4\kappa} |\mathcal{D}_{d, \varepsilon}|^{-3\kappa-2} \\
& \leq 2 \Big( 2^{2(\kappa+4)} p^2 \iota^2  e^{2\kappa^2 T}
\big[ 2+  | 2 p  \iota \max\{1, T\}|^{p \kappa} 
+  \eta \big]d^{p \kappa \iota + \eta + 2\kappa}  \varepsilon^{-2} \Big)^2   \\
& \cdot \iota^2 \kappa^2 (T+2) d^{4\kappa} \Big[ |\!
\min \{ 1, T^{ \nicefrac{ 1 }{ 2 } }  \}
| e^{
	-( 3\iota^2+3)(T+1)
} \, \iota^{-3}  2^{ -(2\iota +  5) }     \\
& \cdot \big( [
6 \iota  \max\{ 1, T \} 
]^{ p\iota + p \kappa}  +  \eta \big)^{\nicefrac{-1}{p}} d^{- (\kappa( 2+ \kappa + \iota) + \eta )} \varepsilon \Big]^{-3\kappa-2} \numberthis\\
& = 2 \Big( 2^{2(\kappa+4)} p^2 \iota^2  e^{2\kappa^2 T}
\big[ 2+  | 2 p  \iota \max\{1, T\}|^{p \kappa} 
+  \eta \big]   \Big)^2   \iota^2 \kappa^2 (T+2) \\
& \cdot  \Big[  |\!
\min \{ 1, T^{ \nicefrac{ 1 }{ 2 } }  \}
| e^{
	-( 3\iota^2+3)(T+1)
} \, \iota^{-3}  2^{ -(2\iota +  5) }  \big( [
6\iota   \max\{ 1, T \} 
]^{ p\iota + p \kappa}  +  \eta \big)^{\nicefrac{-1}{p}}  \Big]^{-3\kappa-2}   \\
& \cdot  d^{2(p \kappa \iota + \eta + 4\kappa) + ( \kappa( 2+\kappa + \iota) + \eta ) (3\kappa +2)} \varepsilon^{-3\kappa-6}.
\end{align*}
Combining this and \eqref{eq:bar:omega} finishes the proof of  item~\eqref{item:main_statement_Lp}. 
The proof of Proposition~\ref{prop:PDE_approx_Lp} is thus completed.
\end{proof}



\begin{cor}
	\label{cor:PDE_approx_Lp_gen}
Let 
$ T, \kappa, \eta, p \in (0,\infty) $, 
let 
$
A_d = ( a_{ d, i, j } )_{ (i, j) \in \{ 1, \dots, d \}^2 } \in \R^{ d \times d }
$,
$ d \in \N $,
be symmetric positive semidefinite matrices, 
for every $ d \in \N $ 
let 
$
\left\| \cdot \right\|_{ \R^d } \colon \R^d \to [0,\infty)
$
be the $ d $-dimensional Euclidean norm
and let $\nu_d  \colon \mathcal{B}(\R^d) \to [0,1] $ be a probability measure on $\R^d$,
let
$ f_{0,d} \colon \R^d \to \R $, $ d \in \N $,
and
$ f_{ 1, d } \colon \R^d \to \R^d $,
$ d \in \N $,
be functions,
let 
$ \mathbf{A}_d \colon \R^d \to \R^d $, 
$ d \in \N $, 
and 
$ \mathbf{a} \colon \R \to \R$
be  continuous functions 
which satisfy 
for all 
$
d \in \N
$,
$ x = ( x_1, \dots, x_d ) \in \R^d $
that
$ 
\mathbf{A}_d(x)
=
( \mathbf{a}(x_1), \ldots, \mathbf{a}(x_d) )
$,
let 
\begin{equation}
\mathcal{N}
=
\cup_{ L \in \{ 2, 3, 4, \dots \} }
\cup_{ ( l_0, l_1, \ldots, l_L ) \in \N^{ L + 1 } }
(
\times_{ n = 1 }^L 
(
\R^{ l_n \times l_{ n - 1 } } \times \R^{ l_n } 
)
)
,
\end{equation}
let 
$
\mathcal{P}
\colon \mathcal{N} \to \N
$,
$ \mathcal{L}
\colon \mathcal{N} \to \cup_{ L \in \{ 2, 3, 4, \dots \} } \N^{L+1}
$,
and
$
\mathcal{R} \colon 
\mathcal{N} 
\to 
\cup_{ k, l \in \N } C( \R^k, \R^l )
$
be the functions which satisfy 
for all 
$ L \in \{ 2, 3, 4, \dots \} $, 
$ l_0, l_1, \ldots, l_L \in \N $, 
$ 
\Phi = ((W_1, B_1), \ldots, $ $ (W_L, B_L)) \in 
( \times_{ n = 1 }^L (\R^{ l_n \times l_{n-1} } \times \R^{ l_n } ) )
$,
$ x_0 \in \R^{l_0} $, 
$ \ldots $, 
$ x_{ L - 1 } \in \R^{ l_{ L - 1 } } $ 
with 
$ 
\forall \, n \in \N \cap [1,L) \colon 
x_n = \mathbf{A}_{ l_n }( W_n x_{ n - 1 } + B_n )
$
that 
$
\mathcal{P}( \Phi )
=
\textstyle
\sum\nolimits_{
	n = 1
}^L
l_n ( l_{ n - 1 } + 1 )
$,
$
\mathcal{R}(\Phi) \in C( \R^{ l_0 } , \R^{ l_L } )
$, $\mathcal{L}(\Phi) = ( l_0, l_1, \ldots, l_L ) $,
and
\begin{equation}
( \mathcal{R} \Phi )( x_0 ) = W_L x_{L-1} + B_L ,
\end{equation}
let 
$
( \phi^{ m, d }_{ \varepsilon } )_{ 
	(m, d, \varepsilon) \in \{ 0, 1 \} \times \N \times (0,1] 
} 
\subseteq \mathcal{N}
$,
$
( \phi^{ 2, d } )_{ 
	d \in \N
} 
\subseteq \mathcal{N}
$,
and 
assume for all
$ d \in \N $, 
$ \varepsilon \in (0,1] $, 
$ 
x, y \in \R^d
$
that
$
\mathcal{R}( \phi^{ 0, d }_{ \varepsilon } )
\in 
C( \R^d, \R )
$,
$
\mathcal{R}( \phi^{ 1, d }_{ \varepsilon } )
,
\mathcal{R}( \phi^{ 2, d } )
\in
C( \R^d, \R^d )
$,
$
|
f_{ 0, d }( x )
| 
+
\sum_{ i , j = 1 }^d
| a_{ d, i, j } |
\leq 
\kappa d^{ \kappa }
( 1 + \| x \|^{ \kappa }_{ \R^d } )
$,
$
\| 
f_{ 1, d }( x ) 
- 
f_{ 1, d }( y )
\|_{ \R^d }
\leq 
\kappa 
\| x - y \|_{ \R^d } 
$,
$
\|
( \mathcal{R} \phi^{ 1, d }_{ \varepsilon } )(x)    
\|_{ \R^d }	
\leq 
\kappa ( d^{ \kappa } + \| x \|_{ \R^d } )
$,
$
\mathcal{P}( \phi^{ 2, d } ) +
\sum_{ m = 0 }^1
\mathcal{P}( \phi^{ m, d }_{ \varepsilon } ) 
\leq \kappa d^{ \kappa } \varepsilon^{ - \kappa }
$, $ |( \mathcal{R} \phi^{ 0, d }_{ \varepsilon } )(x) - ( \mathcal{R} \phi^{ 0, d }_{ \varepsilon } )(y)| \leq \kappa d^{\kappa} (1 + \|x\|_{\R^d}^{\kappa} + \|y \|_{\R^d}^{\kappa})\|x-y\|_{\R^d}$, $\mathcal{L}(\phi^{2,d}) \in \N^3 $, $(\mathcal{R}\phi^{2,d})(x) = x$,
$ \int_{\R^d} 
\|z \|_{ \R^d }^{\max\{p, 2\} (2\kappa+1)} \allowbreak
\, \nu_d (dz) \leq \eta d^{\eta}$,
and
\begin{equation}
| 
f_{ 0, d }(x) 
- 
( \mathcal{R} \phi^{ 0, d }_{ \varepsilon } )(x)
|
+
\| 
f_{ 1, d }(x) 
- 
( \mathcal{R} \phi^{ 1, d }_{ \varepsilon } )(x)
\|_{ \R^d }
\leq 
\varepsilon \kappa d^{ \kappa }
(
1 + \| x \|^{ \kappa }_{ \R^d }
)
.
\end{equation}
Then 
\begin{enumerate}[(i)]
	\item 
	there exist unique 
	at most polynomially growing functions 
	$ u_d \colon [0,T] \times \R^{d} \to \R $,
	$ d \in \N $, 
	such that
	for all $ d \in \N $, $ x \in \R^d $ it holds that
	$ u_d( 0, x ) = f_{ 0, d }( x ) $
	and such that for all $ d \in \N $ 
	it holds that
	$ u_d $ is a viscosity solution of
	\begin{equation}
	\begin{split}
	( \tfrac{ \partial }{\partial t} u_d )( t, x ) 
	& = 
	( \tfrac{ \partial }{\partial x} u_d )( t, x )
	\,
	f_{ 1, d }( x )
	+
	\sum_{ i, j = 1 }^d
	a_{ d, i, j }
	\,
	( \tfrac{ \partial^2 }{ \partial x_i \partial x_j } u_d )( t, x )
	\end{split}
	\end{equation}
	for $ ( t, x ) \in (0,T) \times \R^d $
	and 
	\item
	there exist
	$
	( 
	\psi_{ d, \varepsilon } 
	)_{ (d , \varepsilon)  \in \N \times (0,1] } \subseteq \mathcal{N}
	$,
	$
	c \in \R
	$
	such that
	for all 
	$
	d \in \N 
	$,
	$
	\varepsilon \in (0,1] 
	$
	it holds that
	$
	\mathcal{P}( \psi_{ d, \varepsilon } ) 
	\leq
	c \, d^c \varepsilon^{ - c } 
	$,
	$
	\mathcal{R}( \psi_{ d, \varepsilon } )
	\in C( \R^{ d }, \R )
	$,
	and
	\begin{equation}
	\left[
	\int_{ \R^d }
	|
	u_d(T,x) - ( \mathcal{R} \psi_{ d, \varepsilon } )( x )
	|^p
	\,
	\nu_d(dx)
	\right]^{ \nicefrac{ 1 }{ p } }
	\leq
	\varepsilon 
	.
	\end{equation}
\end{enumerate}
\end{cor}

\subsection{Rectified DNN approximations}
\label{sec:main_result}


\begin{theorem}
\label{thm:PDE_approx_Lp}
Let 
$ T, \kappa, \eta, p \in (0,\infty) $,  
let 
$
A_d = ( a_{ d, i, j } )_{ (i, j) \in \{ 1, \dots, d \}^2 } \in \R^{ d \times d }
$,
$ d \in \N $,
be symmetric positive semidefinite matrices, 
for every $ d \in \N $ 
let 
$
\left\| \cdot \right\|_{ \R^d } \colon \R^d \to [0,\infty)
$
be the $ d $-dimensional Euclidean norm
and let $\nu_d  \colon \mathcal{B}(\R^d) \to [0,1] $ be a probability measure on $\R^d$,
let
$ f_{0,d} \colon \R^d \to \R $, $ d \in \N $,
and
$ f_{ 1, d } \colon \R^d \to \R^d $,
$ d \in \N $,
be  functions,
let 
$ \mathbf{A}_d \colon \R^d \to \R^d $, 
$ d \in \N $, 
be the functions 
which satisfy 
for all 
$
d \in \N
$,
$ x = ( x_1, \dots, x_d ) \in \R^d $
that
$ 
\mathbf{A}_d(x)
=
( \max\{x_1, 0\}, \ldots, \max\{x_d, 0\} )
$,
let 
\begin{equation}
\mathcal{N}
=
\cup_{ L \in \{ 2, 3, 4, \dots \} }
\cup_{ ( l_0, l_1, \ldots, l_L ) \in \N^{ L + 1 } }
(
\times_{ n = 1 }^L 
(
\R^{ l_n \times l_{ n - 1 } } \times \R^{ l_n } 
)
)
,
\end{equation}
let 
$
\mathcal{P}
\colon \mathcal{N} \to \N
$,
$ \mathcal{L}
\colon \mathcal{N} \to \cup_{ L \in \{ 2, 3, 4, \dots \} } \N^{L+1}
$,
and
$
\mathcal{R} \colon 
\mathcal{N} 
\to 
\cup_{ k, l \in \N } C( \R^k, \R^l )
$
be the functions which satisfy 
for all 
$ L \in \{ 2, 3, 4, \dots \} $, 
$ l_0, l_1, \ldots, l_L \in \N $, 
$ 
\Phi = ((W_1, B_1), \ldots, $ $ (W_L, B_L)) \in 
( \times_{ n = 1 }^L (\R^{ l_n \times l_{n-1} } \times \R^{ l_n } ) )
$,
$ x_0 \in \R^{l_0} $, 
$ \ldots $, 
$ x_{ L - 1 } \in \R^{ l_{ L - 1 } } $ 
with 
$ 
\forall \, n \in \N \cap [1,L) \colon 
x_n = \mathbf{A}_{ l_n }( W_n x_{ n - 1 } + B_n )
$
that 
$
\mathcal{P}( \Phi )
=
\textstyle
\sum\nolimits_{
n = 1
}^L
l_n ( l_{ n - 1 } + 1 )
$,
$
\mathcal{R}(\Phi) \in C( \R^{ l_0 } , \R^{ l_L } )
$, $\mathcal{L}(\Phi) = ( l_0, l_1, \ldots, l_L ) $,
and
\begin{equation}
( \mathcal{R} \Phi )( x_0 ) = W_L x_{L-1} + B_L ,
\end{equation}
let 
$
( \phi^{ m, d }_{ \varepsilon } )_{ 
(m, d, \varepsilon) \in \{ 0, 1 \} \times \N \times (0,1] 
} 
\subseteq \mathcal{N}
$,
and
assume for all
$ d \in \N $, 
$ \varepsilon \in (0,1] $, 
$ 
x, y \in \R^d
$
that
$
\mathcal{R}( \phi^{ 0, d }_{ \varepsilon } )
\in 
C( \R^d, \R )
$,
$
\mathcal{R}( \phi^{ 1, d }_{ \varepsilon } )
\in
C( \R^d, \R^d )
$,
$
|
f_{ 0, d }( x )
| 
+
\sum_{ i , j = 1 }^d
| a_{ d, i, j } |
\leq 
\kappa d^{ \kappa }
( 1 + \| x \|^{ \kappa }_{ \R^d } )
$,
$
\| 
f_{ 1, d }( x ) 
- 
f_{ 1, d }( y )
\|_{ \R^d }
\leq 
\kappa 
\| x - y \|_{ \R^d } 
$,
$
\|
( \mathcal{R} \phi^{ 1, d }_{ \varepsilon } )(x)    
\|_{ \R^d }	
\leq 
\kappa ( d^{ \kappa } + \| x \|_{ \R^d } )
$,
$
\sum_{ m = 0 }^1
\mathcal{P}( \phi^{ m, d }_{ \varepsilon } ) 
\leq \kappa d^{ \kappa } \varepsilon^{ - \kappa }
$, $ |( \mathcal{R} \phi^{ 0, d }_{ \varepsilon } )(x) - ( \mathcal{R} \phi^{ 0, d }_{ \varepsilon } )(y)| \leq \kappa d^{\kappa} (1 + \|x\|_{\R^d}^{\kappa} + \|y \|_{\R^d}^{\kappa})\|x-y\|_{\R^d}$, 
$ \int_{\R^d} 
\|z \|_{ \R^d }^{\max\{p, 2\} (4\kappa+15)} \allowbreak
\, \nu_d (dz) $ $ \leq \eta d^{\eta}$,
and
\begin{equation}
| 
f_{ 0, d }(x) 
- 
( \mathcal{R} \phi^{ 0, d }_{ \varepsilon } )(x)
|
+
\| 
f_{ 1, d }(x) 
- 
( \mathcal{R} \phi^{ 1, d }_{ \varepsilon } )(x)
\|_{ \R^d }
\leq 
\varepsilon \kappa d^{ \kappa }
(
1 + \| x \|^{ \kappa }_{ \R^d }
)
.
\end{equation}
Then 
\begin{enumerate}[(i)]
\item 
\label{item:thm:existence_vis}
there exist unique 
at most polynomially growing functions 
$ u_d \colon [0,T] \times \R^{d} \to \R $,
$ d \in \N $, 
such that
for all $ d \in \N $, $ x \in \R^d $ it holds that
$ u_d( 0, x ) = f_{ 0, d }( x ) $
and such that for all $ d \in \N $ 
it holds that
$ u_d $ is a viscosity solution of
\begin{equation}
\begin{split}
( \tfrac{ \partial }{\partial t} u_d )( t, x ) 
& = 
( \tfrac{ \partial }{\partial x} u_d )( t, x )
\,
f_{ 1, d }( x )
+
\sum_{ i, j = 1 }^d
a_{ d, i, j }
\,
( \tfrac{ \partial^2 }{ \partial x_i \partial x_j } u_d )( t, x )
\end{split}
\end{equation}
for $ ( t, x ) \in (0,T) \times \R^d $
and 
\item
\label{item:thm:main_statement_Lp}
there exist
$
( 
\psi_{ d, \varepsilon } 
)_{ (d , \varepsilon)  \in \N \times (0,1] } \subseteq \mathcal{N}
$,
$
c \in \R
$
such that
for all 
$
d \in \N 
$,
$
\varepsilon \in (0,1] 
$
it holds that
$
\mathcal{P}( \psi_{ d, \varepsilon } ) 
\leq
c \, d^c \varepsilon^{ - c } 
$,
$
\mathcal{R}( \psi_{ d, \varepsilon } )
\in C( \R^{ d }, \R )
$,
and
\begin{equation}
\left[
\int_{ \R^d }
|
u_d(T,x) - ( \mathcal{R} \psi_{ d, \varepsilon } )( x )
|^p
\,
\nu_d(dx)
\right]^{ \nicefrac{ 1 }{ p } }
\leq
\varepsilon 
.
\end{equation}
\end{enumerate}
\end{theorem}
\begin{proof}[Proof of Theorem~\ref{thm:PDE_approx_Lp}]
Throughout this proof let $\mathbf{a} \colon \R \to \R$ be the function which satisfies for all $x \in \R$ that
\begin{equation}
\mathbf{a} ( x )= \max\{x, 0\}
\end{equation}
and let $(\phi^{2,d})_{d \in \N} \subseteq \mathcal{N}$ satisfy for all $d \in \N$, $x \in \R^d$ that  $\mathcal{R}(\psi) \in C(\R^d, \R^d)$, $\mathcal{L}(\phi^{2,d}) = (d, 2d, d)$, and $(\mathcal{R}\phi^{2,d})(x) = x$ (cf.~Lemma~\ref{lem:identity}). Observe that for all $d \in \N$, $x = (x_1, \ldots, x_d) \in \R^d$ it holds that
\begin{equation}
\label{eq:thm:A:d}
\mathbf{A}_d(x) = (\mathbf{a}(x_1), \ldots, \mathbf{a}(x_d)).
\end{equation}
Next note that for all $ d \in \N$ it holds that
\begin{equation}
\mathcal{P}(\phi^{2,d}) = 2d(d+1) + d(2d+1) = 2d^2 +2d +2d^2 +d = 4d^2 +3d \leq 7 d^2.
\end{equation}
This proves that for all $d \in \N$, $\varepsilon \in (0, 1]$ it holds that
\begin{equation}
\label{eq:thm:sum:par}
\begin{split}
\mathcal{P}(\phi^{2,d}) +
\sum_{ m = 0 }^1
\mathcal{P}( \phi^{ m, d }_{ \varepsilon } ) 
&\leq 7d^2 + \kappa d^{ \kappa } \varepsilon^{ - \kappa } \leq  (\kappa +7) d^{ \kappa +2 } \varepsilon^{ - \kappa}\\
&\leq (2\kappa +7) d^{ 2\kappa +7 } \varepsilon^{ - (2\kappa +7)}.
\end{split}
\end{equation}
Moreover, observe that Young's inequality assures that for all $\alpha \in [0, \infty)$ it holds that
\begin{equation}
\alpha^{\kappa} \leq \frac{\kappa }{2\kappa +7} \cdot \alpha^{2\kappa+7} + \frac{\kappa +7}{2\kappa+7} \leq \alpha^{2\kappa+7} + \frac{\kappa +7}{2\kappa+7}.
\end{equation}
This ensures that for all 
 $d \in \N$, $x, y \in \R^d$ it holds that
\begin{equation}
\kappa (1+ \|x\|_{\R^d}^{\kappa}) \leq \kappa(2+\|x\|_{\R^d}^{2\kappa+7}) \leq (2\kappa +7) (1+ \|x\|_{\R^d}^{2\kappa+7})
\end{equation}
and
\begin{equation}
\begin{split}
\kappa (1+ \|x\|_{\R^d}^{\kappa} + \|y\|_{\R^d}^{\kappa}) 
&\leq \kappa \left( 1 + \frac{2(\kappa+7)}{2\kappa + 7} + \|x\|_{\R^d}^{2\kappa+7} + \|y\|_{\R^d}^{2\kappa+7} \right)\\
& = \frac{\kappa(4\kappa +21)}{2\kappa +7} + \kappa (\|x\|_{\R^d}^{2\kappa+7} + \|y\|_{\R^d}^{2\kappa+7})\\
& \leq 2\kappa +7 + (2\kappa +7)(\|x\|_{\R^d}^{2\kappa+7} + \|y\|_{\R^d}^{2\kappa+7})\\
& = (2\kappa +7)(1 + \|x\|_{\R^d}^{2\kappa+7} + \|y\|_{\R^d}^{2\kappa+7}).
\end{split}
\end{equation}
Combining this with \eqref{eq:thm:A:d}, \eqref{eq:thm:sum:par}, the fact that $ \mathbf{A}_d \colon \R^d \to \R^d $, 
$ d \in \N $, 
and 
$ \mathbf{a} \colon \R \to \R$
are 
continuous functions, and  Corollary~\ref{cor:PDE_approx_Lp_gen} (with
$T =T$, 
$\kappa = 2\kappa+7$,
$\eta = \eta$,
$p = p$,
$(A_d)_{d \in \N} = (A_d)_{d \in \N}$,
$(\nu_d)_{d \in \N} = (\nu_d)_{d \in \N}$,
$(f_{0, d})_{d \in \N} = (f_{0,d})_{d \in \N}$,
$(f_{1,d})_{d \in \N} = (f_{1,d})_{d \in \N}$,
$(\mathbf{A}_d)_{d \in \N} = (\mathbf{A}_d)_{d \in \N}$,
$\mathbf{a} = \mathbf{a}$,
$\mathcal{N} = \mathcal{N}$,
$ \mathcal{P} = \mathcal{P}$,
$ \mathcal{L} = \mathcal{L}$,
$ \mathcal{R} = \mathcal{R}$,
$ (\phi^{ m, d }_{ \varepsilon })_{ 
	(m, d, \varepsilon) \in \{ 0, 1 \} \times \N \times (0,1] 
} $ $ = (\phi^{ m, d }_{ \varepsilon })_{ 
(m, d, \varepsilon) \in \{ 0, 1 \} \times \N \times (0,1] 
} $,
$ (\phi^{ 2, d })_{d \in \N}  = (\phi^{ 2, d })_{d \in \N} $
 in the notation of Corollary~\ref{cor:PDE_approx_Lp_gen})
establishes items~\eqref{item:thm:existence_vis}--\eqref{item:thm:main_statement_Lp}.
The proof of Theorem~\ref{thm:PDE_approx_Lp} is thus completed.
\end{proof}

\subsection{Rectified DNN approximations on the $ d $-dimensional unit cube}
\label{sec:lebesgue}

\begin{cor}
\label{cor:lebesgue}
Let 
$ T, \kappa,  p \in (0,\infty) $,  
let 
$
A_d = ( a_{ d, i, j } )_{ (i, j) \in \{ 1, \dots, d \}^2 } \in \R^{ d \times d }
$,
$ d \in \N $,
be symmetric positive semidefinite matrices, 
for every $ d \in \N $ 
let 
$
\left\| \cdot \right\|_{ \R^d } \colon \R^d \to [0,\infty)
$
be the $ d $-dimensional Euclidean norm,
let
$ f_{0,d} \colon \R^d \to \R $, $ d \in \N $,
and
$ f_{ 1, d } \colon \R^d \to \R^d $,
$ d \in \N $,
be  functions,
let 
$ \mathbf{A}_d \colon \R^d \to \R^d $, 
$ d \in \N $, 
be the functions 
which satisfy 
for all 
$
d \in \N
$,
$ x = ( x_1, \dots, x_d ) \in \R^d $
that
$ 
\mathbf{A}_d(x)
=
( \max\{x_1, 0\}, \ldots, \max\{x_d, 0\} )
$,
let 
\begin{equation}
\mathcal{N}
=
\cup_{ L \in \{ 2, 3, 4, \dots \} }
\cup_{ ( l_0, l_1, \ldots, l_L ) \in \N^{ L + 1 } }
(
\times_{ n = 1 }^L 
(
\R^{ l_n \times l_{ n - 1 } } \times \R^{ l_n } 
)
)
,
\end{equation}
let 
$
\mathcal{P}
\colon \mathcal{N} \to \N
$
and
$
\mathcal{R} \colon 
\mathcal{N} 
\to 
\cup_{ k, l \in \N } C( \R^k, \R^l )
$
be the functions which satisfy 
for all 
$ L \in \{ 2, 3, 4, \dots \} $, 
$ l_0, l_1, \ldots, l_L \in \N $, 
$ 
\Phi = ((W_1, B_1), \ldots, $ $ (W_L, B_L)) \in 
( \times_{ n = 1 }^L (\R^{ l_n \times l_{n-1} } \times \R^{ l_n } ) )
$,
$ x_0 \in \R^{l_0} $, 
$ \ldots $, 
$ x_{ L - 1 } \in \R^{ l_{ L - 1 } } $ 
with 
$ 
\forall \, n \in \N \cap [1,L) \colon 
x_n = \mathbf{A}_{ l_n }( W_n x_{ n - 1 } + B_n )
$
that 
$
\mathcal{P}( \Phi )
=
\textstyle
\sum\nolimits_{
	n = 1
}^L
l_n ( l_{ n - 1 } + 1 )
$,
$
\mathcal{R}(\Phi) \in C( \R^{ l_0 } , \R^{ l_L } )
$,
and
\begin{equation}
( \mathcal{R} \Phi )( x_0 ) = W_L x_{L-1} + B_L ,
\end{equation}
let 
$
( \phi^{ m, d }_{ \varepsilon } )_{ 
	(m, d, \varepsilon) \in \{ 0, 1 \} \times \N \times (0,1] 
} 
\subseteq \mathcal{N}
$
,
and 
assume for all
$ d \in \N $, 
$ \varepsilon \in (0,1] $, 
$ 
x, y \in \R^d
$
that
$
\mathcal{R}( \phi^{ 0, d }_{ \varepsilon } )
\in 
C( \R^d, \R )
$,
$
\mathcal{R}( \phi^{ 1, d }_{ \varepsilon } )
\in
C( \R^d, \R^d )
$,
$
|
f_{ 0, d }( x )
| 
+
\sum_{ i , j = 1 }^d
| a_{ d, i, j } |
\leq 
\kappa d^{ \kappa }
( 1 + \| x \|^{ \kappa }_{ \R^d } )
$,
$
\| 
f_{ 1, d }( x ) 
- 
f_{ 1, d }( y )
\|_{ \R^d }
\leq 
\kappa 
\| x - y \|_{ \R^d } 
$,
$
\|
( \mathcal{R} \phi^{ 1, d }_{ \varepsilon } )(x)    
\|_{ \R^d }	
\leq 
\kappa ( d^{ \kappa } + \| x \|_{ \R^d } )
$,
$ 
\sum_{ m = 0 }^1
\mathcal{P}( \phi^{ m, d }_{ \varepsilon } ) 
\leq \kappa d^{ \kappa } \varepsilon^{ - \kappa }
$, $ |( \mathcal{R} \phi^{ 0, d }_{ \varepsilon } )(x) - ( \mathcal{R} \phi^{ 0, d }_{ \varepsilon } )(y)| \leq \kappa d^{\kappa} (1 + \|x\|_{\R^d}^{\kappa} + \|y \|_{\R^d}^{\kappa})\|x-y\|_{\R^d}$, 
and
\begin{equation}
| 
f_{ 0, d }(x) 
- 
( \mathcal{R} \phi^{ 0, d }_{ \varepsilon } )(x)
|
+
\| 
f_{ 1, d }(x) 
- 
( \mathcal{R} \phi^{ 1, d }_{ \varepsilon } )(x)
\|_{ \R^d }
\leq 
\varepsilon \kappa d^{ \kappa }
(
1 + \| x \|^{ \kappa }_{ \R^d }
)
.
\end{equation}
Then 
\begin{enumerate}[(i)]
	\item 
	\label{item:01:viscosity}
	there exist unique 
	at most polynomially growing functions 
	$ u_d \colon [0,T] \times \R^{d} \to \R $,
	$ d \in \N $, 
	such that
	for all $ d \in \N $, $ x \in \R^d $ it holds that
	$ u_d( 0, x ) = f_{ 0, d }( x ) $
	and such that for all $ d \in \N $ 
	it holds that
	$ u_d $ is a viscosity solution of
	\begin{equation}
	\begin{split}
	( \tfrac{ \partial }{\partial t} u_d )( t, x ) 
	& = 
	( \tfrac{ \partial }{\partial x} u_d )( t, x )
	\,
	f_{ 1, d }( x )
	+
	\sum_{ i, j = 1 }^d
	a_{ d, i, j }
	\,
	( \tfrac{ \partial^2 }{ \partial x_i \partial x_j } u_d )( t, x )
	\end{split}
	\end{equation}
	for $ ( t, x ) \in (0,T) \times \R^d $
	and 
	\item
	\label{item:01:bound}
	there exist
	$
	( 
	\psi_{ d, \varepsilon } 
	)_{ (d , \varepsilon)  \in \N \times (0,1] } \subseteq \mathcal{N}
	$,
	$
	c \in \R
	$
	such that
	for all 
	$
	d \in \N 
	$,
	$
	\varepsilon \in (0,1] 
	$
	it holds that
	$
	\mathcal{P}( \psi_{ d, \varepsilon } ) 
	\leq
	c \, d^c \varepsilon^{ - c } 
	$,
	$
	\mathcal{R}( \psi_{ d, \varepsilon } )
	\in C( \R^{ d }, \R )
	$,
	and
	\begin{equation}
	\left[
	\int_{ [0,1]^d }
	|
	u_d(T,x) - ( \mathcal{R} \psi_{ d, \varepsilon } )( x )
	|^p
	\,
	dx
	\right]^{ \nicefrac{ 1 }{ p } }
	\leq
	\varepsilon 
	.
	\end{equation}
\end{enumerate}
\end{cor}
\begin{proof}[Proof of Corollary~\ref{cor:lebesgue}]
	Throughout this proof for every $d \in \N$ let $\lambda_{d} \colon \mathcal{B}(\R^d) \to [0, \infty]$ be the Lebesgue-Borel measure on $\R^d$ and let  $\nu_d \colon  \mathcal{B}(\R^d) \to [0,1]$ be the function which satisfies for all $B \in \mathcal{B}(\R^d)$ that
	\begin{equation}
	\label{eq:measure:def}
	\nu_d(B) = \lambda_{d}(B \cap [0, 1]^d).
	\end{equation}
Observe that \eqref{eq:measure:def} implies that for all $d \in \N$ it holds that $\nu_d$ is a probability measure on $\R^d$. This  and \eqref{eq:measure:def} ensure that for all $d \in \N$, $g \in C(\R^d, \R)$ it holds that
\begin{equation}
\label{eq:integral:equiv}
\int_{\R^d} |g(x)| \, \nu_d(dx) = \int_{ [0,1]^d } |g(x)| \, dx.
\end{equation}		
Combining this with,  e.g., Grohs et al.~\cite[Lemma~3.15]{GrohsWurstemberger2018} demonstrates that for all $d \in \N$ it holds that
\begin{equation}
\begin{split}
&\int_{\R^d} 
\|x \|_{ \R^d }^{\max\{p, 2\} (4\kappa+15)}
\, \nu_d (dx) 
= 
\int_{[0,1]^d} 
\|x \|_{ \R^d }^{\max\{p, 2\} (4\kappa+15)}
\, dx \\
&  \leq d^{\nicefrac{\max\{p, 2\} (4\kappa+15)}{2} } \leq d^{\max\{p, 2\}(2\kappa+8)} \leq \max\{p, 2\}(2\kappa+8) d^{\max\{p, 2\}(2\kappa+8)}.
\end{split}
\end{equation}
Theorem~\ref{thm:PDE_approx_Lp} (with
$T =T$, 
$\kappa = \kappa$,
$\eta = \max\{p, 2\}(2\kappa+8)$,
$p = p$,
$(A_d)_{d \in \N} = (A_d)_{d \in \N}$,
$(\nu_d)_{d \in \N} = (\nu_d)_{d \in \N}$,
$(f_{0, d})_{d \in \N} = (f_{0,d})_{d \in \N}$,
$(f_{1,d})_{d \in \N} = (f_{1,d})_{d \in \N}$,
$(\mathbf{A}_d)_{d \in \N} = (\mathbf{A}_d)_{d \in \N}$,
$\mathcal{N} = \mathcal{N}$,
$ \mathcal{P} = \mathcal{P}$,
$ \mathcal{R} = \mathcal{R}$,
$ (\phi^{ m, d }_{ \varepsilon })_{ 
	(m, d, \varepsilon) \in \{ 0, 1 \} \times \N \times (0,1] 
}  = (\phi^{ m, d }_{ \varepsilon })_{ 
	(m, d, \varepsilon) \in \{ 0, 1 \} \times \N \times (0,1] 
} $,
$ (\phi^{ 2, d })_{d \in \N}  = (\phi^{ 2, d })_{d \in \N} $
in the notation of Theorem~\ref{thm:PDE_approx_Lp}) and \eqref{eq:integral:equiv}  hence 
establish items~\eqref{item:01:viscosity}--\eqref{item:01:bound}.
The proof of Corollary~\ref{cor:lebesgue} is thus completed.
\end{proof}

\subsection*{Acknowledgements}
David Kofler is gratefully acknowledged for his useful 
comments regarding the a~priori estimates in Subsections~\ref{sec:a_priori_SDEs}--\ref{sec:a_priori_BMs}.
This work has been partially supported through the research grant 
with the title 
``Higher order numerical approximation methods 
for stochastic partial differential equations''
(Number 175699)
from the Swiss National Science Foundation (SNSF).
Furthermore,
this work has been partially supported through
the ETH Research Grant \mbox{ETH-47 15-2}
``Mild stochastic calculus and numerical approximations for nonlinear stochastic evolution equations with L\'evy noise''.

\bibliographystyle{acm}
\bibliography{../bibfileNN}

\end{document}